\documentclass[11pt,a4paper]{amsart}

\usepackage[margin=1.25in]{geometry}
\usepackage{amsthm}
\usepackage{amssymb}
\usepackage{mathtools}
\mathtoolsset{showonlyrefs}
\usepackage[english]{babel}
\usepackage[utf8]{inputenc}
\usepackage[T1]{fontenc}
\usepackage{newpxtext}
\usepackage{newpxmath}
\usepackage{microtype}

\usepackage{xcolor}
\usepackage{hyperref}
\usepackage[shortlabels]{enumitem}

\theoremstyle{plain}
\newtheorem{theorem}{Theorem}[section]
\newtheorem{lemma}[theorem]{Lemma}
\newtheorem{proposition}[theorem]{Proposition}
\newtheorem{corollary}[theorem]{Corollary}

\theoremstyle{definition}
\newtheorem{definition}[theorem]{Definition}
\newtheorem{example}[theorem]{Example}
\newtheorem{remark}[theorem]{Remark}

\newcommand{\kk}{\Bbbk}
\newcommand{\id}{\operatorname{Id}}
\newcommand{\cop}{\operatorname{cop}}
\newcommand{\Hom}{\operatorname{Hom}}
\newcommand{\aug}[1]{\mathfrak{a}_{#1}}
\newcommand{\Gc}[1]{G_{\mathrm{c}}(#1)}
\newcommand{\conj}[2]{{}^{#1}\!#2}

\makeatletter
\newcommand{\condlabel}[2]{\item[\textnormal{(#1)}]\phantomsection\def\@currentlabel{#1}\label{#2}}
\newcommand{\condref}[1]{\textup{(\ref{#1})}}
\@namedef{subjclassname@2020}{\textup{2020} Mathematics Subject Classification}
\makeatother

\numberwithin{equation}{section}
\hypersetup{
  colorlinks = true,
  linkcolor  = blue,
  urlcolor   = blue,
  citecolor  = black
}

\title{Quotients of graded quasi-Hopf algebras}

\author{Fabio Calder\'{o}n}
\address{Escuela de Matem\'{a}ticas, Universidad Industrial de Santander, Bucaramanga, Colombia}
\email{facalmat@uis.edu.co}

\author{C\'{e}sar Galindo}
\address{Departamento de Matem\'{a}ticas, Universidad de los Andes, Bogot\'{a}, Colombia}
\email{cn.galindo1116@uniandes.edu.co}

\subjclass[2020]{16T05, 16T15, 16W50, 16S35}
\keywords{quasi-Hopf algebra, strong grading, quasi-Hopf ideal, cocentral map, crossed product}

\begin{document}

\begin{abstract}
We classify the quasi-Hopf ideals of quasi-Hopf algebras faithfully graded by arbitrary groups. Each ideal determines a normal subgroup, recording which degrees are identified in the quotient, together with an invariant ideal in the corresponding neutral component. We refine this description using an ideal of the original neutral component and a retraction of a naturally associated intermediate quotient. We prove that this intermediate quotient splits as the tensor product of its neutral component with the group algebra of the normal subgroup. For fixed subgroup and neutral-component data, the possible ideals, whenever they exist, form a torsor described by equivariant homomorphisms into central group-like elements. Our results require neither finite dimensionality nor bijectivity of the antipode. We apply them to cocentral abelian extensions with possibly infinite grading, finite-dimensional neutral component, and nontrivial reassociator.
\end{abstract}

\maketitle


\section{Introduction}\label{sec:Intro}

Let $H=\bigoplus_{g\in G}H_g$ be a faithfully $G$-graded quasi-Hopf algebra over a field $\kk$ (Definition~\ref{def:Grad}). By Proposition~\ref{prop:FaithStrong}(iii), its grading is strong in the sense of Definition~\ref{def:SG}; we freely use this equivalence throughout. We classify the quasi-Hopf ideals of $H$, and hence the corresponding quotient quasi-Hopf algebras, for arbitrary grading groups without assuming that $H$ is finite-dimensional or that its antipode is bijective.

Although the graded ideals of a strongly graded algebra are determined by their intersections with the neutral component~\cite{NVO04}, this correspondence does not extend to arbitrary ideals. Even if the algebra carries a (quasi-)Hopf structure, the same failure occurs for quasi-Hopf ideals. We show, however, that the quasi-Hopf structure yields a refined correspondence once a normal subgroup of $G$ is added to the data recorded by the neutral component. This phenomenon is already visible in the group algebra $\kk[G]$, whose canonical strong $G$-grading is $\kk[G]=\bigoplus_{g\in G}\kk\zeta_g$, where $\zeta_g$ denotes the basis element corresponding to $g$. Each normal subgroup $N\unlhd G$ gives a Hopf ideal $\ker(\kk[G]\to\kk[G/N])$, and these ideals are pairwise distinct although their intersections with the neutral component $\kk\zeta_e$ are all trivial. Consequently, the neutral-component data do not detect $N$. The subgroup $N$ is therefore a genuinely new datum, recording which degrees are identified in the quotient. Moreover, once this subgroup is known, the ideal becomes graded with respect to the induced $G/N$-grading and is recovered from its intersection with the neutral component of that grading.

For a general $H$, the normal subgroup attached to a quasi-Hopf ideal is detected by the grading map $\pi\colon H\to\kk[G]$, defined by $\pi(x_g)=\varepsilon(x_g)\zeta_g$ for $x_g\in H_g$ (Theorem~\ref{thm:Cocentral}). If $I\subseteq H$ is a quasi-Hopf ideal, then $\pi(I)$ is a Hopf ideal of $\kk[G]$. Writing $\aug{N}$ for the augmentation ideal of $\kk[N]$, there is therefore a unique normal subgroup $N_I\unlhd G$ such that $\pi(I)=\aug{N_I}\cdot\kk[G]=\ker(\kk[G]\to\kk[G/N_I])$ (Proposition~\ref{prop:AssocSubgrp}). More generally, for $N\unlhd G$ and $\bar g=gN\in G/N$, the $G/N$-grading of $H$ induced by the canonical projection $G\to G/N$ has components $H_{\bar g}=\bigoplus_{n\in N}H_{gn}$ and neutral component $H_N=\bigoplus_{n\in N}H_n$. Taking $N=N_I$, the quotient $H/I$ inherits a $G/N$-grading and the quotient map is graded. Since $I$ is graded and the induced $G/N$-grading is strong, Proposition~\ref{prop:Intersect}(i) shows that $K:=I\cap H_N$ determines each homogeneous component of $I$:
\[
I\cap H_{\bar g}=H_{\bar g}K=KH_{\bar g}
\qquad(\bar g\in G/N).
\]
Thus $I=\sum_{\bar g\in G/N}H_{\bar g}K$, and this equality shows that $(N,K)$ determines $I$.

We call a pair $(N,K)$ \emph{admissible} if $N\unlhd G$, $K$ is a $G/N$-invariant quasi-Hopf ideal of $H_N$, and $\aug{N}\subseteq\pi(K)$ (Definition~\ref{def:AdmPair}). Our first main result, Theorem~\ref{thm:MainA}, states that the assignments
\[
I\longmapsto (N_I,\,I\cap H_{N_I})
\qquad\text{and}\qquad
(N,K)\longmapsto K^{G/N}:=\sum_{\bar g\in G/N}H_{\bar g}K
\]
define a bijective correspondence between the set of quasi-Hopf ideals of $H$ and the set of admissible pairs for $H$. The case $N=\{e\}$ recovers the $G$-graded quasi-Hopf ideals (Corollary~\ref{cor:GradIdeals}), whereas $N=G$ corresponds to the ideals $I$ satisfying $\pi(I)=\aug{G}$, called augmented ideals below (Definition~\ref{def:AugIdeal}).

Given an admissible pair $(N,K)$, set $J:=K\cap H_e$. Then $J$ is a $G$-invariant quasi-Hopf ideal of $H_e$, and the ideal $J^N:=H_NJ=JH_N$ generated by $J$ in $H_N$ is contained in $K$. The quotient $H_N/J^N$, called the lifted quotient, is strongly $N$-graded, with neutral component $H_e/J$ and grading map $\pi_{N,J}\colon H_N/J^N\to\kk[N]$ (Corollary~\ref{cor:RelLift}(ii) and Lemma~\ref{lem:RedSetup}(ii)). Thus $L:=K/J^N$ is a quasi-Hopf ideal of the lifted quotient. The ideal $J$ records the relations visible in $H_e$, while $L$ records the remaining relations. Corollary~\ref{cor:RedClass} shows that this construction is reversible and parametrizes the quasi-Hopf ideals of $H$ by triples $(N,J,L)$, called reduced data, for which $L$ is $G/N$-invariant, $\pi_{N,J}(L)=\aug{N}$, and $L\cap(H_e/J)=0$.

For every reduced datum, the natural map $H_e/J\to(H_N/J^N)/L$ is injective. We call $L$ \emph{transversal} to $H_e/J$ if this map is an isomorphism, and the associated quasi-Hopf ideal of $H$ is then called \emph{split}. By Theorem~\ref{thm:AugTransv} and Corollary~\ref{cor:TransvEquiv}, a quasi-Hopf ideal of the lifted quotient is transversal precisely when it is augmented and intersects the neutral component trivially. Both conditions hold for reduced data; therefore every quasi-Hopf ideal of $H$ is split (Corollary~\ref{cor:SplitAll}).

A transversal ideal $L$ determines the unique quasi-Hopf algebra retraction $\tau\colon H_N/J^N\to H_e/J$ satisfying $\tau|_{H_e/J}=\id$ and $\ker\tau=L$ (Proposition~\ref{prop:TransvRetr}). Our second main result, Theorem~\ref{thm:MainB}, therefore classifies all quasi-Hopf ideals of $H$ by triples $(N,J,\tau)$ for which $\ker\tau$ is $G/N$-invariant.

The retraction has a strong structural consequence. Theorem~\ref{thm:Trivialization} shows that a retraction of a faithfully $\Gamma$-graded quasi-Hopf algebra onto its neutral component is equivalent to a multiplicative family of group-like homogeneous units that centralize the neutral component. Under these conditions, the algebra is the untwisted tensor product of its neutral component with $\kk[\Gamma]$. Consequently, the lifted quotient associated with every quasi-Hopf ideal is isomorphic to $(H_e/J)\otimes\kk[N]$ as an $N$-graded quasi-Hopf algebra (Corollary~\ref{cor:LiftTensor}).

This splitting also gives an intrinsic description of the ideals lying over a fixed pair $(N,J)$. Proposition~\ref{prop:Realized} shows that such a pair is realized by a quasi-Hopf ideal precisely when its lifted quotient admits a $G$-conjugation-equivariant group-like splitting. Once one such splitting is fixed, Theorem~\ref{thm:Char} identifies the ideals lying over $(N,J)$ with the $G$-equivariant group homomorphisms
\[
N\longrightarrow\Gc{H_e/J},
\]
where $\Gc{H_e/J}$ denotes the group of invertible central group-like elements of $H_e/J$. Thus, whenever this fibre is nonempty, it is a torsor under this group of homomorphisms. The retraction is intrinsic, whereas the homomorphism records it relative to the chosen splitting.

For applications, we also express these retractions in terms of a chosen crossed-product presentation. Under suitable hypotheses on the associated weak action and factor set, Proposition~\ref{prop:CohObstr} gives a class $[\sigma]\in H^2(N,\Gc{H_e/J})$ that measures the existence of algebra retractions with central group-like values relative to the chosen homogeneous units. Because this class depends on the presentation, it is not an intrinsic obstruction; the intrinsic obstruction is that the fibre over $(N,J)$ may be empty.

The examples further clarify the structure of the classification. Example~\ref{ex:SL2} gives a strongly graded Hopf algebra that is not a crossed product, whereas Example~\ref{ex:CrossedZSq} shows that, even for a crossed product, a nontrivial factor-set obstruction may make the fibre over a pair $(N,J)$ empty. Finally, Example~\ref{ex:SplitNotChar} shows that this cohomological obstruction depends on the chosen homogeneous section and is therefore not intrinsic to the ideal.

Beyond these structural examples, our main concrete application concerns the cocentral Masuoka-type extensions $(\kk F)^*\#_\sigma^\theta\kk[G]$, where $F$ is finite and $G$ is arbitrary. Masuoka's construction permits infinite grading groups~\cite[Rem.~2.14]{Mas02}; our results classify their quasi-Hopf ideals without restricting the associated normal subgroup of $G$ and allow a nontrivial reassociator.

A principal motivation for the present work was the parametrization of equivariant fusion subcategories by triples $(\mathcal E,N,\eta)$ established in previous work of the second author with Corey Jones~\cite[Thm.~4.8]{GJ25}. The triples $(N,J,\tau)$ introduced here may be viewed as algebraic counterparts of those categorical data, with $J$ encoding the neutral datum and $\tau$ playing the role of the equivariant trivialization. In the finite-dimensional semisimple setting, the two parametrizations are expected to agree under the appropriate duality. The present results are obtained by independent, purely algebraic methods and apply to arbitrary grading groups without requiring the algebras to be finite-dimensional or semisimple.

The paper is organized as follows. Section~\ref{sec:Prelim} fixes notation and reviews quasi-Hopf algebras and their gradings. Section~\ref{sec:Grad} establishes the correspondence between faithful $G$-gradings and $G$-grading morphisms to $\kk[G]$. It also studies induced $G/N$-gradings and their descent through quasi-Hopf quotients. Section~\ref{sec:Class} classifies quasi-Hopf ideals by admissible pairs. Section~\ref{sec:Explicit} introduces reduced data and transversal ideals, together with the associated retractions. It proves that every quasi-Hopf ideal is split and gives the resulting classification by retractions. Section~\ref{sec:CharForm} characterizes which pairs $(N,J)$ occur. For each such pair, it parametrizes the corresponding ideals by $G$-equivariant group homomorphisms. It then expresses this parametrization in a chosen crossed-product presentation. Section~\ref{sec:Apps} presents the applications and examples.


\section{Preliminaries and notation}\label{sec:Prelim}

In this section we recall the definitions and basic facts on quasi-Hopf algebras and on group-graded algebras that are used throughout the paper. For quasi-Hopf algebras we refer the reader to Drinfeld's original paper~\cite{D89} and to~\cite{BCPV19}, and for group-graded algebras to~\cite{NVO04}.

Fix a base field $\kk$ and a (not necessarily finite) group $G$, written multiplicatively with identity $e$. All structures and tensor products are over $\kk$. Algebras are unital, and algebra morphisms and anti-morphisms preserve units. All comultiplications are counital. We write $\{\zeta_g\}_{g\in G}$ for the standard basis of the group algebra $\kk[G]$, with $\zeta_g\zeta_h=\zeta_{gh}$.

\subsection{Quasi-Hopf algebras}\label{subsec:QH}

A \emph{quasi-bialgebra} is an algebra $H$ equipped with algebra morphisms $\Delta\colon H\to H\otimes H$ and $\varepsilon\colon H\to\kk$, and an invertible reassociator $\Phi\in H^{\otimes3}$, satisfying the following identities:
\begin{gather}
(1 \otimes \Phi)\,(\id \otimes \Delta \otimes \id)(\Phi)\,(\Phi \otimes 1)
=
(\id \otimes \id \otimes \Delta)(\Phi)\,(\Delta \otimes \id \otimes \id)(\Phi),
\label{eq:QH1}
\\
\Phi\,(\Delta \otimes \id)\left(\Delta(x)\right)
=
(\id \otimes \Delta)\left(\Delta(x)\right)\,\Phi,
\qquad \text{for all } x \in H,
\label{eq:QH2}
\\
(\id \otimes \varepsilon \otimes \id)(\Phi) = 1 \otimes 1,
\label{eq:QH3}
\\
(\id \otimes \varepsilon)\left(\Delta(x)\right) = x = (\varepsilon \otimes \id)\left(\Delta(x)\right),
\qquad \text{for all } x \in H.
\label{eq:QH4}
\end{gather}
Identity~\eqref{eq:QH1} is the \emph{pentagon axiom}, while~\eqref{eq:QH2} expresses the \emph{quasi-coassociativity} of $\Delta$. Identities~\eqref{eq:QH3}--\eqref{eq:QH4} are the \emph{counit axioms}. From these axioms one further deduces that
\begin{equation}\label{eq:QH5}
    (\varepsilon \otimes \id \otimes \id)(\Phi)=1\otimes 1 = (\id \otimes \id \otimes \varepsilon)(\Phi).
\end{equation}

We adopt \emph{primed Sweedler notation} for the comultiplication, writing $\Delta(x)=x'\otimes x''$ for all $x\in H$, with the understanding that $\Delta$ need not be coassociative. The reassociator $\Phi$ and its inverse $\Phi^{-1}$ are expressed in components as $\Phi=\Phi^{(1)}\otimes\Phi^{(2)}\otimes\Phi^{(3)}$ and $\Phi^{-1}=\Phi^{(-1)}\otimes\Phi^{(-2)}\otimes\Phi^{(-3)}$. As usual, a finite summation is implicitly understood and the summation sign is omitted; the same convention applies to all element-wise tensor decompositions below. Within a single expression, distinct occurrences of $\Phi$ or $\Phi^{-1}$ refer to distinct copies, each carrying its own suppressed summation index.

These axioms guarantee that the category of left $H$-modules carries a monoidal structure. The comultiplication determines the tensor product, and the base field $\kk$, regarded as a left $H$-module via the counit, is the monoidal unit. The action of $\Phi$ gives the associativity constraint.

A \emph{morphism of quasi-bialgebras} $f\colon H\to H'$ is an algebra morphism such that $(f\otimes f)\,\Delta_H=\Delta_{H'}\,f$, $\varepsilon_{H'}\,f=\varepsilon_H$, and $f^{\otimes 3}(\Phi_H)=\Phi_{H'}$.

\begin{definition}[Quasi-Hopf algebra]\label{def:QH}
A quasi-bialgebra $H$ is called a \emph{quasi-Hopf algebra} if it is equipped with an algebra anti-morphism $S \colon H \to H$, referred to as the \emph{antipode}, together with distinguished elements $\alpha, \beta \in H$ such that, for every $x \in H$, the following identities are satisfied:
\begin{gather}
S(x')\, \alpha\, x'' = \varepsilon(x)\,\alpha,
\qquad
x'\, \beta\, S(x'') = \varepsilon(x)\,\beta,
\label{eq:QH6}
\\
\Phi^{(1)}\, \beta\, S(\Phi^{(2)})\, \alpha\, \Phi^{(3)} = 1 = S(\Phi^{(-1)})\, \alpha\, \Phi^{(-2)}\, \beta\, S(\Phi^{(-3)}).
\label{eq:QH7}
\end{gather}

A \emph{morphism of quasi-Hopf algebras} $f\colon H\to H'$ is a quasi-bialgebra morphism satisfying $f\,S_H=S_{H'}\,f$, $f(\alpha_H)=\alpha_{H'}$, and $f(\beta_H)=\beta_{H'}$.
\end{definition}

A subalgebra $A \subseteq H$ is called a \emph{quasi-Hopf subalgebra} if $\Delta(A) \subseteq A \otimes A$, $S(A) \subseteq A$, $\alpha, \beta \in A$, and both $\Phi$ and $\Phi^{-1}$ lie in $A^{\otimes 3}$. With the restricted structure maps, a quasi-Hopf subalgebra is itself a quasi-Hopf algebra, and the inclusion $A\hookrightarrow H$ is a morphism of quasi-Hopf algebras.

Any bialgebra (respectively, Hopf algebra) can be viewed as a quasi-bialgebra (respectively, quasi-Hopf algebra) with trivial reassociator $\Phi = 1 \otimes 1 \otimes 1$ and, in the Hopf case, $\alpha = \beta = 1$. Note that Definition~\ref{def:QH} does not assume that the antipode is bijective, unlike the original definition in~\cite{D89}; in this we follow~\cite{BCPV19}. When $H$ is finite-dimensional or quasitriangular, the axioms themselves guarantee that the antipode is bijective (see~\cite{BC03,BN03}).

\begin{remark}\label{rem:Rescaling}
Applying $\varepsilon$ to~\eqref{eq:QH7} and using~\eqref{eq:QH3} and~\eqref{eq:QH5} yields $\varepsilon(\alpha)\,\varepsilon(\beta)=1$; in particular $\varepsilon(\alpha)\neq 0$, hence applying $\varepsilon$ to~\eqref{eq:QH6} gives $\varepsilon\circ S=\varepsilon$. Since replacing $(\alpha,\beta)$ by $(\lambda\alpha,\lambda^{-1}\beta)$ for $\lambda\in\kk^{\times}$ preserves~\eqref{eq:QH6}--\eqref{eq:QH7}, we may and do assume throughout the paper that $\varepsilon(\alpha) = \varepsilon(\beta) = 1$.
\end{remark}

In contrast to the Hopf algebra setting, the antipode of a quasi-Hopf algebra need not be a coalgebra anti-morphism. Nonetheless, there is a distinguished invertible element $\mathfrak{f}=\mathfrak{f}^{(1)}\otimes\mathfrak{f}^{(2)}\in H\otimes H$, called the \emph{Drinfeld twist}, which satisfies the normalization conditions $\varepsilon(\mathfrak{f}^{(1)})\,\mathfrak{f}^{(2)}=1=\varepsilon(\mathfrak{f}^{(2)})\,\mathfrak{f}^{(1)}$, thanks to the normalization $\varepsilon(\alpha)=\varepsilon(\beta)=1$ of Remark~\ref{rem:Rescaling}, and is such that, writing $\Delta^{\cop}(x):=x''\otimes x'$ for the opposite comultiplication,
\begin{equation}\label{eq:QH8}
\mathfrak{f}\,\Delta(S(x))\,\mathfrak{f}^{-1}
=
(S\otimes S)\left(\Delta^{\cop}(x)\right),
\qquad \text{for all } x\in H.
\end{equation}
Both $\mathfrak f$ and $\mathfrak f^{-1}$ are given by explicit formulas in terms of $\Phi^{\pm1}$, $\alpha$, $\beta$, $\Delta$, and $S$; see~\cite[(3.2.15)--(3.2.16)]{BCPV19}. We always mean by $\mathfrak f$ the element defined in this way.

\begin{definition}[Quasi-Hopf ideal]\label{def:QHIdeal}
A two-sided ideal $I$ of a quasi-Hopf algebra $H$ is called a \emph{quasi-Hopf ideal} if $\Delta(I)\subseteq I\otimes H+H\otimes I$, $\varepsilon(I)=0$, and $S(I)\subseteq I$.
\end{definition}

If $I$ is a quasi-Hopf ideal in $H$, then the quotient algebra $H/I$ admits a unique quasi-Hopf algebra structure for which the canonical projection $q \colon H \rightarrow H/I$ is a morphism of quasi-Hopf algebras. Concretely, the quasi-Hopf structure on $H/I$ is determined by the relations $\overline{\Delta}\,q=(q\otimes q)\,\Delta$, $\overline{\varepsilon}\,q=\varepsilon$, $\overline{S}\,q=q\,S$, $\overline{\Phi}=q^{\otimes 3}(\Phi)$, $\overline{\alpha}=q(\alpha)$, and $\overline{\beta}=q(\beta)$. The Drinfeld twist descends to $q^{\otimes2}(\mathfrak f)$ by its explicit formula. Moreover, the quotient has the usual universal property: every quasi-Hopf morphism $f\colon H\to H'$ with $I\subseteq\ker f$ factors uniquely through $q$. Conversely, kernels of quasi-Hopf morphisms are quasi-Hopf ideals, since $\ker(f\otimes f)=\ker f\otimes H+H\otimes\ker f$ and $f$ intertwines the counits and the antipodes. When $H$ is finite-dimensional, every quotient quasi-bialgebra of $H$ is automatically quasi-Hopf, of dimension dividing that of $H$~\cite{Sch05}. When $I$ is fixed, we write $\overline{x}=q(x)$.

\subsection{Graded quasi-Hopf algebras}\label{subsec:GradQH}

A \emph{$G$-graded algebra} is a $\kk$-algebra $A$ equipped with a decomposition into a direct sum of $\kk$-vector spaces $A=\bigoplus_{g\in G}A_g$ such that $A_gA_h\subseteq A_{gh}$ for all $g,h\in G$. The subspaces $A_g$ are called the \emph{homogeneous components} of $A$. In particular $1\in A_e$ and the neutral component $A_e$ is a subalgebra of $A$. The grading is said to be \emph{faithful} if $A_g\neq 0$ for every $g\in G$, and \emph{strong} if $A_gA_{g^{-1}}=A_e$ for every $g\in G$, which is equivalent to requiring $A_gA_h=A_{gh}$ for all $g,h\in G$. In particular, a nonzero strongly graded algebra is faithfully graded, since $1\in A_gA_{g^{-1}}$ implies $A_g\neq0$ for every $g\in G$. Writing $A_g^\times:=A_g\cap A^\times$ for the set of invertible elements of degree $g$, the inverse of an element of $A_g^\times$ lies in $A_{g^{-1}}$, as one sees by comparing homogeneous components in $uu^{-1}=1=u^{-1}u$. We say that $A$ is a \emph{$G$-crossed product} if $A_g^\times\neq\emptyset$ for every $g\in G$. In that case $A$ is strongly graded, since $A_e=A_eu_gu_g^{-1}\subseteq A_gA_{g^{-1}}\subseteq A_e$ for any $u_g\in A_g^\times$. Moreover, fixing a homogeneous unit $u_g\in A_g^\times$ for each $g\in G$, one has $A_g=A_eu_g$; hence every element of $A$ is uniquely a finite sum $\sum_{g\in G}b_gu_g$ with $b_g\in A_e$.

If $A=\bigoplus_{g\in G}A_g$ and $A'=\bigoplus_{g\in G}A'_g$ are two $G$-graded algebras, a morphism $f\colon A\to A'$ of algebras is called \emph{$G$-graded} if $f(A_g)\subseteq A'_g$ for all $g\in G$. An ideal $I$ of a $G$-graded algebra $A$ is called \emph{$G$-graded}, or simply \emph{graded}, if $I=\bigoplus_{g\in G}(I\cap A_g)$; equivalently, the homogeneous components of every element of $I$ lie in $I$.

Throughout, the decomposition of a $G$-graded algebra $A$ is implicitly denoted $A=\bigoplus_{g\in G} A_g$, and, for $g \in G$, the notation $x_g$ always denotes a homogeneous element of \emph{degree} $g$, that is, $x_g \in A_g$.

\begin{definition}[Graded quasi-Hopf algebra]\label{def:Grad}
A quasi-Hopf algebra $H$ is called \emph{$G$-graded} if its underlying algebra is $G$-graded, with decomposition $H=\bigoplus_{g\in G} H_g$, and in addition:
\begin{gather}
    \Delta(H_g)\subseteq H_g\otimes H_g,
    \qquad \text{for all } g\in G, \label{eq:Grad1}\\
    \Phi\in H_e^{\otimes 3},
    \qquad
    \alpha,\beta\in H_e,
    \qquad
    S(H_g)\subseteq H_{g^{-1}}, \qquad \text{for all } g\in G.\label{eq:Grad2}
\end{gather}

If $H=\bigoplus_{g\in G}H_g$ and $H'=\bigoplus_{g\in G}H'_g$ are two $G$-graded quasi-Hopf algebras, then a morphism of quasi-Hopf algebras $f\colon H\to H'$ is called \emph{$G$-graded} if it is $G$-graded as a morphism of algebras.
\end{definition}

The grading axioms have the following consequences; compare~\cite[\S\S3--4]{Ce02} for the Hopf case.

\begin{proposition}\label{prop:GradProp}
Let $H$ be a $G$-graded quasi-Hopf algebra. Then:
\begin{enumerate}
    \item For every $g\in G$ with $H_g\neq0$ there exists an element $t_g\in H_g$ such that $\varepsilon(t_g)=1$.
    \item The neutral component $H_e$ is a quasi-Hopf subalgebra of $H$.
\end{enumerate}
\end{proposition}

\begin{proof}
(i) Choose a nonzero $x\in H_g$. If $\varepsilon$ vanished on $H_g$, then~\eqref{eq:Grad1} would give $(\varepsilon\otimes \id)\Delta(x)=0$, whereas the counit axiom~\eqref{eq:QH4} gives $(\varepsilon\otimes \id)\Delta(x)=x\neq 0$. Hence $\varepsilon(y)\neq 0$ for some $y\in H_g$, and $t_g:=\varepsilon(y)^{-1}y$ has the required property.

(ii) Since $H$ is $G$-graded as an algebra, $H_e$ is a subalgebra containing $1$, and~\eqref{eq:Grad1}--\eqref{eq:Grad2} give $\Delta(H_e)\subseteq H_e\otimes H_e$, $S(H_e)\subseteq H_e$, $\alpha,\beta\in H_e$, and $\Phi\in H_e^{\otimes3}$, hence also $\Phi^{-1}\in H_e^{\otimes3}$. Thus $H_e$ is a quasi-Hopf subalgebra.
\end{proof}

Conventions in the literature differ, even in the Hopf case: some authors require instead that $H$ be a graded coalgebra, that is, $\Delta(H_g)\subseteq\bigoplus_{hk=g}H_h\otimes H_k$ (see, e.g.,~\cite[\S10.5]{Mon93}). The two requirements are not comparable. The block $H_g\otimes H_g$ occurs in the second decomposition only for $g=e$; hence imposing both forces $H_g=0$ for every $g\neq e$. The same remark is made in~\cite[p.~359]{Ce02}. We adopt the diagonal condition~\eqref{eq:Grad1}, which is what makes the grading map of the next section land in $\kk[G]$ rather than in a larger coalgebra.

\begin{definition}[Strongly graded]\label{def:SG}
A $G$-graded quasi-Hopf algebra $H$ is called \emph{strongly graded} (respectively, a \emph{$G$-crossed product}) if its underlying $G$-graded algebra is strongly graded (respectively, a $G$-crossed product).
\end{definition}

The group algebra $\kk[G]$ with basis $\{\zeta_g\}_{g\in G}$ is a Hopf algebra with $\Delta(\zeta_g)=\zeta_g\otimes\zeta_g$, $\varepsilon(\zeta_g)=1$, and $S(\zeta_g)=\zeta_{g^{-1}}$. The decomposition $(\kk[G])_g:=\kk\,\zeta_g$ yields its canonical $G$-grading, under which $\kk[G]$ is a $G$-crossed product.

If $H$ is a quasi-Hopf algebra, then $H\otimes\kk[G]$ is a quasi-Hopf algebra with the componentwise product and
\[
\Delta(x\otimes\zeta_g)=\left(x'\otimes\zeta_g\right)\otimes\left(x''\otimes\zeta_g\right),
\qquad
\varepsilon(x\otimes\zeta_g)=\varepsilon(x),
\qquad
S(x\otimes\zeta_g)=S(x)\otimes\zeta_{g^{-1}},
\]
its reassociator and distinguished elements being the images of $\Phi$, $\alpha$ and $\beta$ under $x\mapsto x\otimes\zeta_e$. It is $G$-graded as a quasi-Hopf algebra by $\left(H\otimes\kk[G]\right)_g=H\otimes\kk\zeta_g$.

The next proposition shows that every faithful quasi-Hopf grading in our
sense is strong. Both the quasi-antipode axioms and the diagonal coproduct
condition are essential. Indeed, let $A=\kk[x]$ carry the
$\mathbb{Z}/2\mathbb{Z}$-grading
\[
A_0=\kk[x^2],
\qquad
A_1=x\kk[x^2].
\]
This grading is faithful but not strong, since
$A_1^2=x^2\kk[x^2]\neq A_0$. Declaring $x$ group-like makes $A$ a
bialgebra satisfying~\eqref{eq:Grad1}, but not a Hopf algebra. Declaring
$x$ primitive instead makes $A$ a Hopf algebra satisfying the
graded-coalgebra convention above, but not~\eqref{eq:Grad1}.

\begin{proposition}\label{prop:FaithStrong}
Let $H$ be a $G$-graded quasi-Hopf algebra. Then:
\begin{enumerate}
    \item $H_{g^{-1}}H_g=H_e=H_gH_{g^{-1}}$ for every $g\in G$ with $H_g\neq0$.

    \item The set of $g\in G$ with $H_g\neq0$ is a subgroup of $G$, and $H$ is strongly graded by it.

    \item If the grading is faithful, then it is strong.
\end{enumerate}
\end{proposition}

\begin{proof}
(i) Let $g\in G$ with $H_g\neq0$. By Proposition~\ref{prop:GradProp}(i) there is $t_g\in H_g$ with $\varepsilon(t_g)=1$. Set $T:=H_{g^{-1}}H_g$, a two-sided ideal of $H_e$ because $H$ is $G$-graded as an algebra. By~\eqref{eq:Grad1} we have $\Delta(t_g)\in H_g\otimes H_g$, and by~\eqref{eq:Grad2} we have $S(H_g)\subseteq H_{g^{-1}}$ and $\alpha\in H_e$. The first identity of~\eqref{eq:QH6} therefore gives
\[
\alpha=\varepsilon(t_g)\,\alpha=S(t_g')\,\alpha\,t_g''\in S(H_g)\,\alpha\,H_g\subseteq H_{g^{-1}}H_eH_g\subseteq T.
\]
Since $\Phi\in H_e^{\otimes3}$, $\beta\in H_e$ and $S(H_e)\subseteq H_e$ by~\eqref{eq:Grad2}, the first identity of~\eqref{eq:QH7} yields $1=\Phi^{(1)}\beta S(\Phi^{(2)})\alpha\Phi^{(3)}\in H_e\,\alpha\,H_e\subseteq T$. Hence $H_e=H_e\cdot1\subseteq T\subseteq H_e$, that is, $H_{g^{-1}}H_g=H_e$. Moreover $\varepsilon(S(t_g))=\varepsilon(t_g)=1$ by Remark~\ref{rem:Rescaling} and $S(t_g)\in H_{g^{-1}}$; hence $H_{g^{-1}}\neq0$, and the same argument applied to $g^{-1}$ gives $H_gH_{g^{-1}}=H_e$.

(ii) The neutral component is nonzero, since $1\in H_e$. If $H_g\neq0$, then $H_{g^{-1}}\neq0$ by the previous paragraph. If $H_g$ and $H_h$ are nonzero, then $t_gt_h\in H_{gh}$ has counit $1$, and therefore $H_{gh}\neq0$. Thus the support is a subgroup, over which the grading is strong by (i).

(iii) Immediate from (i).
\end{proof}

Consequently, for graded quasi-Hopf algebras, faithful and strong gradings are equivalent. We state subsequent hypotheses in terms of faithful gradings and freely use the equivalent strong-grading identities in the proofs.

For Hopf algebras, Proposition~\ref{prop:FaithStrong}(iii) also follows
from a standard Hopf--Galois argument. A surjective Hopf algebra morphism
$f\colon H\to\kk[G]$ makes $H$ a $\kk[G]$-comodule algebra with
coinvariants $H_e$. The associated canonical map
\[
H\otimes_{H_e}H\longrightarrow H\otimes\kk[G],
\qquad
x\otimes y\longmapsto xy'\otimes f(y''),
\]
is surjective~\cite[Ex.~3.6]{Sch90}, and for a $\kk[G]$-coaction the
canonical map is surjective precisely when the grading is strong
~\cite[Thm.~8.1.7]{Mon93}. This argument does not extend directly to
quasi-Hopf algebras with nontrivial reassociator, because it relies on
coassociativity of the coproduct and on the Hopf identity
$S(x')x''=\varepsilon(x)1$.

In finite dimension, a graded quasi-Hopf algebra is faithfully graded if and only if it is a crossed product.

\begin{proposition}\label{prop:FinDimCrossed}
Let $H$ be a finite-dimensional faithfully $G$-graded quasi-Hopf algebra. Then $G$ is finite and $H$ is a $G$-crossed product.
\end{proposition}

\begin{proof}
By Proposition~\ref{prop:FaithStrong}(iii), the grading is strong. Hence $1\in H_gH_{g^{-1}}$ gives $H_g\neq0$ for every $g\in G$. As the sum $H=\bigoplus_{g\in G}H_g$ is direct and $H$ is finite-dimensional, $|G|\leq\dim_\kk H$, and hence $G$ is finite. The neutral component $H_e$ is a quasi-Hopf subalgebra by Proposition~\ref{prop:GradProp}(ii); Schauenburg's theorem~\cite[Thm.~4.2]{Sch04} therefore shows that $H$ is free as a right $H_e$-module, say $H\cong H_e^{\oplus r}$, necessarily with $r=\dim H/\dim H_e\geq1$. The antipode convention of~\cite{Sch04} requires the antipodes to be bijective; they are, since $H$ and $H_e$ are finite-dimensional. Since $H_gH_e\subseteq H_g$, the grading is a decomposition of right $H_e$-modules. In particular each component $H_g$, being a direct summand of $H$ as a right $H_e$-module, is finitely generated projective.

Since the grading is strong, multiplication induces isomorphisms of $H_e$-bimodules $H_g\otimes_{H_e}H_h\to H_{gh}$, and each $H_g$ is an invertible $H_e$-bimodule~\cite[Cor.~3.1.2]{NVO04}. Hence $-\otimes_{H_e}H_g$ is an equivalence of the category of right $H_e$-modules and permutes the isomorphism classes of finitely generated indecomposable projective modules. By the Krull--Schmidt theorem, write $H_e=\bigoplus_i P_i^{\oplus m_i}$ as a right $H_e$-module, with the $P_i$ pairwise non-isomorphic and indecomposable. Every finitely generated indecomposable projective right $H_e$-module is isomorphic to some $P_i$, and we write $s_g$ for the permutation of the indices induced by $-\otimes_{H_e}H_g$, with $s_hs_g=s_{gh}$. The multiplicity of $P_j$ in $H_g\cong H_e\otimes_{H_e}H_g$ is $m_{s_g^{-1}(j)}$, and comparing multiplicities in $H=\bigoplus_{g\in G}H_g\cong H_e^{\oplus r}$ gives
\[
\sum_{g\in G}m_{s_g^{-1}(j)}=r\,m_j,
\qquad \text{for all } j.
\]
Since $s_g^{-1}=s_{g^{-1}}$, the left-hand side equals $\sum_{g\in G}m_{s_g(j)}$. Replacing $j$ by $s_h(j)$ and using $s_gs_h=s_{hg}$ reindexes this sum by $g\mapsto hg$, and therefore leaves it unchanged. As $r\geq1$, the function $j\mapsto m_j$ is therefore fixed by every $s_h$. Hence the multiplicities of $H_g$ and $H_e$ coincide, and $H_g\cong H_e$ as right $H_e$-modules for every $g\in G$.

Fix $g\in G$ and an isomorphism of right $H_e$-modules $H_e\to H_g$. Its value at $1$ is an element $w\in H_g$ such that $b\mapsto wb$ is a bijection from $H_e$ onto $H_g$. Since $H_gH_{g^{-1}}=H_e$ and $H_g=wH_e$, there exists $v\in H_{g^{-1}}$ with $wv=1$. Then $vw-1\in H_e$ and $w\left(vw-1\right)=\left(wv\right)w-w=0$; hence $vw=1$ because the map $b\mapsto wb$ is injective. Thus $w$ is a homogeneous unit of degree $g$, and $H$ is a $G$-crossed product.
\end{proof}

The quasi-Hopf hypothesis is essential. The algebra $M_3(\kk)$ carries a strong $\mathbb{Z}/2\mathbb{Z}$-grading that is not a crossed product~\cite[Ex.~1.3.5]{NVO04}; consequently, no purely ring-theoretic argument can establish the conclusion. For ordinary Hopf algebras the statement is due to~\cite[Thm.~2.2]{Sch92}, see also~\cite[Thm.~8.4.6]{Mon93}. We have not found the quasi-Hopf case in the literature. The finite-dimensional hypothesis is also essential: Subsection~\ref{subsec:SL2} below exhibits an infinite-dimensional faithfully graded Hopf algebra that is not a crossed product.


\section{Gradings and cocentral morphisms}\label{sec:Grad}

We identify faithful $G$-gradings with suitable cocentral morphisms to $\kk[G]$ and study the $G/N$-gradings induced by quotient maps $G\to G/N$, as well as their descent through quasi-Hopf quotients.

\subsection{The grading correspondence}\label{subsec:Cocentral}

Recall that a morphism of quasi-Hopf algebras $\pi\colon H\to H'$ is said to be \emph{cocentral} if it satisfies
\[
x'\otimes \pi(x'')=x''\otimes \pi(x'),
\qquad \text{for all } x\in H.
\]
\begin{definition}[Grading morphisms]\label{def:GradMor}
Let $\pi\colon H\to \kk[G]$ be a cocentral morphism of quasi-Hopf algebras. Consider the map $\rho_\pi:=(\pi\otimes \id)\Delta\colon H\to \kk[G]\otimes H$ and, for $g\in G$, the \emph{weight spaces} $H_g^\pi := \{\,x\in H\mid \rho_\pi(x)=\zeta_g\otimes x\,\}$.
\begin{enumerate}
    \item We call $\pi$ \emph{strict} if $\Psi_\pi:=(\pi\otimes \pi\otimes \id)(\Phi)=\zeta_e\otimes \zeta_e\otimes 1$.

    \item We say that $\pi$ is a \emph{$G$-grading morphism} if it is surjective and
    \[
    \Phi\in (H_e^\pi)^{\otimes 3},
    \qquad
    \alpha,\beta\in H_e^\pi,
    \qquad
    \mathfrak f\in (H_e^\pi)^{\otimes 2},
    \]
    where $\mathfrak{f}$ is the Drinfeld twist of $H$.
\end{enumerate}
\end{definition}

In contrast to the Hopf case, this last condition is not vacuous: it is used below to prove $S(H_g^\pi)\subseteq H_{g^{-1}}^\pi$, which does not follow from the other axioms.

Note that applying $\id\otimes\varepsilon$ to the identity $\rho_\pi(x)=\zeta_g\otimes x$ defining the weight space $H_g^\pi$, and using the counit axiom~\eqref{eq:QH4}, gives
\begin{equation}\label{eq:WeightId}
\pi(x)=\varepsilon(x)\,\zeta_g,
\qquad
\text{for all } x\in H_g^\pi.
\end{equation}
Every $G$-grading morphism is therefore strict: applying~\eqref{eq:WeightId} with $g=e$ to the first two legs of $\Phi\in (H_e^\pi)^{\otimes 3}$ and then~\eqref{eq:QH5} gives $\Psi_\pi=\zeta_e\otimes\zeta_e\otimes(\varepsilon\otimes\varepsilon\otimes\id)(\Phi) =\zeta_e\otimes\zeta_e\otimes 1$.

\begin{lemma}\label{lem:Coaction}
If $\pi\colon H\to \kk[G]$ is a strict cocentral morphism of quasi-Hopf algebras, then $\rho_\pi$ is a counital and coassociative left $\kk[G]$-coaction. It makes $H$ a left $\kk[G]$-comodule algebra.
\end{lemma}

\begin{proof}
Counitality follows from the counit axiom~\eqref{eq:QH4} and the identity $\varepsilon_{\kk[G]}\pi=\varepsilon_H$. For coassociativity, using quasi-coassociativity~\eqref{eq:QH2} we obtain, for $x\in H$,
\begin{align*}
(\Delta_{\kk[G]}\otimes \id)\rho_\pi(x)
&=(\pi\otimes \pi\otimes \id)(\Delta\otimes \id)\Delta(x) = \Psi_\pi^{-1}
(\pi\otimes \pi\otimes \id)(\id\otimes \Delta)\Delta(x)
\Psi_\pi \\
&=(\pi\otimes \pi\otimes \id)(\id\otimes \Delta)\Delta(x) = (\id\otimes \rho_\pi)\rho_\pi(x),
\end{align*}
where the third equality uses that $\pi$ is strict. Finally, $\rho_\pi$ is an algebra morphism, being a composite of algebra morphisms.
\end{proof}

This result explains the terminology in Definition~\ref{def:GradMor}(i): the coaction induced by a strict morphism is genuinely coassociative. For a general cocentral morphism, the same computation gives coassociativity only up to conjugation by $\Psi_\pi$, making $H$ a twisted comodule algebra in the sense of~\cite{BCM86,HN99}. Compare with~\cite[Rem.~2.1]{MN14}, where the analogous right coaction exhibits the grading of a cleft extension of a group algebra by a dual group algebra.

\begin{proposition}\label{prop:Cocentral1}
Let $H$ be a faithfully $G$-graded quasi-Hopf algebra. Then there exists a $G$-grading morphism $\pi\colon H\to \kk[G]$, given on homogeneous elements by
\begin{equation}\label{eq:GradMap}
\pi(x_g)=\varepsilon(x_g)\,\zeta_g,
\qquad
\text{for all } x_g\in H_g.
\end{equation}
The morphism $\pi$ is $G$-graded, and its weight spaces recover the given grading, i.e., $H_g^\pi=H_g$ for every $g\in G$.
\end{proposition}

\begin{proof}
The assignment in~\eqref{eq:GradMap} extends uniquely to a degree-preserving $\kk$-linear map $\pi\colon H\to \kk[G]$ with $\pi(H_g)\subseteq\kk\,\zeta_g=(\kk[G])_g$. It remains to show below that $\pi$ is a morphism of quasi-Hopf algebras. Every component is nonzero by faithfulness. By Proposition~\ref{prop:GradProp}(i) there exists $t_g\in H_g$ with $\varepsilon(t_g)=1$ for each $g\in G$, whence $\pi(t_g)=\zeta_g$ and $\pi$ is surjective.

We check that $\pi$ is a morphism of quasi-Hopf algebras. If $x_g\in H_g$ and $y_h\in H_h$, then $x_gy_h\in H_{gh}$, and hence
\[
\pi(x_gy_h) =\varepsilon(x_gy_h)\,\zeta_{gh}  =\varepsilon(x_g)\varepsilon(y_h)\,\zeta_g\zeta_h  =\pi(x_g)\pi(y_h).
\]
The unit is preserved because $1\in H_e$. Compatibility with the comultiplication follows in the same way from $\Delta(x_g)\in H_g\otimes H_g$ and the counit axiom~\eqref{eq:QH4}, which give
\[
(\pi\otimes\pi)\Delta(x_g)=\varepsilon(x_g)\,\zeta_g\otimes\zeta_g=\Delta_{\kk[G]}(\pi(x_g)),
\]
and compatibility with the counit is immediate. Since $\Phi\in H_e^{\otimes3}$ and $\alpha,\beta\in H_e$ with $\varepsilon(\alpha)=\varepsilon(\beta)=1$, identity~\eqref{eq:QH3} gives $\pi^{\otimes3}(\Phi)=\zeta_e\otimes\zeta_e\otimes\zeta_e$ and $\pi(\alpha)=\zeta_e=\pi(\beta)$. Finally, $S(x_g)\in H_{g^{-1}}$ and $\varepsilon\circ S=\varepsilon$ give $\pi(S(x_g))=\varepsilon(x_g)\,\zeta_{g^{-1}}=S_{\kk[G]}(\pi(x_g))$. Thus $\pi$ is a morphism of quasi-Hopf algebras.

For $x_g\in H_g$, the counit axiom~\eqref{eq:QH4} gives
\[
x_g'\otimes \pi(x_g'') =x_g'\otimes \varepsilon(x_g'')\,\zeta_g  =x_g\otimes \zeta_g  =x_g''\otimes \varepsilon(x_g')\,\zeta_g  =x_g''\otimes \pi(x_g').
\]
Hence $\pi$ is cocentral.

It remains to verify the membership conditions of Definition~\ref{def:GradMor}(ii), since we have already proved that $\pi$ is surjective. The weight spaces of $\pi$ coincide with the original homogeneous components, which also proves the last assertion of the proposition. Indeed, if $x_g\in H_g$, then $(\pi\otimes\id)\Delta(x_g)=\varepsilon(x_g')\zeta_g\otimes x_g''=\zeta_g\otimes x_g$. Conversely, if $x\in H_g^\pi$ has homogeneous decomposition $x=\sum_h x_h$, then $\zeta_g\otimes x=(\pi\otimes\id)\Delta(x)=\sum_h\zeta_h\otimes x_h$, and linear independence of the $\zeta_h$ gives $x_h=0$ for $h\neq g$. Hence $H_g^\pi=H_g$ for all $g\in G$.

Since $H_g^\pi=H_g$, the conditions $\Phi\in (H_e^\pi)^{\otimes 3}$ and $\alpha,\beta\in H_e^\pi$ hold by~\eqref{eq:Grad2}. The Drinfeld twist $\mathfrak f^{\pm 1}$ is given by explicit formulas in terms of $\Phi^{\pm 1}$, $\alpha$, $\beta$, $\Delta$, and $S$. Since $\Phi^{\pm 1}\in H_e^{\otimes 3}$ (Proposition~\ref{prop:GradProp}(ii)) and all the operations involved preserve $H_e$, we conclude $\mathfrak f^{\pm 1}\in H_e^{\otimes 2}$. Thus $\pi$ is a $G$-grading morphism.
\end{proof}

Before proving the converse, recall the following fact (see, e.g.,~\cite[Ex.~1.6.7]{Mon93}). If $(V,\rho)$ is a left $\kk[G]$-comodule, then $V_g:=\{\,v\in V\mid \rho(v)=\zeta_g\otimes v\,\}$ defines a direct sum decomposition $V=\bigoplus_{g\in G}V_g$.

\begin{proposition}\label{prop:Cocentral2}
Let $H$ be a quasi-Hopf algebra and let $\pi\colon H\to \kk[G]$ be a $G$-grading morphism. For each $g\in G$, set
\begin{equation}\label{eq:Weights}
H_g:=H_g^\pi
=
\{\,x\in H\mid (\pi\otimes \id)\Delta(x)=\zeta_g\otimes x\,\}.
\end{equation}
Then the family $\{H_g\}_{g\in G}$ defines a faithful, and hence strong, $G$-grading on $H$ as a quasi-Hopf algebra.
\end{proposition}

\begin{proof}
By~\eqref{eq:Weights}, the identity~\eqref{eq:WeightId} reads $\pi(x)=\varepsilon(x)\,\zeta_g$ for all $x\in H_g$. Moreover, since every $G$-grading morphism is strict, Lemma~\ref{lem:Coaction} shows that $H$ is a left $\kk[G]$-comodule algebra, and the fact recalled above gives $H=\bigoplus_{g\in G}H_g$. Also, $1\in H_e$, since $\rho_\pi(1)=\pi(1)\otimes 1=\zeta_e\otimes 1$.

We prove that every component is nonzero. Since $\pi$ is surjective, for each $g\in G$ there exists $x\in H$ such that $\pi(x)=\zeta_g$. Write $x=\sum_h x_h$ with $x_h\in H_h$. By~\eqref{eq:WeightId}, $\pi(x_h)=\varepsilon(x_h)\,\zeta_h$, and therefore
\[
\zeta_g=\pi(x)=\sum_h \varepsilon(x_h)\zeta_h.
\]
By linear independence, $\varepsilon(x_g)=1$, hence $x_g\neq 0$. Therefore $H_g\neq 0$.

We check that the multiplication is compatible with the grading. If $x\in H_g$ and $y\in H_h$, then, since $\rho_\pi$ is an algebra morphism,
\[
(\pi\otimes\id)\Delta(xy)=(\pi\otimes\id)\Delta(x)(\pi\otimes\id)\Delta(y)=(\zeta_g\otimes x)(\zeta_h\otimes y)=\zeta_{gh}\otimes xy.
\]
Thus $xy\in H_{gh}$. We also prove that $\Delta(H_g)\subseteq H_g\otimes H_g$ for all $g\in G$. Let $x\in H_g$. Applying $\id\otimes \Delta$ to $(\pi\otimes \id)\Delta(x)=\zeta_g\otimes x$ gives
\begin{equation}\label{eq:ComultL}
(\pi\otimes \id\otimes \id)(\id\otimes \Delta)\Delta(x)
=
\zeta_g\otimes \Delta(x).
\end{equation}
Since $\Phi\in H_e^{\otimes 3}$, \eqref{eq:WeightId} and the first identity of~\eqref{eq:QH5} give $(\pi\otimes \id\otimes \id)(\Phi) = \zeta_e\otimes 1\otimes 1$. Because $\pi\otimes\id\otimes\id$ is an algebra morphism, also $(\pi\otimes \id\otimes \id)(\Phi^{-1})=\zeta_e\otimes 1\otimes 1$. Thus quasi-coassociativity~\eqref{eq:QH2} implies
\[
(\pi\otimes \id\otimes \id)(\id\otimes \Delta)\Delta(x)
=
(\pi\otimes \id\otimes \id)(\Delta\otimes \id)\Delta(x).
\]
Combining this with~\eqref{eq:ComultL}, we get $(\rho_\pi\otimes \id)\Delta(x)=\zeta_g\otimes \Delta(x)$. Hence $\Delta(x)\in H_g\otimes H$, since $H\otimes H=\bigoplus_{h}H_h\otimes H$ and an element $T\in H\otimes H$ lies in $H_h\otimes H$ if and only if $(\rho_\pi\otimes\id)(T)=\zeta_h\otimes T$.

Symmetrically, cocentrality gives $(\id\otimes\pi)\Delta(x)=x\otimes\zeta_g$, and identifies $\rho_\pi':=(\id\otimes\pi)\Delta$ with the composite of $\rho_\pi$ and the flip $\kk[G]\otimes H\to H\otimes\kk[G]$. Hence $y\in H$ lies in $H_h$ if and only if $\rho_\pi'(y)=y\otimes\zeta_h$. The argument above, applied on the right and using the second identity of~\eqref{eq:QH5}, then gives $(\id\otimes\rho_\pi')\Delta(x)=\Delta(x)\otimes\zeta_g$, whence $\Delta(x)\in H\otimes H_g$. Therefore $\Delta(x)\in (H_g\otimes H)\cap (H\otimes H_g)=H_g\otimes H_g$, as one sees by comparing components in the decomposition $H\otimes H=\bigoplus_{h,k}H_h\otimes H_k$, and hence $\Delta(H_g)\subseteq H_g\otimes H_g$.

We prove compatibility with the antipode. Since $\mathfrak f\in H_e^{\otimes 2}$, \eqref{eq:WeightId} and the normalization $\varepsilon(\mathfrak f^{(1)})\,\mathfrak f^{(2)}=1$ give $(\pi\otimes \id)(\mathfrak f)=\zeta_e\otimes 1$. Applying the algebra morphism $\pi\otimes\id$ to $\mathfrak f\,\mathfrak f^{-1}=1\otimes 1$ then gives $(\pi\otimes \id)(\mathfrak f^{-1})=\zeta_e\otimes 1$ as well. Let $x\in H_g$. Using the defining property~\eqref{eq:QH8} of the Drinfeld twist, we compute
\[
(\pi\otimes \id)\Delta(S(x)) =(\pi\otimes \id)\left( \mathfrak f^{-1} (S\otimes S)(\Delta^{\cop}(x)) \mathfrak f \right) =\pi(S(x''))\otimes S(x'),
\]
where the second equality uses $(\pi\otimes\id)(\mathfrak f^{\pm 1})=\zeta_e\otimes 1$ and the multiplicativity of $\pi\otimes\id$. Since $\pi$ intertwines antipodes and $\Delta(x)\in H_g\otimes H_g$,
\[
\pi(S(x'')) =S_{\kk[G]}(\pi(x'')) =\varepsilon(x'')\zeta_{g^{-1}}.
\]
Hence
\[
(\pi\otimes \id)\Delta(S(x)) =\varepsilon(x'')\zeta_{g^{-1}}\otimes S(x') =\zeta_{g^{-1}}\otimes S(x).
\]
Thus $S(x)\in H_{g^{-1}}$, and hence $S(H_g)\subseteq H_{g^{-1}}$.

Finally, the membership conditions of Definition~\ref{def:GradMor}(ii) give $\Phi\in H_e^{\otimes 3}$ and $\alpha,\beta\in H_e$, which are the remaining requirements of~\eqref{eq:Grad2}. Hence $\{H_g\}_{g\in G}$ is a faithful $G$-grading of the quasi-Hopf algebra $H$, every component being nonzero, and it is strong by Proposition~\ref{prop:FaithStrong}(iii).
\end{proof}

The two preceding propositions pass in both directions between faithful gradings and grading morphisms, by mutually inverse constructions.

\begin{theorem}\label{thm:Cocentral}
Let $H$ be a quasi-Hopf algebra. The assignments
\[
H=\bigoplus_{g\in G}H_g\longmapsto \pi
\qquad\text{and}\qquad
\pi\longmapsto H=\bigoplus_{g\in G}H_g^\pi
\]
define a bijective correspondence between the set of faithful $G$-gradings of $H$ as a quasi-Hopf algebra and the set of $G$-grading morphisms $H\to \kk[G]$, where $\pi$ is the morphism given by $\pi(x_g)=\varepsilon(x_g)\zeta_g$ for all $x_g\in H_g$, and $H=\bigoplus_{g\in G}H_g^\pi$ is the decomposition of $H$ into the weight spaces $H_g^\pi$ of $\pi$.
\end{theorem}

\begin{proof}
By Propositions~\ref{prop:Cocentral1} and~\ref{prop:Cocentral2}, both assignments are well defined. It remains to check that they are mutually inverse.

Start with a $G$-grading morphism $\pi$ and let $H=\bigoplus_g H_g^\pi$ be the faithful grading provided by Proposition~\ref{prop:Cocentral2}. The morphism $\widehat \pi$ associated to this grading by Proposition~\ref{prop:Cocentral1} satisfies $\widehat \pi(x)=\varepsilon(x)\zeta_g$ for $x\in H_g^\pi$, as does $\pi$ by~\eqref{eq:WeightId}. Hence $\widehat \pi=\pi$.

Conversely, starting from a faithful $G$-grading, the composite is the identity by the last assertion of Proposition~\ref{prop:Cocentral1}.
\end{proof}

In the Hopf case this recovers the correspondence of~\cite[Lem.~1.3]{BNY16}. We call $\pi$ the \emph{grading map} of the grading.

\subsection{Gradings over quotient groups and descent}\label{subsec:QuotGrad}

We next describe the induced gradings over quotient groups and their descent through quasi-Hopf quotients.

\begin{proposition}\label{prop:Coarsening}
Let $H$ be a faithfully $G$-graded quasi-Hopf algebra with grading map $\pi\colon H\to \kk[G]$, let $N\unlhd G$, and let $p\colon \kk[G]\to \kk[G/N]$ be the canonical Hopf algebra projection, with $p(\zeta_h)=\zeta_{\bar h}$ for all $h\in G$. Then:
\begin{enumerate}
    \item The composite $p\pi\colon H\to \kk[G/N]$ is a $G/N$-grading morphism, and the induced $G/N$-grading of $H$ has homogeneous components
    \[
    H_{\bar g}
    =
    \bigoplus_{n\in N}H_{gn},
    \qquad
    \text{for all } \bar g=gN\in G/N.
    \]

    \item The induced $G/N$-grading is strong as well.
\end{enumerate}
\end{proposition}

\begin{proof}
(i) The composite $p\pi$ is a surjective morphism of quasi-Hopf algebras, since $\pi$ and $p$ are. Moreover, it inherits cocentrality from $\pi$, by applying $\id\otimes p$ to the cocentrality identity of $\pi$.

We compute the weight spaces of $p\pi$. Let $x=\sum_{h\in G}x_h$, with $x_h\in H_h$, be the homogeneous decomposition of $x\in H$. Using $H_h=H_h^{\pi}$, we obtain
\[
(p\pi\otimes\id)\Delta(x)=\sum_{h\in G}\zeta_{\bar h}\otimes x_h=\sum_{\bar a\in G/N}\zeta_{\bar a}\otimes\left(\sum_{h\in\bar a}x_h\right).
\]
Since the $\zeta_{\bar a}$ are linearly independent and $H=\bigoplus_{h\in G}H_h$, we conclude that $(p\pi\otimes\id)\Delta(x)=\zeta_{\bar g}\otimes x$ if and only if $x_h=0$ for all $h\notin\bar g$, that is, $H_{\bar g}^{p\pi}=\bigoplus_{n\in N}H_{gn}$ for all $\bar g=gN\in G/N$. Since $H_e\subseteq H_{\bar e}^{p\pi}$, the memberships $\Phi\in H_e^{\otimes 3}$ and $\alpha,\beta\in H_e$ of~\eqref{eq:Grad2}, together with $\mathfrak f\in H_e^{\otimes 2}$ (since $H_e^{\pi}=H_e$), yield the membership conditions of Definition~\ref{def:GradMor}(ii) for $p\pi$. Thus $p\pi$ is a $G/N$-grading morphism, and Theorem~\ref{thm:Cocentral} gives the $G/N$-grading with components $H_{\bar g}=H_{\bar g}^{p\pi}$.

(ii) By Proposition~\ref{prop:FaithStrong}(iii), the $G$-grading is strong. Fix $\bar g\in G/N$ and note that $1\in H_e=H_gH_{g^{-1}}\subseteq H_{\bar g}H_{\bar g^{-1}}$, by the inclusions $H_g\subseteq H_{\bar g}$ and $H_{g^{-1}}\subseteq H_{\overline{g^{-1}}}=H_{\bar g^{-1}}$ of part (i), which also give $H_{\bar g}H_{\bar g^{-1}}\subseteq H_{\bar e}$. For the reverse inclusion,
\[
H_{\bar e} = H_{\bar e}\,1 \subseteq H_{\bar e}H_{\bar g}H_{\bar g^{-1}} \subseteq H_{\bar g}H_{\bar g^{-1}}.
\]
Thus $H_{\bar g}H_{\bar g^{-1}}=H_{\bar e}$, and the induced $G/N$-grading is strong.
\end{proof}

A similar argument applies to quotients of $H$ itself.

\begin{proposition}\label{prop:Descent}
Let $H$ be a faithfully $G$-graded quasi-Hopf algebra with grading map $\pi\colon H\to \kk[G]$, let $I\subseteq H$ be a quasi-Hopf ideal, and let $q\colon H\to H/I$ be the quotient map. Suppose that there exist a Hopf algebra $Q$, a surjective morphism of Hopf algebras $p\colon \kk[G]\to Q$, and a morphism of quasi-Hopf algebras $\overline{\pi}\colon H/I\to Q$ such that $\overline{\pi}q=p\pi$. Then:
\begin{enumerate}
    \item $Q\cong\kk[\Gamma]$ as a Hopf algebra, where $\Gamma$ is the group of group-like elements of $Q$.

    \item Under the identification of (i), the morphisms $\overline{\pi}$ and $p\pi$ are $\Gamma$-grading morphisms. In particular, $\overline{\pi}$ induces a strong $\Gamma$-grading $H/I=\bigoplus_{a\in\Gamma}(H/I)_a$, and $p\pi$ induces a strong $\Gamma$-grading $H=\bigoplus_{a\in\Gamma}H_a^{p\pi}$.

    \item $q(H_a^{p\pi})\subseteq (H/I)_a$ for every $a\in \Gamma$. Hence $q$ is a $\Gamma$-graded morphism of quasi-Hopf algebras.
\end{enumerate}
\end{proposition}

\begin{proof}
(i) The elements $p(\zeta_g)$, for $g\in G$, are group-like and span $Q$. Since group-like elements are linearly independent, the pairwise distinct ones among them form a basis of $Q$. Hence $Q=\kk[\Gamma]$, where $\Gamma:=\{p(\zeta_g)\mid g\in G\}$ is the image of $G$ under the group homomorphism $g\mapsto p(\zeta_g)$. Every group-like element of $Q$ lies in $\Gamma$, since for group-like $u=\sum_{a\in\Gamma}c_a\zeta_a$ comparison of $\Delta(u)$ with $u\otimes u$ forces all but one coefficient to vanish, and the counit forces the remaining coefficient to equal $1$. We henceforth regard $\overline{\pi}$ as a morphism $H/I\to \kk[\Gamma]$, and we write $\bar e$ for the neutral element of $\Gamma$, with $p(\zeta_e)=\zeta_{\bar e}$.

(ii) The composite $p\pi$ is surjective, since $\pi$ and $p$ are, and $\overline{\pi}$ is surjective because $\overline{\pi}q=p\pi$. Both are cocentral: applying $\id\otimes p$ to the cocentrality identity of $\pi$ gives it for $p\pi$, and, since every element of $H/I$ has the form $q(x)$, applying $q\otimes\id$ to the result, using $\overline{\Delta}q=(q\otimes q)\Delta$ and $\overline{\pi}q=p\pi$, gives it for $\overline{\pi}$. For the membership conditions, let $u\in H_e=H_e^{\pi}$. Then
\[
(\overline{\pi}\otimes \id)\overline{\Delta}(q(u)) =(\overline{\pi}q\otimes q)\Delta(u) =(p\pi\otimes q)\Delta(u) =(p\otimes q)(\zeta_e\otimes u) =\zeta_{\bar e}\otimes q(u),
\]
and hence $q(H_e)\subseteq (H/I)_{\bar e}^{\overline{\pi}}$. The same computation with $\id_H$ in place of $q$ and $p\pi$ in place of $\overline{\pi}$ gives $(p\pi\otimes\id)\Delta(u)=\zeta_{\bar e}\otimes u$; hence $H_e\subseteq H_{\bar e}^{p\pi}$. By~\eqref{eq:Grad2} and $\mathfrak f\in H_e^{\otimes 2}$ (since $H_e^{\pi}=H_e$), the elements $\Phi$, $\alpha$, $\beta$, and $\mathfrak f$ lie in the corresponding tensor powers of $H_e$. Hence $p\pi$ satisfies the membership conditions of Definition~\ref{def:GradMor}(ii). Their images $\overline{\Phi}=q^{\otimes 3}(\Phi)$, $\overline{\alpha}=q(\alpha)$, $\overline{\beta}=q(\beta)$, and $\overline{\mathfrak f}=q^{\otimes 2}(\mathfrak f)$ lie in the corresponding tensor powers of $q(H_e)\subseteq (H/I)_{\bar e}^{\overline{\pi}}$; therefore $\overline{\pi}$ satisfies them as well. Thus both are $\Gamma$-grading morphisms. Theorem~\ref{thm:Cocentral} gives faithful $\Gamma$-gradings $H/I=\bigoplus_{a\in\Gamma}(H/I)_a$ and $H=\bigoplus_{a\in\Gamma}H_a^{p\pi}$, which are strong by Proposition~\ref{prop:FaithStrong}(iii).

(iii) Let $x\in H_a^{p\pi}$ for $a\in\Gamma$, that is, $(p\pi\otimes \id)\Delta(x)=\zeta_a\otimes x$. Applying $\id\otimes q$ and using $\overline{\pi}q=p\pi$ together with $\overline{\Delta}q=(q\otimes q)\Delta$ gives $(\overline{\pi}\otimes \id)\overline{\Delta}(q(x))=\zeta_a\otimes q(x)$. Hence $q(H_a^{p\pi})\subseteq (H/I)_a$. Since $q$ is a morphism of quasi-Hopf algebras, it is a $\Gamma$-graded morphism.
\end{proof}

\section{General classification via admissible pairs}\label{sec:Class}

We now classify quasi-Hopf ideals by passing to an appropriate grading over a quotient of $G$ and reducing to its neutral component. Throughout this section, $H$ is a faithfully $G$-graded quasi-Hopf algebra with grading map $\pi\colon H\to \kk[G]$.

\subsection{The associated subgroup and reduction to the neutral component}\label{subsec:AssocSubgrp}

Since $\pi$ is a surjective morphism of quasi-Hopf algebras, it sends quasi-Hopf ideals of $H$ to Hopf ideals of $\kk[G]$. This allows us to describe the image under $\pi$ of a quasi-Hopf ideal by a normal subgroup of $G$.

Recall that the \emph{augmentation ideal} $\aug{\Gamma}$ of the group algebra $\kk[\Gamma]$ of a group $\Gamma$ is the kernel of its counit. It is spanned by the elements $\zeta_a-\zeta_e$ with $a\in\Gamma$ (see, e.g.,~\cite[\S I.1]{Pas79}).

\begin{proposition}\label{prop:AssocSubgrp}
Let $I \subseteq H$ be a quasi-Hopf ideal. Then:
\begin{enumerate}
    \item There exists a unique normal subgroup $N_I \unlhd G$ such that $\pi(I) = \aug{N_I} \cdot \kk[G]$.

    \item The map $\pi$ induces a unique surjective cocentral morphism of quasi-Hopf algebras $\overline{\pi}\colon H/I \to \kk[G/N_I]$ satisfying $\overline{\pi}q=p\pi$, where $q\colon H\to H/I$ and $p\colon \kk[G]\to \kk[G/N_I]$ denote the canonical projections.

\item The ideal $I$ is graded with respect to the induced $G/N_I$-grading of $H$, that is,
    \[
    I = \bigoplus_{\bar g \in G/N_I} (I \cap H_{\bar g}),
    \qquad
    \text{where } H_{\bar g} = \bigoplus_{n \in N_I} H_{gn}.
    \]
\end{enumerate}
\end{proposition}

\begin{proof}
(i) The image $\pi(I)$ is a Hopf ideal of $\kk[G]$, and it is proper, since $\varepsilon_{\kk[G]}(\pi(I))=\varepsilon(I)=0$ while $\varepsilon(\zeta_e)=1$. The quotient $Q:=\kk[G]/\pi(I)$ is therefore a nonzero Hopf algebra and, as in the proof of Proposition~\ref{prop:Descent}(i), $Q=\kk[\Gamma]$ for the group $\Gamma$ formed by the images of the elements $\zeta_g$. The map sending $g$ to the image of $\zeta_g$ is a surjective group homomorphism $G\to\Gamma$. Let $N_I$ denote its kernel. Then $\Gamma\cong G/N_I$, and under this identification the quotient map $\kk[G]\to Q$ becomes the canonical projection $p\colon \kk[G]\to \kk[G/N_I]$. Hence $\pi(I)=\ker(p)$. The kernel of $p$ is spanned by the elements $\zeta_{ng}-\zeta_g=(\zeta_n-\zeta_e)\zeta_g$ with $n\in N_I$ and $g\in G$, since an element of $\kk[G]$ lies in $\ker(p)$ if and only if its coefficients sum to zero over each coset $N_I g$. Hence $\pi(I)=\aug{N_I}\cdot\kk[G]$. This subgroup is unique because, for any normal subgroup $N\unlhd G$, the same computation gives $\aug{N}\cdot\kk[G]=\ker\left(\kk[G]\to \kk[G/N]\right)$; hence $N$ is recovered from the ideal $\aug{N}\cdot\kk[G]$ as $\{n\in G\mid \zeta_n-\zeta_e\in \aug{N}\cdot\kk[G]\}$.

(ii) Since $\pi(I)=\ker(p)$, the composite $p\pi$ is a morphism of quasi-Hopf algebras (Proposition~\ref{prop:Coarsening}(i)) vanishing on $I$. The universal property of quotient quasi-Hopf algebras therefore yields a unique morphism of quasi-Hopf algebras $\overline{\pi}\colon H/I\to\kk[G/N_I]$ satisfying $\overline{\pi}q=p\pi$. Since $\overline{\pi}q=p\pi$ is surjective, $\overline{\pi}$ is surjective as well. Cocentrality follows by applying $q\otimes\id$ to the cocentrality identity of $p\pi$, using $\overline{\Delta}q=(q\otimes q)\Delta$ and $\overline{\pi}q=p\pi$.

(iii) By Proposition~\ref{prop:Coarsening}(i), $H$ is $G/N_I$-graded via $p\pi\colon H \to \kk[G/N_I]$, with homogeneous components $H_{\bar g}=\bigoplus_{n\in N_I}H_{gn}$ equal to the weight spaces $H_{\bar g}^{p\pi}$. By part (ii) and Proposition~\ref{prop:Descent}(i)--(iii), applied with $Q=\kk[G/N_I]$, whose group of group-like elements is $G/N_I$, the quotient $H/I$ inherits a $G/N_I$-grading and the projection $q$ is a $G/N_I$-graded morphism. The kernel of a graded morphism is a graded ideal. Indeed, if $x=\sum_{\bar g}x_{\bar g}\in I$ with $x_{\bar g}\in H_{\bar g}$, then $0=q(x)=\sum_{\bar g}q(x_{\bar g})$ with $q(x_{\bar g})\in (H/I)_{\bar g}$. Hence $q(x_{\bar g})=0$ for every $\bar g$, that is, $x_{\bar g}\in I\cap H_{\bar g}$ for every $\bar g$. Therefore $I=\bigoplus_{\bar g\in G/N_I}(I\cap H_{\bar g})$, the sum being direct because $I\cap H_{\bar g}\subseteq H_{\bar g}$.
\end{proof}

We call $N_I$ the normal subgroup \emph{associated to} the quasi-Hopf ideal $I$. Explicitly,
\[
N_I=\{\,g\in G\mid \zeta_g-\zeta_e\in\pi(I)\,\},
\]
the usual subgroup determined by an ideal of a group algebra (see, e.g.,~\cite[\S II.1]{Pas79}).

For a subgroup $N\leq G$, we write
\[
H_N:=\bigoplus_{n\in N}H_n.
\]

\begin{remark}\label{rem:HNSubalg}
For every subgroup $N\leq G$, the subspace $H_N$ is a quasi-Hopf subalgebra of $H$. Indeed, $H_N$ is a subalgebra, since $H_nH_m\subseteq H_{nm}$ for all $n,m\in N$, and it contains $1\in H_e$. It satisfies $\Phi^{\pm 1}\in H_e^{\otimes 3}$ and $\alpha,\beta\in H_e$ (Proposition~\ref{prop:GradProp}(ii)), as well as $\Delta(H_n)\subseteq H_n\otimes H_n$ and $S(H_n)\subseteq H_{n^{-1}}$ for all $n\in N$. Moreover, $H_N$ is itself faithfully $N$-graded, and hence strongly $N$-graded by Proposition~\ref{prop:FaithStrong}(iii), with neutral component $H_e$. Its grading map is the corestriction of $\pi|_{H_N}$ to $\kk[N]$, by~\eqref{eq:GradMap}.
\end{remark}

Assume now that $N \unlhd G$. By Proposition~\ref{prop:Coarsening}(i)--(ii), the quotient group $G/N$ induces a strong $G/N$-grading on $H$, whose neutral component is precisely $H_N$:
\[
H=\bigoplus_{\bar g\in G/N}H_{\bar g},
\qquad
H_{\bar g}=\bigoplus_{n\in N}H_{gn},
\qquad
H_{\bar e}=H_N.
\]
\begin{definition}[Invariant ideal]\label{def:InvIdeal}
Let $A=\bigoplus_{a\in\Gamma}A_a$ be a strongly $\Gamma$-graded algebra, for an arbitrary group $\Gamma$. An ideal $J$ of the neutral component $A_e$ is called \emph{$\Gamma$-invariant} if
\[
A_a J=J A_a
\qquad
\text{for all }a\in\Gamma.
\]
\end{definition}

Note that if $A$ is a crossed product, with homogeneous unit $u_a\in A_a^\times$ for each $a\in\Gamma$, then $\Gamma$-invariance amounts to $u_aJu_a^{-1}=J$ for all $a\in\Gamma$.

\begin{proposition}\label{prop:Intersect}
Let $I$ be a two-sided ideal of $H$ and let $N\unlhd G$. Set $K:=I\cap H_N$. Then:
\begin{enumerate}
    \item For every coset $\bar g\in G/N$, we have $I\cap H_{\bar g}=H_{\bar g}K=K H_{\bar g}$.

    \item $K$ is a $G/N$-invariant ideal of $H_N$.
\end{enumerate}
\end{proposition}

\begin{proof}
(i) Fix $\bar g\in G/N$. The inclusions $H_{\bar g}K\subseteq I\cap H_{\bar g}$ and $K H_{\bar g}\subseteq I\cap H_{\bar g}$ are immediate, since $I$ is a two-sided ideal and $K=I\cap H_N\subseteq H_{\bar e}$. Indeed, $H_{\bar g}K\subseteq H_{\bar g}H_{\bar e}\subseteq H_{\bar g}$ and likewise for $KH_{\bar g}$.

Conversely, let $x\in I\cap H_{\bar g}$. Since the induced $G/N$-grading is strong (Proposition~\ref{prop:Coarsening}(ii)), we have $H_{\bar g}H_{\bar g^{-1}}=H_{\bar g^{-1}}H_{\bar g}=H_N$. Since $1\in H_e\subseteq H_N$, we may choose decompositions
\[
1=\sum_i a_i b_i=\sum_j c_j d_j,
\qquad
a_i,d_j\in H_{\bar g},
\quad
b_i,c_j\in H_{\bar g^{-1}}.
\]
Then $x=\sum_i a_i(b_i x)$, where each $b_i x$ lies in $H_{\bar g^{-1}}H_{\bar g}\subseteq H_N$ and, since $x\in I$, also in $I$. Thus $b_i x\in K$ and $x\in H_{\bar g}K$. Similarly, $x=\sum_j (x c_j)d_j$ with $x c_j\in H_{\bar g}H_{\bar g^{-1}}\cap I\subseteq K$, whence $x\in K H_{\bar g}$, proving (i).

(ii) Since $I$ is a two-sided ideal of $H$ and $H_N$ is a subalgebra, $K$ is an ideal of $H_N$. By part (i), $H_{\bar g}K=I\cap H_{\bar g}=KH_{\bar g}$ for every $\bar g\in G/N$. Thus $K$ is $G/N$-invariant.
\end{proof}

Applied to the original $G$-grading, the preceding result also describes $I$ through its intersection with the neutral component $H_e$.

\begin{corollary}\label{cor:InvNeutral}
Let $I\subseteq H$ be a quasi-Hopf ideal. Then $J:=I\cap H_e$ is a $G$-invariant quasi-Hopf ideal of $H_e$, and $I\cap H_g=H_gJ=JH_g$ for all $g\in G$. Moreover, $J=K\cap H_e$, where $N:=N_I$ is the associated subgroup and $K:=I\cap H_N$.
\end{corollary}

\begin{proof}
The inclusion $H_e\hookrightarrow H$ (Proposition~\ref{prop:GradProp}(ii)) followed by the quotient morphism $q\colon H\to H/I$ is a morphism of quasi-Hopf algebras. Its kernel is precisely $I\cap H_e=J$, which is therefore a quasi-Hopf ideal of $H_e$. Applying Proposition~\ref{prop:Intersect} with $N=\{e\}$, where the induced $G/\{e\}$-grading is the original $G$-grading and $H_{\{e\}}=H_e$, parts (i) and (ii) give $I\cap H_g=H_gJ=JH_g$ for all $g\in G$ and the $G$-invariance of $J$. Finally, since $H_e\subseteq H_N$, we have $J=I\cap H_e=(I\cap H_N)\cap H_e=K\cap H_e$.
\end{proof}

\subsection{The classification theorem}\label{subsec:Main}

Recall that $H$ is faithfully $G$-graded, with grading map $\pi$. By Propositions~\ref{prop:AssocSubgrp} and~\ref{prop:Intersect}, every quasi-Hopf ideal $I\subseteq H$ determines a normal subgroup $N_I\unlhd G$ together with a $G/N_I$-invariant ideal $I\cap H_{N_I}$ of $H_{N_I}$. The following definition axiomatizes the pairs that arise in this way.

\begin{definition}[Admissible pair]\label{def:AdmPair}
An \emph{admissible pair} for $H$ is a pair $(N,K)$ such that:
\begin{enumerate}
    \item $N\unlhd G$;

    \item $K$ is a quasi-Hopf ideal of the quasi-Hopf subalgebra $H_N$;

    \item $K$ is $G/N$-invariant, that is, $H_{\bar g}K=K H_{\bar g}$ for every $\bar g\in G/N$;

    \item the image of $K$ under the grading map contains the augmentation ideal of $\kk[N]$, that is, $\aug{N}\subseteq\pi(K)$.
\end{enumerate}
\end{definition}

\begin{remark}\label{rem:AugCond}
The reverse inclusion to (iv) follows from (ii): $K\subseteq H_N$ gives $\pi(K)\subseteq\kk[N]$, and $\varepsilon_{\kk[G]}(\pi(K))=\varepsilon(K)=0$, hence $\pi(K)\subseteq\aug{N}$. Therefore, in an admissible pair, the equality $\pi(K)=\aug{N}$ holds. Nevertheless, condition (iv) does not follow from (i)--(iii): for any normal subgroup $N\neq\{e\}$, the pair $(N,0)$ satisfies (i)--(iii) but not (iv), and the correspondence below would send every such pair to the zero ideal.
\end{remark}

The following theorem shows that these conditions characterize such pairs, and that the pair determines the ideal.

\begin{theorem}\label{thm:MainA}
Let $H$ be a faithfully $G$-graded quasi-Hopf algebra. The assignments
\[
I\longmapsto \left(N_I,\,I\cap H_{N_I}\right)
\qquad\text{and}\qquad
(N,K)\longmapsto K^{G/N}:=\sum_{\bar g\in G/N}H_{\bar g}K
\]
define a bijective correspondence between the set of quasi-Hopf ideals of $H$ and the set of admissible pairs for $H$.
\end{theorem}

\begin{proof}
Let $I\subseteq H$ be a quasi-Hopf ideal, set $N=N_I$, and put $K=I\cap H_N$. We first show that $(N,K)$ is an admissible pair. The intersection $K$ is the kernel of the composite morphism of quasi-Hopf algebras $H_N\hookrightarrow H\xrightarrow{q}H/I$ (Remark~\ref{rem:HNSubalg}), hence $K$ is a quasi-Hopf ideal of $H_N$. By Proposition~\ref{prop:Intersect}(ii), $K$ is $G/N$-invariant.

It remains to verify the augmentation condition (iv). Let $n\in N$. Since $\zeta_n-\zeta_e\in \aug{N}\cdot\kk[G]=\pi(I)$ (Proposition~\ref{prop:AssocSubgrp}(i)), there exists $x\in I$ such that $\pi(x)=\zeta_n-\zeta_e$. Decompose $x=\sum_{\bar g\in G/N}x_{\bar g}$ with respect to the induced $G/N$-grading. By Proposition~\ref{prop:AssocSubgrp}(iii), each $x_{\bar g}$ lies in $I$. By~\eqref{eq:GradMap} and the description $H_{\bar g}=\bigoplus_{m\in N}H_{gm}$, each $\pi(x_{\bar g})$ lies in the span of $\{\zeta_h\mid h\in\bar g\}$. Since the cosets partition $G$, the condition $\pi(x)\in\kk[N]$ forces $\pi(x_{\bar g})=0$ for $\bar g\neq \bar e$. It follows that the neutral component $x_{\bar e}$ satisfies $\pi(x_{\bar e})=\zeta_n-\zeta_e$. But $x_{\bar e}\in I\cap H_N=K$. Therefore $\zeta_n-\zeta_e\in \pi(K)$ for every $n\in N$, and hence $\aug{N}\subseteq \pi(K)$.

Now let $(N,K)$ be an admissible pair and set $I:=K^{G/N}=\sum_{\bar g\in G/N}H_{\bar g}K$. Since $K$ is $G/N$-invariant, we also have $I=\sum_{\bar g\in G/N}K H_{\bar g}$. Since $H_{\bar h}H_{\bar g}\subseteq H_{\bar h\bar g}$, the first presentation gives $HI\subseteq I$ and the second gives $IH\subseteq I$. Hence $I$ is a two-sided ideal of $H$.

We prove that $I$ is a quasi-Hopf ideal. Since $K$ is a quasi-Hopf ideal of $H_N$, we have $\Delta(K)\subseteq K\otimes H_N+H_N\otimes K$, $S(K)\subseteq K$, and $\varepsilon(K)=0$. Let $a\in H_{\bar g}$ and $k\in K$. Since the induced $G/N$-grading makes $H$ a $G/N$-graded quasi-Hopf algebra (Proposition~\ref{prop:Coarsening}(i)), we have $\Delta(a)\in H_{\bar g}\otimes H_{\bar g}$, $S(a)\in H_{\bar g^{-1}}$, and $H_{\bar g}H_N\subseteq H_{\bar g}$. Thus
\[
\Delta(ak)=\Delta(a)\Delta(k)\in H_{\bar g}K\otimes H_{\bar g}+H_{\bar g}\otimes H_{\bar g}K\subseteq I\otimes H+H\otimes I.
\]
Also, $S(ak)=S(k)S(a)\in K H_{\bar g^{-1}}=H_{\bar g^{-1}}K\subseteq I$ and $\varepsilon(ak)=\varepsilon(a)\varepsilon(k)=0$. Therefore $I$ is a quasi-Hopf ideal of $H$.

We verify that the subgroup associated to $I$ is $N$. Indeed,
\[
\pi(I)
=
\sum_{\bar g\in G/N}\pi(H_{\bar g})\pi(K)
=
\aug{N}\cdot\kk[G],
\]
hence $N_I=N$ by Proposition~\ref{prop:AssocSubgrp}(i). Here the second equality uses $\pi(K)=\aug{N}$ (Remark~\ref{rem:AugCond}), the description $\pi(H_{\bar g})=\bigoplus_{n\in N}\kk\,\zeta_{gn}$ (by~\eqref{eq:GradMap} and Proposition~\ref{prop:GradProp}(i), every component being nonzero), and the identity $\kk[G]\cdot\aug{N}=\aug{N}\cdot\kk[G]$, valid for $N$ normal since $\zeta_g(\zeta_n-\zeta_e)=(\zeta_{gng^{-1}}-\zeta_e)\zeta_g$ and $n\mapsto gng^{-1}$ is a bijection of $N$.

Finally, the two assignments are mutually inverse. Starting from a quasi-Hopf ideal $I$, with $N=N_I$ as in the first part of the proof, Proposition~\ref{prop:AssocSubgrp}(iii) and Proposition~\ref{prop:Intersect}(i) give $I=\bigoplus_{\bar g\in G/N}(I\cap H_{\bar g})=\sum_{\bar g\in G/N}H_{\bar g}(I\cap H_N)$. Thus the reconstruction from $(N_I,I\cap H_N)$ returns $I$.

Conversely, starting from an admissible pair $(N,K)$ and forming $I=K^{G/N}$, the equality $N_I=N$ was proved above, and it remains to check that $I\cap H_N=K$. Since $H_{\bar g}K\subseteq H_{\bar g}$ for every $\bar g\in G/N$, comparison of the components in $H=\bigoplus_{\bar g\in G/N}H_{\bar g}$ gives $I\cap H_{\bar g}=H_{\bar g}K$. In particular $I\cap H_N=H_NK$, and $H_NK=K$ since $1\in H_N$ and $K$ is an ideal of $H_N$. Thus the recovered admissible pair is exactly $(N,K)$.
\end{proof}

\section{Split quotients}\label{sec:Explicit}

Throughout this section, $H$ is again a faithfully $G$-graded quasi-Hopf algebra with grading map $\pi\colon H\to\kk[G]$. We refine the admissible-pair classification by first passing to the neutral component and then isolating the ideals that admit a retraction.

\subsection{Lifting and reduced admissible data}\label{subsec:Lift}

The following construction lifts a $G$-invariant quasi-Hopf ideal of $H_e$ to $H$.

\begin{proposition}\label{prop:Lift}
Let $J\subseteq H_e$ be a $G$-invariant quasi-Hopf ideal of $H_e$, and set
\[
J^G:=HJ=\sum_{g\in G}H_gJ.
\]
Then:
\begin{enumerate}
    \item $J^G$ is a $G$-graded quasi-Hopf ideal of $H$, with components $J^G\cap H_g=H_gJ$ for all $g\in G$. Moreover, $J^G\cap H_e=J$.

    \item $H/J^G$ is a strongly $G$-graded quasi-Hopf algebra whose neutral component is canonically isomorphic, as a quasi-Hopf algebra, to $H_e/J$.
\end{enumerate}
\end{proposition}

\begin{proof}
(i) We show first that the pair $(\{e\},J)$ is an admissible pair for $H$. Indeed, $H_{\{e\}}=H_e$ and the induced $G/\{e\}$-grading is the original $G$-grading; therefore conditions (ii) and (iii) of Definition~\ref{def:AdmPair} are the hypotheses on $J$, condition (i) is immediate, and condition (iv) holds trivially, since $\aug{\{e\}}=0$. By Theorem~\ref{thm:MainA}, $J^{G/\{e\}}=\sum_{g\in G}H_gJ=J^G$ is a quasi-Hopf ideal of $H$. Since $H_gJ\subseteq H_g$ for all $g\in G$ and the decomposition $H=\bigoplus_{g\in G}H_g$ is direct, the sum $J^G=\sum_{g\in G}H_gJ$ is also direct and $J^G\cap H_g=H_gJ$. For the last assertion of (i), taking $g=e$ gives $J^G\cap H_e=H_eJ=J$, since $1\in H_e$ and $J$ is an ideal of $H_e$.

(ii) By Proposition~\ref{prop:FaithStrong}(iii), the grading of $H$ is strong. Let $q_J\colon H\to H/J^G$ be the quotient map. Since $J^G$ is graded, $H/J^G=\bigoplus_{g\in G}q_J(H_g)$. Since $q_J$ is a surjective morphism of quasi-Hopf algebras, it carries the algebra grading of $H$ and conditions~\eqref{eq:Grad1} and~\eqref{eq:Grad2} to the corresponding statements for the subspaces $q_J(H_g)$. Hence this decomposition makes $H/J^G$ a $G$-graded quasi-Hopf algebra, and the identities $q_J(H_g)q_J(H_{g^{-1}})=q_J(H_e)$ make the grading strong. Finally, the quasi-Hopf morphism $H_e\to q_J(H_e)$ has kernel $J^G\cap H_e=J$ by (i), and therefore induces $H_e/J\cong q_J(H_e)$.
\end{proof}

For a subgroup $N\leq G$, the same construction applies to the faithfully $N$-graded quasi-Hopf subalgebra $H_N$.

\begin{corollary}\label{cor:RelLift}
Let $J\subseteq H_e$ be a $G$-invariant quasi-Hopf ideal of $H_e$, let $N\leq G$ be a subgroup, and set $J^N:=H_NJ$. Then:
\begin{enumerate}
    \item $J^N=JH_N$ is a quasi-Hopf ideal of $H_N$ with $J^N\cap H_e=J$. Moreover, $J^N=J^G\cap H_N$.

    \item $H_N/J^N$ is a strongly $N$-graded quasi-Hopf algebra whose neutral component is canonically identified with $H_e/J$ via the image of $H_e$ in $H_N/J^N$.
\end{enumerate}
\end{corollary}

\begin{proof}
Since $J$ is $G$-invariant, it is in particular $N$-invariant; hence $H_NJ=JH_N$. Proposition~\ref{prop:Lift}, applied to the faithfully $N$-graded quasi-Hopf algebra $H_N$, then yields every remaining assertion except $J^N=J^G\cap H_N$. For the latter, Proposition~\ref{prop:Lift}(i) gives $J^G=\bigoplus_{g\in G}H_gJ$ with $H_gJ\subseteq H_g$. Intersecting this direct sum with $H_N$ gives $J^G\cap H_N=\bigoplus_{n\in N}H_nJ=H_NJ=J^N$.
\end{proof}

An admissible pair $(N,K)$ may be reduced further by passing to $J:=K\cap H_e$ and the quotient $H_N/J^N$. Note that $J^N=H_NJ\subseteq K$; hence the quotient $K/J^N$ makes sense. We refer to $H_N/J^N$ as the \emph{lifted quotient} of $H$ determined by $N$ and $J$.

\begin{lemma}\label{lem:RedSetup}
    Let $N\unlhd G$ and let $J\subseteq H_e$ be a $G$-invariant quasi-Hopf ideal of $H_e$, and write $q_J\colon H\to H/J^G$ for the canonical projection. Then:
\begin{enumerate}
    \item $H/J^G$ is strongly $G/N$-graded with components $q_J(H_{\bar g})$, $\bar g\in G/N$, and its neutral component $q_J(H_N)$ is canonically identified with $H_N/J^N$. Under this identification, $q_J|_{H_N}$ is the quotient map $H_N\to H_N/J^N$.

    \item The grading map of $H_N$ descends to a surjective morphism of quasi-Hopf algebras $\pi_{N,J}\colon H_N/J^N\to\kk[N]$, which is the grading map of $H_N/J^N$.
\end{enumerate}
\end{lemma}

\begin{proof}
(i) By Propositions~\ref{prop:Lift}(i)--(ii) and~\ref{prop:Coarsening}(i)--(ii), $H/J^G$ is strongly $G/N$-graded with components $q_J(H_{\bar g})$. Its neutral component is $q_J(H_N)\cong H_N/(J^G\cap H_N)=H_N/J^N$, since $J^G\cap H_N=J^N$ by Corollary~\ref{cor:RelLift}(i), and under this identification $q_J|_{H_N}$ is the quotient map.

(ii) The grading map of $H_N$ (Remark~\ref{rem:HNSubalg}) vanishes on $J^N$, since $\varepsilon(J)=0$ gives $\pi(H_nJ)=\varepsilon(H_nJ)\,\zeta_n=0$ by~\eqref{eq:GradMap}, and therefore descends, by the universal property of quotients, to a surjective morphism of quasi-Hopf algebras $\pi_{N,J}\colon H_N/J^N\to\kk[N]$. It is the grading map of $H_N/J^N$, as both maps send the class of $x_n\in H_n$ to $\varepsilon(x_n)\,\zeta_n$.
\end{proof}

We use the identifications of Corollary~\ref{cor:RelLift}(ii) and Lemma~\ref{lem:RedSetup}(i) throughout.

\begin{definition}[Reduced admissible datum]\label{def:RedAdmDatum}
A \emph{reduced admissible datum} for $H$ is a triple $(N,J,L)$ such that:
\begin{enumerate}
    \item $N\unlhd G$;

    \item $J\subseteq H_e$ is a $G$-invariant quasi-Hopf ideal;

    \item $L\subseteq H_N/J^N$ is a quasi-Hopf ideal;

    \item $L$ has trivial intersection with the neutral component $H_e/J$ of $H_N/J^N$, that is, $L\cap (H_e/J)=0$;

    \item the image of $L$ under the grading map $\pi_{N,J}$ of $H_N/J^N$ contains the augmentation ideal of $\kk[N]$, that is, $\aug{N}\subseteq\pi_{N,J}(L)$;

    \item $L$ is $G/N$-invariant for the strong $G/N$-grading of $H/J^G$, that is, $q_J(H_{\bar g})L=Lq_J(H_{\bar g})$ for all $\bar g\in G/N$.
\end{enumerate}
\end{definition}

\begin{remark}\label{rem:RedViaPairs}
Condition (v) is the counterpart, for reduced data, of condition (iv) of Definition~\ref{def:AdmPair}, and Remark~\ref{rem:AugCond} applies here unchanged. Indeed, since $L$ is a quasi-Hopf ideal and $\pi_{N,J}$ is counital, the inclusion $\pi_{N,J}(L)\subseteq\aug{N}$ always holds. Hence condition (v) is equivalent to $\pi_{N,J}(L)=\aug{N}$. Conditions (iii) and (vi) make $(\{\bar e\},L)$ admissible for $H/J^G$, while (iii) and (v) make $(N,L)$ admissible for $H_N/J^N$. Condition (iv) ensures that the inverse image of $L$ in $H_N$ meets $H_e$ precisely in $J$.
\end{remark}

The next result shows that this repackaging loses no information.

\begin{proposition}\label{prop:RedData}
The assignments
\[
(N,K)\longmapsto \left(N,\,J,\,K/J^N\right)
\qquad\text{and}\qquad
(N,J,L)\longmapsto (N,K_L),
\]
where $J:=K\cap H_e$ in the first assignment and, in the second, $K_L\subseteq H_N$ denotes the inverse image of $L$ under the quotient map $H_N\to H_N/J^N$, define a bijective correspondence between the set of admissible pairs for $H$ and the set of reduced admissible data for $H$.
\end{proposition}

\begin{proof}
Let $(N,K)$ be an admissible pair and set $J:=K\cap H_e$. Applied to $I=K^{G/N}$, Theorem~\ref{thm:MainA} and Corollary~\ref{cor:InvNeutral} show that $J$ is a $G$-invariant quasi-Hopf ideal.

Since $J^N\subseteq K$, the canonical projection $H_N\to H_N/K$ factors through $H_N/J^N$, with quasi-Hopf kernel $L:=K/J^N$.

If the class of $x\in K$ belongs to $H_e/J$, write $x=b+y$ with $b\in H_e$ and $y\in J^N$. Then $b\in K\cap H_e=J$; hence $x\in J^N$ and $L\cap(H_e/J)=0$. Moreover, $L$ is the image of $K$ under $H_N\to H_N/J^N$ and $\pi_{N,J}$ is induced by the grading map of $H_N$ (Lemma~\ref{lem:RedSetup}(ii)). Hence $\pi_{N,J}(L)=\pi(K)\supseteq\aug{N}$, and applying $q_J$ to $H_{\bar g}K=KH_{\bar g}$ gives $q_J(H_{\bar g})L=Lq_J(H_{\bar g})$. Thus $(N,J,L)$ is reduced admissible.

Conversely, let $(N,J,L)$ be a reduced admissible datum, and let $K:=K_L\subseteq H_N$ be the inverse image of $L$ under the quotient map $H_N\to H_N/J^N$. Then $K$ is the kernel of the composite morphism of quasi-Hopf algebras $H_N\to H_N/J^N\to (H_N/J^N)/L$, hence a quasi-Hopf ideal of $H_N$ containing $J^N$. Moreover, $K\cap H_e=J$. Indeed, $J\subseteq J^N\subseteq K$ and $J\subseteq H_e$ give one inclusion. Conversely, if $x\in K\cap H_e$, then its class in $H_N/J^N$ lies in $L\cap (H_e/J)=0$, whence $x\in J^N\cap H_e=J$ by Corollary~\ref{cor:RelLift}(i). Since $H_N\to H_N/J^N$ is surjective and $K$ is the full inverse image of $L$, the image of $K$ in $H_N/J^N$ is $L$. Condition (v) gives $\aug{N}\subseteq\pi_{N,J}(L)=\pi(K)$, which is condition (iv) of Definition~\ref{def:AdmPair}. It remains to prove the $G/N$-invariance of $K$. We prove $H_{\bar g}K\subseteq KH_{\bar g}$ for every $\bar g\in G/N$. Applying $q_J$, which carries $K$ onto $L$, and using condition (vi) gives $q_J(H_{\bar g}K)=q_J(H_{\bar g})L=Lq_J(H_{\bar g})=q_J(KH_{\bar g})$. Both $H_{\bar g}K$ and $KH_{\bar g}$ are contained in $H_{\bar g}$, since $K\subseteq H_N$. Then every element of $H_{\bar g}K$ differs from an element of $KH_{\bar g}$ by an element of $\ker(q_J)\cap H_{\bar g}=J^G\cap H_{\bar g}$. Since $J^G\cap H_g=H_gJ$ for all $g\in G$ (Proposition~\ref{prop:Lift}(i)), summing over the coset gives $J^G\cap H_{\bar g}=H_{\bar g}J$, which equals $JH_{\bar g}$ by the $G$-invariance of $J$, and $J\subseteq K$ shows that this subspace is contained in $KH_{\bar g}$. Therefore $H_{\bar g}K\subseteq KH_{\bar g}$, and the opposite inclusion is proved in the same way. Thus $(N,K)$ is an admissible pair.

The two assignments are mutually inverse. Starting from $(N,K)$, the inverse image of $L=K/J^N$ in $H_N$ equals $K+J^N=K$, since $J^N\subseteq K$. Starting from $(N,J,L)$, the constructed pair satisfies $K\cap H_e=J$, as shown above, and the image of $K$ in $H_N/J^N$ is $L$ because the quotient map is surjective. Hence the associated reduced datum is again $(N,J,L)$.
\end{proof}

\begin{corollary}\label{cor:RedClass}
The assignments
\[
I\longmapsto \left(N,\,J,\,(I\cap H_N)/J^N\right)
\qquad\text{and}\qquad
(N,J,L)\longmapsto \sum_{\bar g\in G/N}H_{\bar g}K_L,
\]
where $N:=N_I$ and $J:=I\cap H_e$ in the first assignment, and, in the second, $K_L\subseteq H_N$ denotes the inverse image of $L$ under the quotient map $H_N\to H_N/J^N$, define a bijective correspondence between the set of quasi-Hopf ideals of $H$ and the set of reduced admissible data for $H$.
\end{corollary}

\begin{proof}
This is the composite of the bijections of Theorem~\ref{thm:MainA} and Proposition~\ref{prop:RedData}. Since $H_e\subseteq H_N$, we have $(I\cap H_N)\cap H_e=I\cap H_e$; hence the middle component of the composite is the one displayed.
\end{proof}

\subsection{Transversal ideals and retractions}\label{subsec:Retr}

Fix $N\unlhd G$ and a $G$-invariant quasi-Hopf ideal $J\subseteq H_e$. Then $H_N/J^N$ is strongly $N$-graded with neutral component $H_e/J$ (Corollary~\ref{cor:RelLift}(ii)) and grading map $\pi_{N,J}$ (Lemma~\ref{lem:RedSetup}(ii)). Now, we single out the ideals of $H_N/J^N$ that complement $H_e/J$, and identify them with the retractions onto it.

\begin{definition}[Transversal ideal]\label{def:TransvIdeal}
A quasi-Hopf ideal $L\subseteq H_N/J^N$ is called \emph{transversal to $H_e/J$} if the natural map $H_e/J\to (H_N/J^N)/L$ is an isomorphism of quasi-Hopf algebras.
\end{definition}

\begin{lemma}\label{lem:TransvChar}
A quasi-Hopf ideal $L\subseteq H_N/J^N$ is transversal to $H_e/J$ if and only if $L\cap (H_e/J)=0$ and $H_N/J^N=(H_e/J)+L$.
\end{lemma}

\begin{proof}
The natural map $H_e/J\to (H_N/J^N)/L$ is the composite of the inclusion $H_e/J\hookrightarrow H_N/J^N$ with the projection $H_N/J^N\to (H_N/J^N)/L$, hence a morphism of quasi-Hopf algebras with kernel $L\cap (H_e/J)$ and image $((H_e/J)+L)/L$. It is therefore injective if and only if $L\cap (H_e/J)=0$, and surjective if and only if $H_N/J^N=(H_e/J)+L$.
\end{proof}

The terminology is borrowed from the categorical setting, where Nikshych calls a fusion subcategory transversal to a connected algebra~\cite[Def.~3.1]{Nik19}. The analogy is with the characterization just proved. In both settings a distinguished subobject is required to map isomorphically onto the quotient, thereby allowing the quotient to be computed inside the subobject.

The next condition is weaker than transversality, and it can be stated for an arbitrary faithfully graded quasi-Hopf algebra, over an arbitrary group.

\begin{definition}[Augmented ideal]\label{def:AugIdeal}
A quasi-Hopf ideal $L$ of a quasi-Hopf algebra $A$ faithfully graded by an arbitrary group $\Gamma$ is called \emph{augmented} if its grading map sends $L$ onto $\aug{\Gamma}$.
\end{definition}

An augmented ideal is thus an ideal of $A$, not to be confused with the augmentation ideal $\aug{\Gamma}$ of $\kk[\Gamma]$ onto which it maps.

By Remark~\ref{rem:RedViaPairs}, condition (v) of Definition~\ref{def:RedAdmDatum} amounts to $L$ being augmented.

\begin{lemma}\label{lem:TransvAug}
If $L\subseteq H_N/J^N$ is transversal to $H_e/J$, then $L$ is augmented.
\end{lemma}

\begin{proof}
Since $L$ is a quasi-Hopf ideal, $\varepsilon(L)=0$. As $\pi_{N,J}$ is a morphism of quasi-Hopf algebras, it is counital; hence $\pi_{N,J}(L)\subseteq\aug{N}$. We prove the reverse inclusion. Let $n\in N$. The component of degree $n$ of $H_N/J^N$ is nonzero, since this grading is strong by Corollary~\ref{cor:RelLift}(ii). By Proposition~\ref{prop:GradProp}(i) there is a homogeneous element $a_n$ of degree $n$ with $\varepsilon(a_n)=1$. Since $L$ is transversal, $H_N/J^N=(H_e/J)+L$. Write $a_n=b_n+\ell_n$ with $b_n\in H_e/J$ and $\ell_n\in L$. Since $\varepsilon(\ell_n)=0$, applying the counit to $a_n=b_n+\ell_n$ gives $\varepsilon(b_n)=\varepsilon(a_n)=1$. By Lemma~\ref{lem:RedSetup}(ii), $\pi_{N,J}(a_n)=\zeta_n$ and $\pi_{N,J}(b_n)=\zeta_e$, hence $\pi_{N,J}(\ell_n)=\zeta_n-\zeta_e$. As $n\in N$ was arbitrary and these elements span $\aug{N}$, we obtain $\aug{N}\subseteq\pi_{N,J}(L)$, whence $\pi_{N,J}(L)=\aug{N}$.
\end{proof}

In the following, by a \emph{retraction} of a faithfully graded quasi-Hopf algebra onto its neutral component we mean a morphism of quasi-Hopf algebras restricting to the identity there, and we write $\tau\colon H_N/J^N\to H_e/J$ for the case at hand. To a transversal ideal $L$ we associate the composite $\tau_L\colon H_N/J^N\to (H_N/J^N)/L\xrightarrow{\ \sim\ }H_e/J$, whose second map is the inverse of the isomorphism $H_e/J\cong (H_N/J^N)/L$. The transversal ideals are precisely the kernels of such retractions.

\begin{proposition}\label{prop:TransvRetr}
The assignments
\[
\tau\longmapsto \ker(\tau)
\qquad\text{and}\qquad
L\longmapsto \tau_L
\]
define a bijective correspondence between the set of quasi-Hopf algebra retractions of $H_N/J^N$ onto $H_e/J$ and the set of quasi-Hopf ideals $L\subseteq H_N/J^N$ transversal to $H_e/J$.
\end{proposition}

\begin{proof}
Let $\tau\colon H_N/J^N\to H_e/J$ be a quasi-Hopf algebra retraction and put $L:=\ker(\tau)$, a quasi-Hopf ideal of $H_N/J^N$. Since $\tau|_{H_e/J}=\id$, we have $L\cap (H_e/J)=0$. Moreover, every $a\in H_N/J^N$ decomposes as $a=\tau(a)+\left(a-\tau(a)\right)$ with $\tau(a)\in H_e/J$ and $a-\tau(a)\in L$, whence $H_N/J^N=(H_e/J)+L$. Thus $L$ is transversal to $H_e/J$ by Lemma~\ref{lem:TransvChar}.

Conversely, let $L\subseteq H_N/J^N$ be transversal to $H_e/J$ and let $q\colon H_N/J^N\to (H_N/J^N)/L$ be the quotient morphism. By transversality, $q|_{H_e/J}\colon H_e/J\to (H_N/J^N)/L$ is an isomorphism of quasi-Hopf algebras. Then $\tau_L:=(q|_{H_e/J})^{-1}q\colon H_N/J^N\to H_e/J$ is a quasi-Hopf algebra retraction with $\ker(\tau_L)=\ker(q)=L$. Finally, if $\tau$ is any quasi-Hopf algebra retraction with $\ker(\tau)=L$, then $\tau$ and $\tau_L$ agree on $H_e/J$, where both restrict to the identity, and on $L$, where both vanish. Since both maps are $\kk$-linear and $H_N/J^N=(H_e/J)+L$, they agree on all of $H_N/J^N$. Thus the two constructions are mutually inverse.
\end{proof}

Transversality is not easy to verify directly, since the spanning condition involves all of $H_N/J^N$. The results that follow show that augmentation and trivial intersection already force transversality. As in Definition~\ref{def:AugIdeal}, none of them is specific to $H_N/J^N$, and none assumes a crossed product. The first assumes neither faithfulness nor trivial intersection.

\begin{proposition}\label{prop:AugSurj}
Let $\Gamma$ be a group, let $A=\bigoplus_{g\in\Gamma}A_g$ be a $\Gamma$-graded quasi-Hopf algebra, and let $P_g\colon A\to A_g$ denote the projections attached to this decomposition. If $L\subseteq A$ is a quasi-Hopf ideal with $\varepsilon\left(P_h(L)\right)\neq0$ for some $h\in\Gamma$, then $P_g(L)=A_g$ for every $g\in\Gamma$.
\end{proposition}

\begin{proof}
Fix $\ell\in L$ with $\varepsilon\left(P_h(\ell)\right)\neq0$. Then $A_h\neq0$, hence $A_{h^{-1}}\neq0$ by Proposition~\ref{prop:FaithStrong}(ii), and Proposition~\ref{prop:GradProp}(i) provides $a\in A_{h^{-1}}$ with $\varepsilon(a)=1$. Since $P_e(a\ell)=aP_h(\ell)$ and $\varepsilon$ is multiplicative, the element $a\ell\in L$ satisfies $\varepsilon\left(P_e(a\ell)\right)\neq0$. Therefore the subspace $T:=P_e(L)$ of $A_e$ has $\varepsilon(T)\neq0$, and it is a two-sided ideal of $A_e$ because $A$ is $\Gamma$-graded as an algebra.

By~\eqref{eq:Grad1} we have $\left(P_e\otimes P_e\right)\Delta=\Delta P_e$, and by~\eqref{eq:Grad2} we have $P_eS=SP_e$. Applying the first identity to $\Delta(L)\subseteq L\otimes A+A\otimes L$ gives $\Delta(T)\subseteq T\otimes A_e+A_e\otimes T$, and the second gives $S(T)\subseteq T$. Choose $x\in T$ with $\varepsilon(x)=1$ and write $\Delta(x)=\sum_ix_i\otimes b_i+\sum_jd_j\otimes y_j$ with $x_i,y_j\in T$ and $b_i,d_j\in A_e$. Since $\alpha\in A_e$ by~\eqref{eq:Grad2}, the first identity of~\eqref{eq:QH6} gives
\[
\alpha=\varepsilon(x)\,\alpha=S(x')\,\alpha\,x''=\sum_iS(x_i)\,\alpha\,b_i+\sum_jS(d_j)\,\alpha\,y_j\in T,
\]
using $S(T)\subseteq T$, $S(A_e)\subseteq A_e$, and that $T$ is a two-sided ideal of $A_e$. As $\Phi\in A_e^{\otimes3}$ and $\beta\in A_e$ by~\eqref{eq:Grad2}, the first identity of~\eqref{eq:QH7} then yields $1=\Phi^{(1)}\,\beta\,S\left(\Phi^{(2)}\right)\,\alpha\,\Phi^{(3)}\in A_e\,\alpha\,A_e\subseteq T$, and hence $T=A_e$.

Now let $g\in\Gamma$ and $x\in A_g$. Choosing $\ell\in L$ with $P_e(\ell)=1$, we get $x\ell\in L$ and $P_g(x\ell)=xP_e(\ell)=x$, whence $A_g\subseteq P_g(L)$ and $P_g(L)=A_g$.
\end{proof}

The two results below rest on a pair of operators attached to a linear functional that annihilates the ideal. We record their properties once.

\begin{lemma}\label{lem:CoidealOp}
Let $\Gamma$ be a group, let $A$ be a $\Gamma$-graded quasi-Hopf algebra, and let $L\subseteq A$ be a quasi-Hopf ideal. For a linear functional $\phi\colon A\to\kk$ vanishing on $L$, put $E_\phi:=\left(\phi\otimes\id\right)\Delta$ and $D_\phi:=\left(\id\otimes\phi\right)\Delta$. Then:
\begin{enumerate}
    \item $E_\phi(L)\subseteq L$ and $D_\phi(L)\subseteq L$.

    \item Both maps send $A_h$ into $A_h$ for every $h\in\Gamma$, and both vanish on $A_h$ whenever $\phi$ does.

    \item $\varepsilon E_\phi=\varepsilon D_\phi=\phi$.
\end{enumerate}
\end{lemma}

\begin{proof}
Let $\ell\in L$ and write $\Delta(\ell)=\sum_i\ell_i\otimes a_i+\sum_jd_j\otimes\ell_j$ with $\ell_i,\ell_j\in L$, as permitted by Definition~\ref{def:QHIdeal}. Then $E_\phi(\ell)=\sum_i\phi(\ell_i)a_i+\sum_j\phi(d_j)\ell_j$, where the first sum vanishes because $\phi$ does on $L$ and the second lies in $L$. The same computation on the other tensor factor gives $D_\phi(\ell)\in L$, which proves (i). For (ii), the inclusion $\Delta(A_h)\subseteq A_h\otimes A_h$ of~\eqref{eq:Grad1} places both legs in $A_h$; hence each operator maps $A_h$ into itself and returns zero as soon as $\phi$ annihilates $A_h$. For (iii), the counit axiom~\eqref{eq:QH4} gives
\[
\varepsilon\left(E_\phi(x)\right)=\phi(x')\varepsilon(x'')=\phi\left(x'\varepsilon(x'')\right)=\phi(x),
\]
and likewise for $D_\phi$.
\end{proof}

Even if every component projection is surjective, it does not follow that $A=A_e+L$, since an element of $L$ with prescribed component of degree $g$ has further components, which need not lie in $A_e+L$. The next result proceeds one degree at a time, at the cost of the hypothesis $L\cap A_e=0$.

\begin{proposition}\label{prop:TransvDeg}
Let $\Gamma$ be a group, let $A$ be a faithfully $\Gamma$-graded quasi-Hopf algebra, and let $L\subseteq A$ be a quasi-Hopf ideal with $L\cap A_e=0$. If $g\in\Gamma$ is such that $\varepsilon$ does not vanish on $A_g\cap\left(A_e+L\right)$, then $A_g\subseteq A_e+L$.
\end{proposition}

\begin{proof}
By Proposition~\ref{prop:FaithStrong}(iii), the grading of $A$ is strong. Write $C:=A_e+L$ and $T:=A_g\cap C$. Since $A_e$ is a subalgebra of $A$ (Proposition~\ref{prop:GradProp}(ii)) and $L$ is a two-sided ideal, $C$ is a subalgebra of $A$, and $T$ is an $A_e$-sub-bimodule of $A_g$. We first record that $L\cap A_g=0$. Indeed, if $x\in L\cap A_g$, then $A_{g^{-1}}x\subseteq L\cap A_e=0$; hence $A_ex=(A_gA_{g^{-1}})x=0$ and $x=0$.

Next we show that $\Delta(T)\subseteq T\otimes A_g$. Let $x\in T$ and write $\Delta(x)=\sum_ix_i\otimes y_i$ with $x_i,y_i\in A_g$ and the $y_i$ linearly independent. Let $\phi\colon A\to\kk$ be a linear functional vanishing on $C$, and hence on $L$ and on $A_e$, and set $E_\phi:=\left(\phi\otimes\id\right)\Delta$. By Lemma~\ref{lem:CoidealOp} we have $E_\phi(L)\subseteq L$, and $E_\phi$ both preserves $A_g$ and vanishes on $A_e$. Writing $x=b+\ell$ with $b\in A_e$ and $\ell\in L$, we get $E_\phi(x)=E_\phi(\ell)\in L\cap A_g=0$, whence $\sum_i\phi(x_i)y_i=0$ and $\phi(x_i)=0$ for every $i$. This holds for every $\phi$ vanishing on $C$, and every subspace of a vector space is the intersection of the kernels of the linear functionals vanishing on it. Therefore each $x_i$ lies in $A_g\cap C=T$.

Choose $x\in T$ with $\varepsilon(x)=1$. By~\eqref{eq:Grad1} we have $x''\in A_g$, hence $S(x'')\in A_{g^{-1}}$ by~\eqref{eq:Grad2}, while $\beta\in A_e$ and $x'\in T$ by the previous paragraph. The second identity of~\eqref{eq:QH6} therefore gives
\[
\beta=\varepsilon(x)\,\beta=x'\,\beta\,S(x'')\in TA_eA_{g^{-1}}\subseteq W:=TA_{g^{-1}}.
\]
The subspace $W$ is contained in $A_gA_{g^{-1}}=A_e$, and it is a two-sided ideal of $A_e$, since $A_eW=(A_eT)A_{g^{-1}}\subseteq W$ and $WA_e=T(A_{g^{-1}}A_e)\subseteq W$. As $\Phi\in A_e^{\otimes3}$, $\alpha\in A_e$ and $S(A_e)\subseteq A_e$ by~\eqref{eq:Grad2}, the first identity of~\eqref{eq:QH7} yields
\[
1=\Phi^{(1)}\,\beta\,S(\Phi^{(2)})\,\alpha\,\Phi^{(3)}\in A_e\,\beta\,A_e\subseteq W.
\]
Hence $W=A_e$, and strong grading gives $A_g=A_eA_g=T(A_{g^{-1}}A_g)=TA_e\subseteq T$. Therefore $A_g=T\subseteq A_e+L$.
\end{proof}

It remains to produce such an element in every degree. This is achieved by an elementary reduction of the support of a single element of $L$, supplied by the augmentation.

\begin{theorem}\label{thm:AugTransv}
Let $\Gamma$ be a group, let $A$ be a faithfully $\Gamma$-graded quasi-Hopf algebra, and let $L\subseteq A$ be an augmented quasi-Hopf ideal with $L\cap A_e=0$. Then $A=A_e+L$.
\end{theorem}

\begin{proof}
Write $C:=A_e+L$ and $\Sigma:=\{h\in\Gamma:A_h\subseteq C\}$, a subset containing $e$, and let $P_h\colon A\to A_h$ denote the projections attached to the decomposition of $A$. For $\ell\in L$, let $F(\ell)$ be its support, the finite set of degrees in which $\ell$ has a nonzero component, and put $\nu(\ell):=\sum_{h\in\Sigma}\varepsilon(P_h(\ell))$. Since $\varepsilon(\ell)=0$,
\[
\nu(\ell)=-\sum_{h\notin\Sigma}\varepsilon(P_h(\ell)).
\]
We claim that no pair $(\ell,g)$ satisfies $\ell\in L$, $g\notin\Sigma$, $\varepsilon(P_g(\ell))\neq0$ and $\nu(\ell)\neq0$. Suppose otherwise and choose such a pair with $|F(\ell)|$ minimal. Write $b_h:=P_h(\ell)$, $F:=F(\ell)$ and $F'':=F\setminus(\Sigma\cup\{g\})$, and note that $g\in F$.

Suppose first that $b_g\notin V:=L+\sum_{h\in F''}A_h$, and choose a linear functional $\phi\colon A\to\kk$ vanishing on $V$ with $\phi(b_g)\neq0$. Set $D_\phi:=(\id\otimes\phi)\Delta$. Since $\phi$ vanishes on $L$ and on $A_h$ for every $h\in F''$, Lemma~\ref{lem:CoidealOp} gives $D_\phi(L)\subseteq L$, $D_\phi(A_h)\subseteq A_h$ for every $h\in\Gamma$, $D_\phi(b_h)=0$ for $h\in F''$, and $\varepsilon D_\phi=\phi$. Hence $x:=D_\phi(\ell)\in L$ has support contained in $(F\cap\Sigma)\cup\{g\}$. Consequently, $P_g(x)=x-\sum_{h\in F\cap\Sigma}P_h(x)$ lies in $A_g\cap C$, since $x\in L$ and $P_h(x)\in A_h\subseteq C$ for every $h\in\Sigma$. Its counit is $\varepsilon(D_\phi(b_g))=\phi(b_g)\neq0$. Proposition~\ref{prop:TransvDeg} then gives $A_g\subseteq C$, contradicting $g\notin\Sigma$.

Suppose instead that $b_g=\lambda+\sum_{h\in F''}c_h$ with $\lambda\in L$ and $c_h\in A_h$. Comparing components of degree $g$, of degree $h\in F''$, and of the remaining degrees, we get $P_g(\lambda)=b_g$, $P_h(\lambda)=-c_h$ for $h\in F''$, and $P_h(\lambda)=0$ otherwise. Hence $\ell'':=\ell-\lambda\in L$ has support contained in $F\setminus\{g\}$, which is strictly smaller than $F$. As $F''$ excludes both $\Sigma$ and $g$, the support of $\lambda$ misses $\Sigma$ entirely; hence $\nu(\lambda)=0$ and $\nu(\ell'')=\nu(\ell)\neq0$. Since $\varepsilon(\ell'')=0$ and $\nu(\ell'')\neq0$, the displayed identity applied to $\ell''$ gives $h_0\notin\Sigma$ with $\varepsilon(P_{h_0}(\ell''))\neq0$, and the pair $(\ell'',h_0)$ contradicts the choice of $(\ell,g)$ with $|F(\ell)|$ minimal. The claim is proved.

Let $g\in\Gamma$ with $g\neq e$, the case $g=e$ being trivial. Since $L$ is augmented, there is $\ell\in L$ whose image under the grading map of $A$ is $\zeta_g-\zeta_e$. By~\eqref{eq:GradMap} this means $\varepsilon(P_g(\ell))=1$ and $\varepsilon(P_e(\ell))=-1$, the remaining components having zero counit. If $g\notin\Sigma$, then $\nu(\ell)=-1$, since $e\in\Sigma$, and the pair $(\ell,g)$ contradicts the claim. Therefore $A_g\subseteq C$ for every $g\in\Gamma$, and $A=A_e+L$.
\end{proof}

Recall that $H_N/J^N$ is strongly $N$-graded with neutral component $H_e/J$ and grading map $\pi_{N,J}$. Applied to it, the theorem completes the characterization of transversality begun in Lemmas~\ref{lem:TransvChar} and~\ref{lem:TransvAug}.

\begin{corollary}\label{cor:TransvEquiv}
A quasi-Hopf ideal $L\subseteq H_N/J^N$ is transversal to $H_e/J$ if and only if $L\cap(H_e/J)=0$ and $L$ is augmented.
\end{corollary}

\begin{proof}
If $L$ is transversal, then $L\cap(H_e/J)=0$ by Lemma~\ref{lem:TransvChar} and $L$ is augmented by Lemma~\ref{lem:TransvAug}. Conversely, Theorem~\ref{thm:AugTransv} gives $H_N/J^N=(H_e/J)+L$, and Lemma~\ref{lem:TransvChar} applies.
\end{proof}

The antipode axioms are used in Propositions~\ref{prop:AugSurj} and~\ref{prop:TransvDeg}, and they cannot be dispensed with, as the following bialgebra shows.

\begin{example}\label{ex:AntipodeNec}
Let $\kk[z]$ be the polynomial bialgebra on one generator $z$ with $\Delta(z)=z\otimes z$ and $\varepsilon(z)=1$, let $\Gamma\cong\mathbb{Z}/2\mathbb{Z}$ be written multiplicatively as $\{e,t\}$, and let $A:=\kk[z]\otimes\kk[\Gamma]$, graded by the second tensor factor: a $\Gamma$-crossed product bialgebra with $A_e=\kk[z]$ and homogeneous unit $u:=1\otimes\zeta_t$, satisfying $u^2=1$. Its grading satisfies~\eqref{eq:Grad1}, since $\Delta(z^ku^i)=z^ku^i\otimes z^ku^i$ has both legs of degree $t^i$. The subspace $L:=z\kk[z](u-1)$ is a two-sided ideal, as $A$ is commutative and $u(u-1)=-(u-1)$, and it satisfies every hypothesis of Theorem~\ref{thm:AugTransv} that does not involve the antipode: it is a coideal with $\varepsilon(L)=0$, since
\[
\Delta(z^k(u-1))=z^ku\otimes z^k(u-1)+z^k(u-1)\otimes z^k,
\qquad k\geq1.
\]
It intersects $A_e$ trivially, since $c(u-1)\in A_e$ forces $cu=0$. The grading is strong, since $u$ is a unit. Its image under the map $x_g\mapsto\varepsilon(x_g)\zeta_g$ of~\eqref{eq:GradMap} is $\aug{\Gamma}$, since each $z^k(u-1)$ is sent to $\zeta_t-\zeta_e$. Nevertheless $A\neq A_e+L$, because the degree-$t$ component of $A_e+L$ is $z\kk[z]u\subsetneq A_eu$. Consistently, $A$ admits no quasi-Hopf algebra structure, whatever the reassociator: given a candidate $(\Phi,S,\alpha,\beta)$, commutativity and~\eqref{eq:QH7} force $\alpha\in A^\times$, and the first identity of~\eqref{eq:QH6} applied to the group-like element $z$ then yields $S(z)z=1$, which is impossible, as the algebra morphism $A\to\kk[z]$ sending $u$ to $1$ would exhibit $z$ as invertible in $\kk[z]$.
\end{example}

A retraction does not merely split the neutral component off. It trivializes the grading entirely.

\begin{theorem}\label{thm:Trivialization}
Let $\Gamma$ be a group and let $A$ be a faithfully $\Gamma$-graded quasi-Hopf algebra. Then the following are equivalent:
\begin{enumerate}
    \item There is a morphism of quasi-Hopf algebras $\tau\colon A\to A_e$ with $\tau|_{A_e}=\id$.
    \item There is an augmented quasi-Hopf ideal $L\subseteq A$ with $L\cap A_e=0$.
    \item There are group-like units $u_g\in A_g$, for $g\in\Gamma$, such that $u_gu_h=u_{gh}$ and $u_gb=bu_g$ for every $b\in A_e$.
\end{enumerate}
In that case $L=\ker(\tau)$ and $\tau(u_g)=1$ determine the three data from one another, and $b\otimes\zeta_g\mapsto bu_g$ is an isomorphism of $\Gamma$-graded quasi-Hopf algebras from $A_e\otimes\kk[\Gamma]$ onto $A$.
\end{theorem}

\begin{proof}
By Proposition~\ref{prop:FaithStrong}(iii), the grading of $A$ is strong; we use this throughout the proof.

(ii) $\Rightarrow$ (i). Theorem~\ref{thm:AugTransv} gives $A=A_e+L$; hence the composite $A_e\hookrightarrow A\to A/L$ is a surjective morphism of quasi-Hopf algebras, and it is injective because $L\cap A_e=0$. Composing its inverse with the projection $A\to A/L$ gives a morphism $\tau$ as in (i), with $\ker(\tau)=L$. Such a retraction is unique, since two retractions with kernel $L$ agree on $A_e$ and vanish on $L$, and $A=A_e+L$.

(i) $\Rightarrow$ (iii). Fix $g\in\Gamma$. Since $\tau$ is an algebra morphism fixing $A_e$, its restriction $\tau|_{A_g}\colon A_g\to A_e$ is a morphism of $A_e$-bimodules. It is injective, because $\tau(x)=0$ for $x\in A_g$ gives $A_{g^{-1}}x\subseteq\ker(\tau)\cap A_e=0$ and hence $A_ex=(A_gA_{g^{-1}})x=0$. It is surjective, because $\tau(A_g)$ is a right $A_e$-submodule of $A_e$ and $A_e=\tau(A_e)=\tau(A_gA_{g^{-1}})=\tau(A_g)\tau(A_{g^{-1}})\subseteq\tau(A_g)$. Put $u_g:=(\tau|_{A_g})^{-1}(1)$.

For $b\in A_e$ the elements $u_gb$ and $bu_g$ both lie in $A_g$ and are sent to $b$ by $\tau$; hence they agree. Likewise $u_gu_h$ and $u_{gh}$ lie in $A_{gh}$ and are sent to $1$; hence they agree, and $u_e=1$ because $\tau(1)=1$. In particular $u_gu_{g^{-1}}=1$; therefore $u_g$ is a unit, and $A_g=A_eu_g$ because $x=\tau(x)u_g$ for every $x\in A_g$.

By~\eqref{eq:Grad1} the element $\Delta(u_g)$ lies in $A_g\otimes A_g$, which equals $(A_e\otimes A_e)(u_g\otimes u_g)$ because $A_g=A_eu_g$. Write $\Delta(u_g)=\sum_i(a_i\otimes a_i')(u_g\otimes u_g)$ with $a_i,a_i'\in A_e$. Applying $\tau\otimes\tau$ and using $\Delta\tau=(\tau\otimes\tau)\Delta$ gives $\sum_ia_i\otimes a_i'=\Delta(1)=1\otimes1$; hence $u_g$ is group-like. This proves (iii).

(iii) $\Rightarrow$ (ii). The two conditions imposed determine three more. First, $u_e^2=u_e$ and $u_e$ is a unit; hence $u_e=1$ and $u_{g^{-1}}=u_g^{-1}$. Second, $\Delta(u_g)=u_g\otimes u_g$ and~\eqref{eq:QH4} give $\varepsilon(u_g)\,u_g=u_g$, whence $\varepsilon(u_g)=1$. Third, put $t:=S(u_g)\,u_g$, which lies in $A_e$ because $S(u_g)\in A_{g^{-1}}$ by~\eqref{eq:Grad2}. The first identity of~\eqref{eq:QH6} applied to $u_g$ reads $S(u_g)\,\alpha\,u_g=\alpha$, and $\alpha\in A_e$ commutes with $u_g$; hence $t\alpha=\alpha$. Applying $S$ to $u_gb=bu_g$ shows that $S(u_g)$ commutes with $S(A_e)$, as does $t$. Since $A_e$ is a quasi-Hopf subalgebra by Proposition~\ref{prop:GradProp}(ii), the components of $\Phi^{-1}$ lie in $A_e$; hence left multiplication of the second identity of~\eqref{eq:QH7} by $t$ yields $t=1$. Therefore $S(u_g)=u_{g^{-1}}$.

For $x\in A_g$ we have $xu_{g^{-1}}\in A_e$ and $x=(xu_{g^{-1}})u_g$; hence $A_g=A_eu_g$ and $\tau(\sum_gb_gu_g):=\sum_gb_g$ is a well-defined linear map fixing $A_e$. It is multiplicative because $(bu_g)(cu_h)=bcu_{gh}$, comultiplicative and counital because each $u_g$ is group-like with $\varepsilon(u_g)=1$, and it commutes with $S$ because $S(bu_g)=S(u_g)S(b)=S(b)u_{g^{-1}}$. As $\Phi$, $\alpha$ and $\beta$ lie in $A_e$ by~\eqref{eq:Grad2}, the map $\tau$ fixes them and is therefore a morphism of quasi-Hopf algebras. Its kernel $L$ meets $A_e$ trivially and contains each $u_g-1$, whose image under the map $x_g\mapsto\varepsilon(x_g)\zeta_g$ of~\eqref{eq:GradMap} is $\zeta_g-\zeta_e$. These elements span $\aug{\Gamma}$, while the image of $L$ is contained in $\aug{\Gamma}$ because $\varepsilon(\ell)=\varepsilon(\tau(\ell))=0$ for every $\ell\in L$. Hence $L$ is augmented and (ii) holds.

For the last assertion, the assignment $b\otimes\zeta_g\mapsto bu_g$ respects the products, the coproducts, the counits and the antipodes by the relations in (iii) and the three identities just derived from them, and it carries the reassociator and the distinguished elements of $A_e\otimes\kk[\Gamma]$ to $\Phi$, $\alpha$ and $\beta$, which lie in $A_e$ by~\eqref{eq:Grad2}. It is bijective because $A_g=A_eu_g$ with $u_g$ a unit.
\end{proof}

A family as in part~(iii) is called a \emph{group-like splitting} of the grading of $A$.

\subsection{Classification by retractions}\label{subsec:SplitClass}

For a fixed quasi-Hopf ideal $I\subseteq H$, in this subsection we write $N=N_I$, $K=I\cap H_N$, and $J=I\cap H_e$. By Corollary~\ref{cor:InvNeutral}, $J$ is $G$-invariant, and $J^N\subseteq K$. Hence $K/J^N$ is defined.

\begin{definition}[Split ideal]\label{def:SplitIdeal}
$I$ is called \emph{split} if the natural map $H_e/J\to H_N/K$ is an isomorphism of quasi-Hopf algebras.
\end{definition}

The following results show that identifying split ideals with transversal ones turns Corollary~\ref{cor:RedClass} into a classification by retractions.

\begin{lemma}\label{lem:SplitTransv}
The ideal $I$ is split if and only if $L:=K/J^N\subseteq H_N/J^N$ is transversal to $H_e/J$.
\end{lemma}

\begin{proof}
Since $J^N\subseteq K$, the canonical projection $H_N\to H_N/K$ factors through a surjective morphism of quasi-Hopf algebras $H_N/J^N\to H_N/K$ with kernel $L=K/J^N$. The induced isomorphism $(H_N/J^N)/L\xrightarrow{\ \sim\ }H_N/K$ carries the natural map $H_e/J\to (H_N/J^N)/L$ of Definition~\ref{def:TransvIdeal} to the natural map $H_e/J\to H_N/K$ of Definition~\ref{def:SplitIdeal}, since both are induced by the inclusion $H_e\hookrightarrow H_N$. Hence one is an isomorphism if and only if the other is.
\end{proof}

\begin{corollary}\label{cor:SplitAll}
Every quasi-Hopf ideal of $H$ is split.
\end{corollary}

\begin{proof}
By Corollary~\ref{cor:RedClass}, the triple $(N,J,L)$ with $L:=K/J^N$ is a reduced admissible datum. In particular $L\cap(H_e/J)=0$, and $L$ is augmented because $\pi_{N,J}(L)=\aug{N}$ by Remark~\ref{rem:RedViaPairs}. Hence $L$ is transversal to $H_e/J$ by Corollary~\ref{cor:TransvEquiv}, and $I$ is split by Lemma~\ref{lem:SplitTransv}.
\end{proof}

The property is thus automatic, but it is the form in which the retraction is produced, and the classification below is stated in those terms.

\begin{corollary}\label{cor:LiftTensor}
Let $I$ be a quasi-Hopf ideal of $H$, with $N=N_I$ and $J=I\cap H_e$. Then the lifted quotient $H_N/J^N$ admits a group-like splitting $(u_n)_{n\in N}$. In particular it is isomorphic to $(H_e/J)\otimes\kk[N]$ as an $N$-graded quasi-Hopf algebra.
\end{corollary}

\begin{proof}
By Corollary~\ref{cor:SplitAll} and Lemma~\ref{lem:SplitTransv}, the ideal $K/J^N$ is transversal to $H_e/J$; hence Proposition~\ref{prop:TransvRetr} provides a quasi-Hopf algebra retraction $H_N/J^N\to H_e/J$. Since $H_N/J^N$ is strongly $N$-graded with neutral component $H_e/J$ by Corollary~\ref{cor:RelLift}(ii), Theorem~\ref{thm:Trivialization} applies.
\end{proof}

The hypothesis that some quasi-Hopf ideal realize the pair $(N,J)$ cannot be replaced by finite dimension. The group algebra $\kk[Q_8]$ of the quaternion group, graded by $Q_8/Z(Q_8)$ with neutral component $\kk[Z(Q_8)]$, is a finite-dimensional strongly graded Hopf algebra and a crossed product. Its group-like elements are the elements of $Q_8$; hence a group-like splitting would split the central extension $Z(Q_8)\to Q_8\to Q_8/Z(Q_8)$ and force $Q_8$ to be abelian.

Together with Proposition~\ref{prop:TransvRetr}, these results show that every quasi-Hopf ideal $I\subseteq H$ determines a normal subgroup $N=N_I$, a $G$-invariant quasi-Hopf ideal $J=I\cap H_e$, and a quasi-Hopf algebra retraction $\tau\colon H_N/J^N\to H_e/J$, which is the splitting that the name records. The following definition axiomatizes the triples that arise in this way.

\begin{definition}[Split admissible datum]\label{def:SplitAdmDatum}
A \emph{split admissible datum} for $H$ is a triple $(N,J,\tau)$ such that:
\begin{enumerate}
    \item $N\unlhd G$;

    \item $J\subseteq H_e$ is a $G$-invariant quasi-Hopf ideal;

    \item $\tau\colon H_N/J^N\to H_e/J$ is a quasi-Hopf algebra retraction;

    \item the kernel of $\tau$ is $G/N$-invariant for the strong $G/N$-grading of $H/J^G$, that is, $q_J(H_{\bar g})\ker(\tau)=\ker(\tau)\,q_J(H_{\bar g})$ for all $\bar g\in G/N$, where $q_J\colon H\to H/J^G$ is the canonical projection and the products are formed in $H/J^G$.
\end{enumerate}
\end{definition}

\begin{remark}\label{rem:SplitViaRed}
For a split admissible datum $(N,J,\tau)$, put $L=\ker\tau$. By Proposition~\ref{prop:TransvRetr}, $L$ is transversal to $H_e/J$; hence it automatically satisfies conditions (iv) and (v) of Definition~\ref{def:RedAdmDatum} (using Lemma~\ref{lem:TransvAug} for the augmentation). Condition (iv) above is precisely condition (vi) there. Thus split admissible data are reduced admissible data with transversal third component, expressed through its retraction.
\end{remark}

\begin{theorem}\label{thm:MainB}
Let $H$ be a faithfully $G$-graded quasi-Hopf algebra. The assignments
\[
I\longmapsto \left(N,\,J,\,\tau\right)
\qquad\text{and}\qquad
(N,J,\tau)\longmapsto \sum_{\bar g\in G/N}H_{\bar g}K_\tau,
\]
where, in the first assignment, $N:=N_I$, $J:=I\cap H_e$, and $\tau$ is the unique quasi-Hopf algebra retraction $\tau\colon H_N/J^N\to H_e/J$ with kernel $(I\cap H_N)/J^N$ and, in the second, $K_\tau\subseteq H_N$ is the inverse image of $\ker(\tau)$ under the quotient map $H_N\to H_N/J^N$, define a bijective correspondence between the set of quasi-Hopf ideals of $H$ and the set of split admissible data $(N,J,\tau)$.
\end{theorem}

\begin{proof}
By Corollary~\ref{cor:RedClass}, the quasi-Hopf ideals of $H$ correspond to the reduced admissible data $(N,J,L)$. Every quasi-Hopf ideal is split by Corollary~\ref{cor:SplitAll}; hence every such $L$ is transversal to $H_e/J$ by Lemma~\ref{lem:SplitTransv}. Proposition~\ref{prop:TransvRetr} replaces $L$ uniquely by the retraction $\tau$ with kernel $L$, and Remark~\ref{rem:SplitViaRed} identifies the remaining conditions with those of Definition~\ref{def:SplitAdmDatum}. This gives the stated mutually inverse assignments.
\end{proof}

\section{The character form}\label{sec:CharForm}

We now describe the retractions of Section~\ref{sec:Explicit} explicitly. The conjugation action of the grading group reduces such a retraction to a single group homomorphism, and over a pair realized by some quasi-Hopf ideal the ideals lying over it form a torsor under those homomorphisms. The last two subsections compute the same data inside a chosen crossed-product presentation, which the lifted quotient of a quasi-Hopf ideal always admits by Corollary~\ref{cor:LiftTensor}, and which is the form the applications use.

\begin{definition}[Central group-like elements]\label{def:CentralGrpLike}
For a quasi-Hopf algebra $H$, let
\[
\Gc{H}:=Z(H)\cap G(H)\cap H^\times
\]
denote the set of invertible central group-like elements of $H$, where
\[
G(H):=\{x\in H\mid \Delta(x)=x\otimes x,\ \varepsilon(x)=1\}
\]
is the set of group-like elements of $H$. Although $G(H)$ need not consist of units in the present generality, the intersection $\Gc{H}$ is an abelian group under multiplication.
\end{definition}

\subsection{Conjugation and the action of the grading group}\label{subsec:Conj}

Let $H$ be a faithfully $G$-graded quasi-Hopf algebra, let $N\unlhd G$, and let $J\subseteq H_e$ be a $G$-invariant quasi-Hopf ideal. Throughout this subsection and the next, we work with the lifted quotient $H_N/J^N$, whose neutral component is $H_e/J$. The quotient $H/J^G$ is strongly $G$-graded, with components $q_J(H_g)$, where $q_J\colon H\to H/J^G$ is the canonical projection. For $g\in G$, call a decomposition
\[
1=\sum_i\bar h_i\bar h_i',
\qquad
\bar h_i\in q_J(H_g),\quad
\bar h_i'\in q_J(H_{g^{-1}})
\]
\emph{adapted} to $g$.

\begin{lemma}\label{lem:ConjComm}
For $g\in G$ and $x\in H/J^G$ commuting with $H_e/J$, set $\conj{g}{x}:=\sum_i\bar h_i\,x\,\bar h_i'$, where $\sum_i\bar h_i\bar h_i'=1$ is any decomposition adapted to $g$. Then $\conj{g}{x}$ is independent of the chosen decomposition. Moreover, $\conj{g}{(xy)}=\conj{g}{x}\,\conj{g}{y}$ for all $x,y\in H/J^G$ commuting with $H_e/J$, and $\conj{g}{1}=1$.
\end{lemma}

\begin{proof}
For independence, let $\sum_i \bar h_i\bar h_i'=1$ and $\sum_j \bar k_j\bar k_j'=1$ be two decompositions adapted to $g$. Each $\bar h_i'\bar k_j$ lies in $q_J(H_{g^{-1}})q_J(H_g)\subseteq H_e/J$ and therefore commutes with $x$. Thus
\[
\sum_i \bar h_i x\bar h_i'
=\sum_{i,j}\bar h_i x(\bar h_i'\bar k_j)\bar k_j'
=\sum_{i,j}\bar h_i(\bar h_i'\bar k_j)x\bar k_j'
=\sum_j \bar k_j x\bar k_j'.
\]
For multiplicativity, fix a decomposition $\sum_i\bar h_i\bar h_i'=1$ adapted to $g$. Each $\bar h_i'\bar h_j$ lies in $H_e/J$ and therefore commutes with $x$; hence
\[
\conj{g}{x}\,\conj{g}{y}
=\sum_{i,j}\bar h_i x(\bar h_i'\bar h_j)y\bar h_j'
=\sum_{i,j}\bar h_i(\bar h_i'\bar h_j)\,xy\,\bar h_j'
=\sum_j\bar h_j\,xy\,\bar h_j'
=\conj{g}{(xy)},
\]
using $\sum_i\bar h_i(\bar h_i'\bar h_j)=\bar h_j$. Finally, $\conj{g}{1}=\sum_i\bar h_i\bar h_i'=1$.
\end{proof}

The same statement and proof apply verbatim in any strongly graded algebra, in particular in the strongly $(G\times G)$-graded algebra $H/J^G\otimes H/J^G$ used below. Conjugation in $H/J^G$ now induces an action of $G$ on $\Gc{H_e/J}$.

\begin{lemma}\label{lem:ConjAction}
For $g\in G$ and $z\in Z(H_e/J)$, set
\[
\conj{g}{z}:=\sum_i \bar h_i\,z\,\bar h_i',
\]
where $\sum_i \bar h_i\bar h_i'=1$ is any finite decomposition with $\bar h_i\in q_J(H_g)$ and $\bar h_i'\in q_J(H_{g^{-1}})$. Then:
\begin{enumerate}
    \item The sum $\conj{g}{z}$ lies in $H_e/J$ and is independent of the chosen decomposition.

    \item For each $g\in G$, the map $z\mapsto\conj{g}{z}$ is an algebra automorphism of $Z(H_e/J)$, and $\conj{e}{z}=z$, $\conj{g}{(\conj{h}{z})}=\conj{gh}{z}$ for all $g,h\in G$. Thus $G$ acts on $Z(H_e/J)$ on the left by algebra automorphisms.

    \item This action restricts to an action of $G$ on the group $\Gc{H_e/J}$ by group automorphisms.
\end{enumerate}
\end{lemma}

\begin{proof}
(i) Since $H/J^G$ is strongly $G$-graded, $1\in H_e/J=q_J(H_g)q_J(H_{g^{-1}})$, hence a decomposition $\sum_i \bar h_i\bar h_i'=1$ as in the statement, adapted to $g$, exists for every $g\in G$. Each product $\bar h_i z\bar h_i'$ then lies in $q_J(H_g)q_J(H_{g^{-1}})\subseteq H_e/J$, and hence $\conj{g}{z}$ also lies in $H_e/J$. Independence of the chosen decomposition is Lemma~\ref{lem:ConjComm}, applied to the central element $z$.

(ii) By Lemma~\ref{lem:ConjComm}, the map $z\mapsto\conj{g}{z}$ is multiplicative and unital. Additivity is clear, and by~\cite[Cor.~3.4.3]{NVO04} the map $\conj{g}{}$ sends $Z(H_e/J)$ into itself; hence it is an algebra endomorphism of $Z(H_e/J)$. Taking the decomposition with $\bar h_1=\bar h_1'=1\in H_e/J$ gives $\conj{e}{z}=z$. If $\sum_i \bar h_i\bar h_i'=1$ is adapted to $g$ and $\sum_j \bar k_j\bar k_j'=1$ to $h$, then
\[
\sum_{i,j}(\bar h_i\bar k_j)(\bar k_j'\bar h_i')=\sum_i\bar h_i\left(\sum_j\bar k_j\bar k_j'\right)\bar h_i'=1
\]
is a decomposition adapted to $gh$, with $\bar h_i\bar k_j\in q_J(H_{gh})$ and $\bar k_j'\bar h_i'\in q_J(H_{(gh)^{-1}})$; since $\conj{h}{z}\in Z(H_e/J)$,
\[
\conj{gh}{z}
=\sum_{i,j}\bar h_i\bar k_j z\bar k_j'\bar h_i'
=\sum_i\bar h_i\left(\sum_j\bar k_j z\bar k_j'\right)\bar h_i'
=\sum_i\bar h_i\,(\conj{h}{z})\,\bar h_i'
=\conj{g}{(\conj{h}{z})}.
\]
In particular $\conj{g}{}\circ\conj{g^{-1}}{}=\conj{e}{}=\id$; hence $\conj{g}{}$ and $\conj{g^{-1}}{}$ are mutually inverse automorphisms of $Z(H_e/J)$, and $g\mapsto\conj{g}{}$ is a left action of $G$.

(iii) Let $z\in\Gc{H_e/J}=Z(H_e/J)\cap G(H_e/J)\cap (H_e/J)^\times$. By (ii), $\conj{g}{z}$ is central, and $\conj{g}{z}\,\conj{g}{(z^{-1})}=\conj{g}{(zz^{-1})}=\conj{g}{1}=1$; hence $\conj{g}{z}$ is invertible with inverse $\conj{g}{(z^{-1})}$. It remains to see that $\conj{g}{z}$ is group-like. Fix a decomposition $\sum_i\bar h_i\bar h_i'=1$ adapted to $g$. Since $\varepsilon$ is an algebra morphism with $\varepsilon(z)=1$, we get $\varepsilon(\conj{g}{z})=\sum_i\varepsilon(\bar h_i)\varepsilon(z)\varepsilon(\bar h_i')=\varepsilon\left(\sum_i\bar h_i\bar h_i'\right)=\varepsilon(1)=1$. For the comultiplication, $H/J^G\otimes H/J^G$ is strongly $(G\times G)$-graded, with neutral component $(H_e/J)\otimes(H_e/J)$. By the grading axiom~\eqref{eq:Grad1}, $\Delta(\bar h_i)\in q_J(H_g)\otimes q_J(H_g)$ and $\Delta(\bar h_i')\in q_J(H_{g^{-1}})\otimes q_J(H_{g^{-1}})$, and since $\Delta$ is an algebra morphism,
\[
\sum_i\Delta(\bar h_i)\Delta(\bar h_i')=\Delta\left(\sum_i\bar h_i\bar h_i'\right)=\Delta(1)=1\otimes 1
\]
is a decomposition adapted to $(g,g)$. The element $z\otimes z$ is central there, since $z$ is central in $H_e/J$. By Lemma~\ref{lem:ConjComm}, applied in $H/J^G\otimes H/J^G$, the action of $(g,g)$ on $z\otimes z$ may be evaluated using either this decomposition, with $\Delta(z)=z\otimes z$, or the product decomposition $\sum_{i,j}(\bar h_i\otimes\bar h_j)(\bar h_i'\otimes\bar h_j')=1\otimes 1$, giving
\begin{align*}
\Delta(\conj{g}{z})
&=\sum_i\Delta(\bar h_i)\,(z\otimes z)\,\Delta(\bar h_i')
=\sum_{i,j}(\bar h_i\otimes\bar h_j)(z\otimes z)(\bar h_i'\otimes\bar h_j')
\\
&=\left(\sum_i\bar h_i z\bar h_i'\right)\otimes\left(\sum_j\bar h_j z\bar h_j'\right)
=\conj{g}{z}\otimes\conj{g}{z}.
\end{align*}
Hence $\conj{g}{z}$ is group-like. Therefore $\conj{g}{z}\in\Gc{H_e/J}$, and $z\mapsto\conj{g}{z}$ restricts to a group automorphism of $\Gc{H_e/J}$.
\end{proof}

Concretely, if $\widetilde z\in H_e$ lifts $z$ and $\sum_i h_ih_i'=1$ in $H$ with $h_i\in H_g$, $h_i'\in H_{g^{-1}}$, then the class of $\sum_i h_i\widetilde z h_i'$ in $H_e/J$ equals $\conj{g}{z}$, independently of the lift $\widetilde z$ because the $G$-invariance of $J$ gives $H_gJH_{g^{-1}}\subseteq J$.

\begin{lemma}\label{lem:ConjUnits}
Let $(u_n)_{n\in N}$ be a group-like splitting of the grading of the lifted quotient $H_N/J^N$. Then, for every $g\in G$:
\begin{enumerate}
    \item For each $n\in N$ one has $\conj{g}{u_n}=c_{g,n}\,u_{gng^{-1}}$ for a unique $c_{g,n}\in(H_e/J)^\times$.

    \item Each $\conj{g}{u_n}$ is group-like and $c_{g,n}\in G(H_e/J)\cap(H_e/J)^\times$.
\end{enumerate}
\end{lemma}

\begin{proof}
Each $u_n$, and likewise each $u_n^{-1}$, commutes with $H_e/J$ by hypothesis; hence $\conj{g}{u_n}$ and $\conj{g}{(u_n^{-1})}$ are well defined by Lemma~\ref{lem:ConjComm}. For~(i), since $u_n\in q_J(H_n)$ and $gng^{-1}\in N$, the element $\conj{g}{u_n}$ lies in $q_J(H_{gng^{-1}})=(H_e/J)\,u_{gng^{-1}}$, and therefore $\conj{g}{u_n}=c_{g,n}u_{gng^{-1}}$ with $c_{g,n}\in H_e/J$. Lemma~\ref{lem:ConjComm} gives both
\[
\conj{g}{u_n}\,\conj{g}{(u_n^{-1})}=\conj{g}{1}=1
\quad\text{and}\quad
\conj{g}{(u_n^{-1})}\,\conj{g}{u_n}=\conj{g}{1}=1;
\]
hence $\conj{g}{u_n}$ is invertible. Consequently $c_{g,n}=\conj{g}{u_n}\,u_{gng^{-1}}^{-1}$ is invertible, and being a degree-$e$ element its inverse lies in $H_e/J$; thus $c_{g,n}\in(H_e/J)^\times$. For~(ii), applying $\varepsilon$ and $\Delta$ as in part~(iii) of Lemma~\ref{lem:ConjAction}, with Lemma~\ref{lem:ConjComm} supplying the decomposition-independence in $H/J^G\otimes H/J^G$ (applied to $u_n\otimes u_n$, which commutes with $(H_e/J)^{\otimes2}$) in place of centrality, shows that $\conj{g}{u_n}$ is group-like. Cancelling the group-like unit $u_{gng^{-1}}$ then gives $c_{g,n}\in G(H_e/J)$.
\end{proof}

Thus conjugation by $g$ carries each member of a group-like splitting to a group-like unit of degree $gng^{-1}$. Consequently, $\conj{g}{u_n}$ and $u_{gng^{-1}}$ are elements of the same component and may be compared.

The next lemma is what allows the ideal to be dropped from the datum. It replaces the $G/N$-invariance of $\ker(\tau)$ by an identity between homogeneous units.

\begin{lemma}\label{lem:ConjInv}
Let $(v_n)_{n\in N}$ be a group-like splitting of the grading of $H_N/J^N$, let $\tau\colon H_N/J^N\to H_e/J$ be the retraction it determines, and put $L:=\ker(\tau)$. Then:
\begin{enumerate}
    \item One has $L=\sum_{n\in N}(H_e/J)(v_n-1)$ and $L\cap(H_e/J)=0$. Moreover $(H_e/J)(v_n-1)=(H_e/J)(rv_n-r)$ for every unit $r$ of $H_e/J$.

    \item The ideal $L$ is $G/N$-invariant in $H/J^G$ if and only if $\conj{g}{v_n}=v_{gng^{-1}}$ for all $g\in G$ and $n\in N$.
\end{enumerate}
\end{lemma}

\begin{proof}
By Theorem~\ref{thm:Trivialization} the retraction $\tau$ is the unique morphism of quasi-Hopf algebras with $\tau|_{H_e/J}=\id$ and $\tau(v_n)=1$, and the component of degree $n$ is $(H_e/J)v_n$. If $x=\sum_nb_nv_n\in L$, then $\sum_nb_n=\tau(x)=0$; hence $x=\sum_nb_n(v_n-1)$, and the reverse inclusion is immediate. It follows that $L=\sum_n(H_e/J)(v_n-1)$, while $L\cap(H_e/J)=0$ because $\tau$ restricts to the identity on $H_e/J$. For a unit $r$ of $H_e/J$ one has $rv_n-r=r(v_n-1)$; therefore the two left ideals coincide.

For (ii), work with the strong $G$-grading of $H/J^G$ and the induced $G/N$-grading. The conjugation yields two identities, valid for every homogeneous $a\in q_J(H_g)$ with $g\in G$ and every $n\in N$, namely
\[
\conj{g}{v_n}\,a=a\,v_n
\qquad\text{and}\qquad
v_n\,a=a\,\conj{g^{-1}}{v_n}.
\]
For the first, take $1=\sum_i\bar h_i\bar h_i'$ adapted to $g$. Then $\bar h_i'a\in q_J(H_{g^{-1}})q_J(H_g)=H_e/J$ commutes with $v_n$, and therefore $\conj{g}{v_n}\,a=\sum_i\bar h_iv_n\bar h_i'a=\sum_i\bar h_i(\bar h_i'a)v_n=\left(\sum_i\bar h_i\bar h_i'\right)av_n=av_n$.

For the second, take $1=\sum_j\bar k_j\bar k_j'$ adapted to $g^{-1}$, with $\bar k_j\in q_J(H_{g^{-1}})$ and $\bar k_j'\in q_J(H_g)$. Then $a\bar k_j\in q_J(H_g)q_J(H_{g^{-1}})=H_e/J$ commutes with $v_n$, and therefore
\[
a\,\conj{g^{-1}}{v_n}=\sum_j a\bar k_jv_n\bar k_j'=\sum_j v_n(a\bar k_j)\bar k_j'=v_na\left(\sum_j\bar k_j\bar k_j'\right)=v_na.
\]
Both right-hand sides are independent of the decomposition; hence the identities hold.

Lemma~\ref{lem:ConjUnits}(i), applied to the family $(v_n)_{n\in N}$, gives $\conj{g}{v_n}=d_{g,n}\,v_{gng^{-1}}$ with $d_{g,n}\in(H_e/J)^\times$. Note that for $w\in N$ and $c\in H_e/J$, we have $c\,v_w-1=c(v_w-1)+(c-1)$ with $c(v_w-1)\in(H_e/J)(v_w-1)\subseteq L$. Hence $c\,v_w-1\in L$ if and only if $c-1\in L\cap(H_e/J)=0$, that is, if and only if $c=1$. In particular $\conj{g}{v_n}-1\in L$ if and only if $d_{g,n}=1$, that is, $\conj{g}{v_n}=v_{gng^{-1}}$.

Assume first that $\conj{g}{v_n}=v_{gng^{-1}}$ for all $g$ and $n$. Fix $\bar g\in G/N$. Since $q_J(H_{\bar g})=\sum_{g'\in\bar g}q_J(H_{g'})$ and $L$ is generated over $H_e/J$ by the elements $v_n-1$, it suffices to check the two inclusions on a homogeneous $a\in q_J(H_{g'})$ with $g'\in\bar g$ and on these generators. By the first identity, $a(v_n-1)=\conj{g'}{v_n}\,a-a=(v_{g'n g'^{-1}}-1)a$, which lies in $L\,q_J(H_{g'})\subseteq L\,q_J(H_{\bar g})$ because $g'n g'^{-1}\in N$ and $v_{g'n g'^{-1}}-1\in L$. Hence $q_J(H_{\bar g})L\subseteq L\,q_J(H_{\bar g})$. By the second identity, $(v_n-1)a=v_na-a=a\,\conj{g'^{-1}}{v_n}-a=a(v_{g'^{-1}ng'}-1)$, which lies in $q_J(H_{g'})L\subseteq q_J(H_{\bar g})L$. Hence $L\,q_J(H_{\bar g})\subseteq q_J(H_{\bar g})L$, and $L$ is $G/N$-invariant.

Conversely, assume that $L$ is $G/N$-invariant, and fix $g\in G$, $n\in N$ and a decomposition $1=\sum_i\bar h_i\bar h_i'$ adapted to $g$. From $\sum_i\bar h_i\bar h_i'=1$ we get
\[
\conj{g}{v_n}-1=\sum_i\bar h_iv_n\bar h_i'-\sum_i\bar h_i\bar h_i'=\sum_i\bar h_i(v_n-1)\bar h_i'.
\]
Each summand lies in $q_J(H_{\bar g})\,L\,q_J(H_{\bar g^{-1}})=L\,q_J(H_{\bar g})\,q_J(H_{\bar g^{-1}})=L\,(H_N/J^N)\subseteq L$, the first equality by invariance, the second by the strong $G/N$-grading of $H/J^G$, and the inclusion because $L$ is an ideal of $H_N/J^N$. Therefore $\conj{g}{v_n}-1\in L$, which as observed above forces $\conj{g}{v_n}=v_{gng^{-1}}$. This proves (ii).
\end{proof}

\subsection{Equivariant splittings and character data}\label{subsec:AdmData}

The three preceding lemmas combine to express a split admissible datum through a single group homomorphism. Conjugation in $H/J^G$ makes $G$ act on the central group-like elements $\Gc{H_e/J}$ by Lemma~\ref{lem:ConjAction} and carries each homogeneous unit to a unit of degree $gng^{-1}$ by Lemma~\ref{lem:ConjUnits}, while the remaining datum, the $G/N$-invariance of $\ker(\tau)$, is by Lemma~\ref{lem:ConjInv} equivalent to the conjugation-invariance of the units themselves. The families for which this holds deserve a name.

\begin{definition}[Equivariant family]\label{def:EquivFamily}
Let $N\unlhd G$ and let $J\subseteq H_e$ be a $G$-invariant quasi-Hopf ideal. An \emph{equivariant family} for $(N,J)$ is a group-like splitting $(u_n)_{n\in N}$ of the grading of the lifted quotient $H_N/J^N$ such that $\conj{g}{u_n}=u_{gng^{-1}}$ for all $g\in G$ and $n\in N$.
\end{definition}

\begin{proposition}\label{prop:Realized}
Let $N\unlhd G$ and let $J\subseteq H_e$ be a $G$-invariant quasi-Hopf ideal. There is a quasi-Hopf ideal $I\subseteq H$ with $N_I=N$ and $I\cap H_e=J$ if and only if $(N,J)$ admits an equivariant family.
\end{proposition}

\begin{proof}
Let $I$ be such an ideal. By Theorem~\ref{thm:MainB} it determines a split admissible datum $(N,J,\tau)$, whose retraction $\tau\colon H_N/J^N\to H_e/J$ has $G/N$-invariant kernel by Definition~\ref{def:SplitAdmDatum}(iv). Theorem~\ref{thm:Trivialization} attaches to $\tau$ a family $(u_n)_{n\in N}$ as in its part~(iii), with $\tau(u_n)=1$, and Lemma~\ref{lem:ConjInv}(ii), applied to that family, gives $\conj{g}{u_n}=u_{gng^{-1}}$. The family is therefore equivariant.

Conversely, let $(u_n)_{n\in N}$ be an equivariant family. Theorem~\ref{thm:Trivialization} provides a quasi-Hopf algebra retraction $\tau\colon H_N/J^N\to H_e/J$ with $\tau(u_n)=1$, and Lemma~\ref{lem:ConjInv}(ii), applied to the same family, makes $\ker(\tau)$ $G/N$-invariant. Hence $(N,J,\tau)$ is a split admissible datum, and Theorem~\ref{thm:MainB} returns a quasi-Hopf ideal $I$ with $N_I=N$ and $I\cap H_e=J$.
\end{proof}

Over a pair realized by some quasi-Hopf ideal, Proposition~\ref{prop:Realized} lets a family of homogeneous units be taken equivariant, and the datum that remains is then a single group homomorphism, which the following definition axiomatizes.

\begin{definition}[Character datum]\label{def:AdmData}
A \emph{character datum} for $H$ is a triple $(N,J,\gamma)$ such that:
\begin{enumerate}
    \item $N\unlhd G$;

    \item $J\subseteq H_e$ is a $G$-invariant quasi-Hopf ideal such that $(N,J)$ admits an equivariant family;

    \item $\gamma\colon N\to\Gc{H_e/J}$ is a group homomorphism satisfying
    \begin{equation}\label{eq:Equivar}
    \conj{g}{\gamma(n)}=\gamma(gng^{-1}),
    \qquad
    g\in G,\ n\in N.
    \end{equation}
\end{enumerate}
\end{definition}

A group homomorphism satisfying~\eqref{eq:Equivar} is called \emph{$G$-equivariant}, the action on $\Gc{H_e/J}$ being that of Lemma~\ref{lem:ConjAction}.

The name is chosen in analogy with the collective characters of~\cite[\S4]{Ce02}, which attach to a strongly graded Hopf algebra a group homomorphism of the grading group into an invariant of the neutral component.

For each pair $(N,J)$ admitting an equivariant family, we fix one and refer to its members as the \emph{reference units} for $(N,J)$. They are background data rather than a fourth component of the datum, and $\gamma$ is defined relative to them.

\begin{theorem}\label{thm:Char}
Let $H$ be a faithfully $G$-graded quasi-Hopf algebra, let $N\unlhd G$ and let $J\subseteq H_e$ be a $G$-invariant quasi-Hopf ideal admitting an equivariant family, with reference units $\{u_n\}_{n\in N}$. The assignments
\[
(N,J,\gamma)\longmapsto I_\gamma:=\left\langle\,J^G,\ \widetilde u_n-\widetilde\gamma(n)\ \mid\ n\in N\,\right\rangle_H
\qquad\text{and}\qquad
I\longmapsto\left(N_I,\,I\cap H_e,\,\gamma\right)
\]
define a bijective correspondence between the set of character data for $H$ with first two components $N$ and $J$ and the set of quasi-Hopf ideals $I\subseteq H$ with $N_I=N$ and $I\cap H_e=J$, where $\widetilde u_n\in H_n$ and $\widetilde\gamma(n)\in H_e$ are arbitrary lifts of $u_n$ and $\gamma(n)$, respectively, and, in the second assignment, $\gamma(n)=\tau(u_n)$ for the retraction $\tau$ of Theorem~\ref{thm:MainB}.
\end{theorem}

\begin{proof}
Let $\gamma$ be as in Definition~\ref{def:AdmData} and put $v_n:=\gamma(n)^{-1}u_n$. Each $v_n$ is a group-like unit of degree $n$ commuting with $H_e/J$, and $v_nv_m=v_{nm}$; hence $(v_n)_{n\in N}$ is a group-like splitting. Since the reference units are equivariant, Lemmas~\ref{lem:ConjComm} and~\ref{lem:ConjAction} give $\conj{g}{v_n}=\conj{g}{\gamma(n)}^{-1}u_{gng^{-1}}$; therefore~\eqref{eq:Equivar} holds if and only if $(v_n)_{n\in N}$ is an equivariant family. Conversely an equivariant family $(v_n)_{n\in N}$ determines $\gamma(n):=u_nv_n^{-1}\in\Gc{H_e/J}$, a group homomorphism because both families are multiplicative and centralize $H_e/J$.

Theorem~\ref{thm:Trivialization} matches such families with the quasi-Hopf algebra retractions of $H_N/J^N$ onto $H_e/J$, sending $(v_n)$ to the retraction $\tau$ with $\tau(u_n)=\gamma(n)$. By Proposition~\ref{prop:Realized} and Theorem~\ref{thm:MainB} these retractions correspond in turn to the quasi-Hopf ideals $I$ with $N_I=N$ and $I\cap H_e=J$. By Lemma~\ref{lem:ConjInv}(i) and its rescaling clause, $\ker\tau=\sum_n(H_e/J)(v_n-1)=\sum_n(H_e/J)(u_n-\gamma(n))$; hence the inverse image $K\subseteq H_N$ of $\ker\tau$ is generated by $J^N$ and the lifted differences, and the reconstruction $\sum_{\bar g\in G/N}H_{\bar g}K$ of Theorem~\ref{thm:MainB} is the displayed ideal $I_\gamma$. The two assignments are therefore mutually inverse.
\end{proof}

The fibre over $(N,J)$ is thus a torsor under the group of $G$-equivariant homomorphisms $N\to\Gc{H_e/J}$. The reference units supply a base point, but no choice of them is canonical. By Theorem~\ref{thm:MainB} the ideal determines its retraction, whereas $\gamma(n)=\tau(u_n)$ depends on the chosen units. Thus the retraction is intrinsic, while $\gamma$ is the coordinate of the ideal with respect to that base point. The obstruction to the existence of an ideal over $(N,J)$ is accordingly that the fibre is empty, as recorded in Proposition~\ref{prop:Realized}, and not a cohomology class.

\subsection{Crossed-product presentations}\label{subsec:CrossedRetr}

The preceding subsections work with the group-like splittings directly. We now compute them inside a chosen crossed-product presentation, where a retraction appears as a family of units of the neutral component and the failure of that family to be group-like and multiplicative is measured by a factor set and a cohomology class. Throughout this subsection and the next, $N$ is a group and $A$ denotes an $N$-crossed product quasi-Hopf algebra with neutral component $A_e$. We also fix homogeneous units $u_n\in A_n^\times$, for $n\in N$, with $u_e=1$, and recall that every element of $A$ then has a unique expression as a finite sum $\sum_{n\in N}b_nu_n$ with $b_n\in A_e$.

\begin{definition}[Crossed system]\label{def:CrossedSystem}
A \emph{crossed system} over $A_e$ is a pair $(\varphi,\sigma)$ consisting of algebra automorphisms $\varphi_n\colon A_e\to A_e$ and units $\sigma(n,m)\in A_e^\times$, for $n,m\in N$, satisfying
\begin{gather}
\varphi_n\varphi_m(b)
=
\sigma(n,m)\varphi_{nm}(b)\sigma(n,m)^{-1},
\qquad
b\in A_e, \label{eq:CrossedActId}\\
\varphi_n(\sigma(m,\ell))\,\sigma(n,m\ell)
=
\sigma(n,m)\,\sigma(nm,\ell),
\qquad
n,m,\ell\in N, \label{eq:CrossedCocyId}
\end{gather}
together with the normalization $\varphi_e=\id_{A_e}$ and $\sigma(e,n)=\sigma(n,e)=1$. We call $\varphi=(\varphi_n)_{n\in N}$ a \emph{weak action} of $N$ on $A_e$ and $\sigma$ a \emph{factor set}~\cite[\S1.4]{NVO04}.
\end{definition}

Following~\cite[Prop.~1.4.2]{NVO04}, the chosen homogeneous units determine such a crossed system.

\begin{lemma}\label{lem:CrossedSystem}
For $n,m\in N$, setting
\[
\varphi_n(b)=u_nbu_n^{-1}\quad(b\in A_e)\qquad\text{and}\qquad\sigma(n,m)=u_nu_mu_{nm}^{-1}
\]
defines a crossed system $(\varphi,\sigma)$ over $A_e$. Moreover,
\[
u_nb=\varphi_n(b)u_n
\qquad\text{and}\qquad
u_nu_m=\sigma(n,m)u_{nm}.
\]
\end{lemma}

\begin{definition}[Trivializing family]\label{def:TrivFamily}
A \emph{trivializing family} for the crossed system $(\varphi,\sigma)$ is a family $r=(r_n)_{n\in N}$ of units $r_n\in A_e^\times$ satisfying
\begin{gather}
r_n b=\varphi_n(b)r_n,
\qquad
b\in A_e, \label{eq:RetrAct}\\
r_nr_m=\sigma(n,m)r_{nm}, \label{eq:RetrCocy}\\
r_e=1. \label{eq:RetrUnit}
\end{gather}
\end{definition}

Algebra retractions are equivalent to trivializing families, as follows; compare~\cite[\S1.5, Ex.~9]{NVO04} and~\cite[\S7.3]{Mon93}. The precise statement does not appear in this form in the literature; we therefore include a proof.

\begin{lemma}\label{lem:CrossedAlgRetr}
The algebra retractions $\tau\colon A\to A_e$ with $\tau|_{A_e}=\id_{A_e}$ are in bijective correspondence, via $r_n=\tau(u_n)$, with the trivializing families $r=(r_n)_{n\in N}$. The inverse correspondence sends a trivializing family $r$ to the map $\tau_r$ defined by
\begin{equation}\label{eq:TauRDef}
\tau_r\left(\sum_{n\in N}b_nu_n\right)
=
\sum_{n\in N}b_nr_n.
\end{equation}
\end{lemma}

\begin{proof}
Let $\tau\colon A\to A_e$ be an algebra retraction and set $r_n=\tau(u_n)$. Since each $u_n$ is invertible, each $r_n$ is invertible in $A_e$, and $r_e=\tau(1)=1$; applying $\tau$ to $u_nb=\varphi_n(b)u_n$ and to $u_nu_m=\sigma(n,m)u_{nm}$ gives $r_nb=\varphi_n(b)r_n$ and $r_nr_m=\sigma(n,m)r_{nm}$. Thus the family $(r_n)$ satisfies \eqref{eq:RetrAct}, \eqref{eq:RetrCocy} and \eqref{eq:RetrUnit}.

Conversely, suppose that a family $(r_n)$ satisfies these identities and define $\tau_r$ by~\eqref{eq:TauRDef}. This map is $\kk$-linear, restricts to the identity on $A_e$, and preserves the unit because $r_e=1$. For $b,c\in A_e$ and $n,m\in N$, one has $(bu_n)(cu_m)=b\varphi_n(c)\sigma(n,m)u_{nm}$, and
\[
\tau_r(bu_n)\,\tau_r(cu_m)
=br_ncr_m
=b\varphi_n(c)r_nr_m
=b\varphi_n(c)\sigma(n,m)r_{nm}
=\tau_r\!\left((bu_n)(cu_m)\right),
\]
by \eqref{eq:RetrAct} and \eqref{eq:RetrCocy}. Thus $\tau_r$ is multiplicative on the homogeneous generators $bu_n$, and therefore on all of $A$ by $\kk$-bilinearity. Thus $\tau_r$ is an algebra retraction.

The two constructions are mutually inverse: $\tau_r(u_n)=r_n$ and, since any algebra retraction satisfies $\tau\left(\sum_n b_nu_n\right)=\sum_n b_n\tau(u_n)$, comparison with~\eqref{eq:TauRDef} gives $\tau=\tau_r$ for $r_n=\tau(u_n)$.
\end{proof}

Since $\alpha,\beta\in A_e$ and $\Phi\in A_e^{\otimes3}$, an algebra retraction preserves them automatically. It remains to impose compatibility with $\varepsilon$, $\Delta$, and $S$.

\begin{lemma}\label{lem:CrossedQHRetr}
The algebra retraction $\tau_r\colon A\to A_e$ associated to a trivializing family $(r_n)_{n\in N}$ is a morphism of quasi-Hopf algebras if and only if, for every $n\in N$,
\begin{gather}
\varepsilon(r_n)=\varepsilon(u_n), \label{eq:QHRetrCounit}\\
\Delta(r_n)=(\tau_r\otimes \tau_r)\Delta(u_n), \label{eq:QHRetrCoprod}\\
S(r_n)=\tau_r(S(u_n)). \label{eq:QHRetrAntipode}
\end{gather}
\end{lemma}

\begin{proof}
If $\tau_r$ is a morphism of quasi-Hopf algebras, then \eqref{eq:QHRetrCounit}--\eqref{eq:QHRetrAntipode} are obtained by applying the counit, coproduct and antipode compatibility to the elements $u_n$.

Conversely, assume that \eqref{eq:QHRetrCounit}, \eqref{eq:QHRetrCoprod} and \eqref{eq:QHRetrAntipode} hold for every $n\in N$. Since every element of $A$ is a finite sum of elements $bu_n$ with $b\in A_e$, and $A_e$ is a quasi-Hopf subalgebra of $A$ (Proposition~\ref{prop:GradProp}(ii)) on which $\tau_r$ restricts to the identity, each compatibility reduces to the corresponding hypothesis. For the counit, $\varepsilon(\tau_r(bu_n))=\varepsilon(b)\varepsilon(r_n)=\varepsilon(b)\varepsilon(u_n)=\varepsilon(bu_n)$, by~\eqref{eq:QHRetrCounit}. For the coproduct,
\[
\Delta(\tau_r(bu_n))
=\Delta(b)\Delta(r_n)
=\Delta(b)\,(\tau_r\otimes\tau_r)\Delta(u_n)
=(\tau_r\otimes\tau_r)\Delta(bu_n),
\]
using~\eqref{eq:QHRetrCoprod} in the middle equality and, in the last one, that $\tau_r\otimes\tau_r$ is left $(A_e\otimes A_e)$-linear. For the antipode, since $S$ is an anti-morphism and $\tau_r$ is right $A_e$-linear by~\eqref{eq:RetrAct}, we get 
\[
\tau_r(S(bu_n))=\tau_r(S(u_n))\,S(b)=S(r_n)S(b)=S(br_n)=S(\tau_r(bu_n)),
\]
by~\eqref{eq:QHRetrAntipode}. Finally, $\alpha$, $\beta$ and $\Phi$ are preserved because they lie in the corresponding tensor powers of $A_e$, on which $\tau_r$ restricts to the identity. Hence $\tau_r$ is a morphism of quasi-Hopf algebras.
\end{proof}

Let $H$ be a faithfully $G$-graded quasi-Hopf algebra, fix $N\unlhd G$ and a $G$-invariant quasi-Hopf ideal $J\subseteq H_e$, and suppose that the lifted quotient $H_N/J^N$ is an $N$-crossed product, with fixed homogeneous units $u_n$. We apply the two preceding lemmas with $A=H_N/J^N$ and $A_e=H_e/J$. They identify its quasi-Hopf retractions with the trivializing families satisfying~\eqref{eq:RetrAct}--\eqref{eq:RetrUnit} and~\eqref{eq:QHRetrCounit}--\eqref{eq:QHRetrAntipode}.

\begin{proposition}\label{prop:CrossedSplitData}
With the notation above, the split admissible data $(N,J,\tau)$ are in bijection, via $r_n=\tau(u_n)$, with the trivializing families $r=(r_n)_{n\in N}$ of units $r_n\in (H_e/J)^\times$ such that:
\begin{enumerate}
    \item $\tau_r$ is a morphism of quasi-Hopf algebras;

    \item $\ker(\tau_r)$ is $G/N$-invariant in $H/J^G$.
\end{enumerate}
The corresponding quasi-Hopf ideal of $H$ is $I=\sum_{\bar g\in G/N}H_{\bar g}K$, where $K\subseteq H_N$ is the inverse image under the quotient map $H_N\to H_N/J^N$ of
\[
L:=\ker(\tau_r)=\sum_{n\in N}(H_e/J)(u_n-r_n).
\]
Equivalently, for arbitrary lifts $\widetilde u_n\in H_n$ and $\widetilde r_n\in H_e$ of $u_n$ and $r_n$, the ideal $K$ is generated in $H_N$ by $J^N$ together with the elements $\widetilde u_n-\widetilde r_n$, $n\in N$.
\end{proposition}

\begin{proof}
By Lemmas~\ref{lem:CrossedAlgRetr} and~\ref{lem:CrossedQHRetr}, condition (i) holds exactly for the families $r$ that correspond, via $r_n=\tau(u_n)$, to quasi-Hopf algebra retractions $\tau\colon H_N/J^N\to H_e/J$, and (ii) is condition (iv) of Definition~\ref{def:SplitAdmDatum} for $\ker(\tau_r)$. The corresponding ideal is the one of Theorem~\ref{thm:MainB}. For the description of $L$, if $x=\sum_{n\in N}b_nu_n$ lies in $\ker(\tau_r)$, then $\sum_{n\in N}b_nr_n=\tau_r(x)=0$ by~\eqref{eq:TauRDef}; hence $x=\sum_{n\in N}b_n(u_n-r_n)$ and $L\subseteq\sum_{n\in N}(H_e/J)(u_n-r_n)$. The reverse inclusion holds since $\tau_r(u_n-r_n)=0$. Taking inverse images under the quotient map $H_N\to H_N/J^N$, whose kernel is $J^N$, the ideal $K$ is generated by $J^N$ and the lifts $\widetilde u_n-\widetilde r_n$.
\end{proof}

The family $(r_n)$ and the crossed system $(\varphi,\sigma)$ depend on the homogeneous units; the retraction $\tau$ and its kernel do not.

\subsection{The cohomological obstruction}\label{subsec:CohObstr}

We now identify the obstruction to the existence of a trivializing family with central group-like values, described by a map $\gamma\colon N\to\Gc{A_e}$. Such a map need not be a group homomorphism; it is one exactly when $\sigma=1$, and when solutions exist they form a torsor under $\Hom(N,\Gc{A_e})$.

Return to the abstract $N$-crossed product $A$ fixed at the start of Subsection~\ref{subsec:CrossedRetr}, with homogeneous units $u_n$ and associated crossed system $(\varphi,\sigma)$.

\begin{lemma}\label{lem:CentralRetrData}
Let $r=(r_n)_{n\in N}$ be a trivializing family and suppose that every $r_n$ is central in $A_e$. Then $\varphi_n=\id_{A_e}$ for all $n\in N$. Moreover, if $\varphi_n=\id_{A_e}$ for all $n\in N$, then the factor set $\sigma$ takes values in $Z(A_e)^\times$ and satisfies the ordinary $2$-cocycle identity
\[
\sigma(m,\ell)\,\sigma(n,m\ell)=\sigma(n,m)\,\sigma(nm,\ell),
\qquad
n,m,\ell\in N.
\]
\end{lemma}

\begin{proof}
Suppose that every $r_n$ is central. For $b\in A_e$, equation~\eqref{eq:RetrAct} gives $br_n=r_nb=\varphi_n(b)r_n$. Since $r_n$ is invertible, $\varphi_n(b)=b$; hence $\varphi_n=\id_{A_e}$. For the second assertion, assume that $\varphi_n=\id_{A_e}$ for all $n\in N$. Then identity~\eqref{eq:CrossedActId} reads $b=\sigma(n,m)\,b\,\sigma(n,m)^{-1}$ for all $b\in A_e$; hence each $\sigma(n,m)$ is central. Identity~\eqref{eq:CrossedCocyId} then becomes the displayed cocycle identity, because $\varphi_n(\sigma(m,\ell))=\sigma(m,\ell)$.
\end{proof}

In view of Lemma~\ref{lem:CentralRetrData}, we make the following standing hypotheses for the rest of this subsection:
\begin{enumerate}
    \condlabel{CP}{cond:CP} $A$ is an $N$-crossed product, with a fixed choice of homogeneous units $u_n\in A_n^\times$, $u_e=1$, normalized by $\varepsilon(u_n)=1$ for all $n\in N$ (this is always possible, replacing $u_n$ by $\varepsilon(u_n)^{-1}u_n$, since $\varepsilon(u_n)\in\kk^\times$);

    \condlabel{TA}{cond:TrivAct} the associated action is trivial: $\varphi_n=\id_{A_e}$ for all $n\in N$;

    \condlabel{CG}{cond:CentralCocy} the factor set takes central group-like values: $\sigma(n,m)\in\Gc{A_e}$ for all $n,m\in N$.
\end{enumerate}
Under \condref{cond:TrivAct}, Lemma~\ref{lem:CentralRetrData} shows that $\sigma$ is an ordinary $2$-cocycle with central values, normalized because $u_e=1$ (Definition~\ref{def:CrossedSystem}). Condition \condref{cond:CentralCocy} asks in addition that these values be group-like. Thus, under \condref{cond:CP}, \condref{cond:TrivAct} and \condref{cond:CentralCocy}, $\sigma\in Z^2(N,\Gc{A_e})$ is a normalized $2$-cocycle of $N$ with values in the abelian group $\Gc{A_e}$, for the trivial action, with a cohomology class $[\sigma]\in H^2(N,\Gc{A_e})$.

\begin{proposition}\label{prop:CohObstr}
Assume hypotheses \condref{cond:CP}, \condref{cond:TrivAct} and \condref{cond:CentralCocy}. Then:
\begin{enumerate}
    \item Algebra retractions $\tau\colon A\to A_e$ with $\tau|_{A_e}=\id_{A_e}$ and $\tau(u_n)\in\Gc{A_e}$ for all $n\in N$ are in bijective correspondence, via $\gamma(n)=\tau(u_n)$, with the maps $\gamma\colon N\to\Gc{A_e}$ satisfying $\gamma(e)=1$ and the coboundary condition
    \begin{equation}\label{eq:CobdyCond}
    \sigma(n,m)=\gamma(n)\gamma(m)\gamma(nm)^{-1},
    \qquad
    n,m\in N.
    \end{equation}

    \item If $[\sigma]\neq1$ in $H^2(N,\Gc{A_e})$, the set of such retractions is empty. If $[\sigma]=1$, it is a torsor under the group $\Hom(N,\Gc{A_e})$, acting by pointwise multiplication on $\gamma$.
\end{enumerate}
\end{proposition}

\begin{proof}
(i) By Lemma~\ref{lem:CrossedAlgRetr}, algebra retractions $\tau\colon A\to A_e$ correspond to families $r_n=\tau(u_n)\in A_e^\times$ satisfying \eqref{eq:RetrAct}, \eqref{eq:RetrCocy} and \eqref{eq:RetrUnit}. Under \condref{cond:TrivAct}, equation~\eqref{eq:RetrAct} holds automatically for any central family. Hence retractions with values $r_n=\gamma(n)\in\Gc{A_e}$ correspond exactly to maps $\gamma\colon N\to\Gc{A_e}$ with $\gamma(e)=1$ and $\gamma(n)\gamma(m)=\sigma(n,m)\gamma(nm)$. Since $\gamma(nm)$ is invertible, this is condition~\eqref{eq:CobdyCond}.

(ii) Condition~\eqref{eq:CobdyCond} states precisely that $\sigma=\partial\gamma$ is the coboundary of the cochain $\gamma$ with values in $\Gc{A_e}$. A solution $\gamma$ therefore exists if and only if $[\sigma]=1$ in $H^2(N,\Gc{A_e})$, and the set of retractions is empty otherwise. When it is nonempty, let $\gamma$ and $\gamma'$ be two solutions and put $\chi(n):=\gamma'(n)\gamma(n)^{-1}\in\Gc{A_e}$. Centrality and~\eqref{eq:CobdyCond} for both solutions give 
\[
\chi(n)\chi(m)=\gamma'(n)\gamma'(m)\left(\gamma(n)\gamma(m)\right)^{-1}=\sigma(n,m)\gamma'(nm)\left(\sigma(n,m)\gamma(nm)\right)^{-1}=\chi(nm).
\]
Thus $\chi\in\Hom(N,\Gc{A_e})$. Conversely, if $\gamma$ solves~\eqref{eq:CobdyCond} and $\chi\in\Hom(N,\Gc{A_e})$, then $\chi\gamma$ is again a solution. The action is simply transitive.
\end{proof}

\begin{remark}\label{rem:UnitDep}
Let $(N,J,\tau)$ be split admissible and suppose that $H_N/J^N$ is a crossed product satisfying \condref{cond:TrivAct} for chosen homogeneous units $\{u_n\}_{n\in N}$, with factor set $\sigma$. Put $r_n:=\tau(u_n)$ and $u_n':=r_n^{-1}u_n$. This changes the homogeneous section, not the quasi-Hopf structure. Each $u_n'$ is the unique element of degree $n$ that $\tau$ sends to $1$; hence $(u_n')_{n\in N}$ is the family attached to $\tau$ by Theorem~\ref{thm:Trivialization}. Its members are group-like units satisfying $u_n'u_m'=u_{nm}'$ and commuting with $H_e/J$; therefore the new factor set is trivial. Let $\gamma'$ be the constant map with value $1$. Its twisted units are the $u_n'$ themselves; hence Lemma~\ref{lem:ConjInv}, applied to the new units, supplies their equivariance, and $(N,J,\gamma')$ is a character datum. In particular, the new units form an equivariant family, whatever the original factor set. Thus the factor set and whether the original retraction values are group-like depend on the chosen section. Whenever the factor set is $\Gc{H_e/J}$-valued, it determines a class in $H^2(N,\Gc{H_e/J})$. This class is invariant under regauging by normalized $\Gc{H_e/J}$-valued families. Under regauging by arbitrary central units, the new factor set need not be group-like and hence need not determine such a class; even when it does, its class may change. These section-dependent data remain useful for a distinguished section, such as $u_g=1\#g$ in~\cite[\S5.1]{GM12}.
\end{remark}

By Corollary~\ref{cor:LiftTensor} the weak action is trivial for the units attached to a quasi-Hopf ideal; hence condition~\condref{cond:TrivAct} can always be arranged on the lifted quotient of an ideal. It can fail for a crossed product that is not of that form, as for the skew group algebra $\kk[t]\rtimes\mathbb{Z}/2\mathbb{Z}$ with $g\cdot t=-t$ over a field of characteristic different from $2$~\cite[\S4.1]{Mon93}. Its neutral component is commutative; consequently, changing homogeneous units does not alter the nontrivial action. By contrast, the factor set with central group-like values and the group-like retraction values are section-dependent, as Remark~\ref{rem:UnitDep} shows.

By Corollary~\ref{cor:LiftTensor} the lifted quotient of a quasi-Hopf ideal is always a crossed product, indeed a trivial one. Thus Theorem~\ref{thm:Char} requires no additional crossed-product hypothesis. What it requires is an equivariant family, which by Proposition~\ref{prop:Realized} is exactly what a realized pair provides. The semisimple categorical classification of~\cite[Thm.~4.8]{GJ25} provides the analogous trivialization in its own setting; a direct algebraic comparison requires duality hypotheses unavailable here. Related questions about whether Hopf algebras are faithfully (co)flat over coideal subalgebras remain open even for ordinary Hopf algebras: see~\cite{Skr25} for an open question in Takeuchi's correspondence about whether the relevant comodule is faithfully coflat, \cite{Skr25b} for examples in which a Hopf algebra is not flat over a finite-dimensional Hopf subalgebra, and~\cite{Bic25} for a theorem proving that such an extension is faithfully flat when it admits a bimodule conditional expectation.

\section{Applications and examples}\label{sec:Apps}

Throughout this section, $H$ is a faithfully $G$-graded quasi-Hopf algebra with grading map $\pi$, except where a specific algebra is introduced. We specialize the classification to several classes of quasi-Hopf algebras and to distinguished subgroups. Theorem~\ref{thm:MainA} describes every quasi-Hopf ideal by an admissible pair. Theorem~\ref{thm:MainB} gives the corresponding description by retractions without assuming a crossed-product structure or requiring any algebra to be finite-dimensional or semisimple. Theorem~\ref{thm:Char} refines this description over a pair realized by some quasi-Hopf ideal, where the ideals lying over the pair are classified by the $G$-equivariant group homomorphisms $N\to\Gc{H_e/J}$. The examples first isolate the cohomological obstruction. They then exhibit a strongly graded Hopf algebra that is not a crossed product and clarify the role of the equivariant family.

\subsection{Graded and augmented ideals}\label{subsec:Extreme}

For $N=\{e\}$, Theorem~\ref{thm:MainA} reduces to the neutral component.

\begin{corollary}\label{cor:GradIdeals}
The assignments
\[
I\longmapsto I\cap H_e
\qquad\text{and}\qquad
J\longmapsto J^G
\]
define a bijective correspondence between the set of $G$-graded quasi-Hopf ideals of $H$ and the set of $G$-invariant quasi-Hopf ideals $J\subseteq H_e$.
\end{corollary}

\begin{proof}
For $N=\{e\}$, the augmentation condition is void and the admissible pairs are exactly $(\{e\},J)$ with $J$ a $G$-invariant quasi-Hopf ideal of $H_e$. Theorem~\ref{thm:MainA} sends $(\{e\},J)$ to $J^G$. Conversely, a graded ideal $I$ has $\pi(I)=0$ by~\eqref{eq:GradMap}, since $\varepsilon$ vanishes on its homogeneous components; hence $N_I=\{e\}$. The reverse implication is Proposition~\ref{prop:AssocSubgrp}(iii).
\end{proof}

For $N=G$, Theorem~\ref{thm:MainA} describes the augmented ideals. By Proposition~\ref{prop:AssocSubgrp}(i), $I$ is augmented (Definition~\ref{def:AugIdeal}) if and only if $N_I=G$, since the assignment $N\mapsto\aug{N}\cdot\kk[G]$ is injective on normal subgroups. Thus augmented ideals correspond to admissible pairs $(G,K)$. By Theorem~\ref{thm:MainB}, they correspond to retractions $\tau\colon H/J^G\to H_e/J$ with $J=I\cap H_e$.

The character form records an augmented ideal by $\gamma(g)=\tau(u_g)$ (cf.~\cite[\S4]{Ce02}); conversely, a cochain satisfying the conditions below determines
\[
I_\gamma=\left\langle J^G,\ \widetilde u_g-\widetilde\gamma(g)\mid g\in G\right\rangle_H,
\]
with $\widetilde u_g\in H_g$ and $\widetilde\gamma(g)\in H_e$ arbitrary lifts of $u_g$ and $\gamma(g)$.

\begin{corollary}\label{cor:AugApps}
Let $J\subseteq H_e$ be a $G$-invariant quasi-Hopf ideal such that $(G,J)$ admits an equivariant family $\{u_g\}_{g\in G}$. The assignments
\[
I\longmapsto\gamma
\qquad\text{and}\qquad
\gamma\longmapsto I_\gamma
\]
define a bijective correspondence between the set of augmented quasi-Hopf ideals of $H$ with $I\cap H_e=J$ and the set of $G$-equivariant group homomorphisms $\gamma\colon G\to\Gc{H_e/J}$, a subgroup of $\Hom\!\left(G,\Gc{H_e/J}\right)$. The ideals therefore form a torsor under that subgroup, acting through the correspondence by pointwise multiplication on $\gamma$, with base point the ideal of the trivial homomorphism.
\end{corollary}

\begin{proof}
This is Theorem~\ref{thm:Char} with $N=G$, together with Proposition~\ref{prop:AssocSubgrp}(i), which identifies the augmented ideals as those with $N_I=G$. Since each $\conj{h}{(\cdot)}$ is a group automorphism of $\Gc{H_e/J}$ by Lemma~\ref{lem:ConjAction}(iii), the $G$-equivariant group homomorphisms are closed under products and inverses and contain the trivial one. The torsor structure is the one described after Theorem~\ref{thm:Char}.
\end{proof}

\subsection{Crossed products}\label{subsec:CrossedApp}

Let $A$ be a quasi-Hopf algebra, let $G$ act on $A$ by quasi-Hopf algebra automorphisms, and let $\sigma\colon G\times G\to\Gc{A}$ be a normalized $2$-cocycle. The automorphisms $\varphi_g=\conj{g}{(\cdot)}$ and the factor set $\sigma$ form a crossed system over $A$ in the sense of Definition~\ref{def:CrossedSystem}, with $G$ and $A$ in place of $N$ and $A_e$. The quasi-Hopf crossed systems of~\cite[\S5.1]{GM12} carry in addition a coproduct twist for each group element and an antipode datum, both recorded below. Let $H=A\#_\sigma\kk[G]$ be the associated crossed product, with underlying space $A\otimes\kk[G]$, elements written $a\#g$, and multiplication $(a\#g)(b\#h)=a\,(\conj{g}{b})\,\sigma(g,h)\#gh$. We use the specialization of the crossed-product quasi-Hopf construction in~\cite[Props.~5.2--5.3]{GM12}:
\[
\begin{gathered}
\Delta_H(a\# g)=a'\#g\otimes a''\#g,\qquad
\varepsilon_H(a\# g)=\varepsilon_A(a),\\
\Phi_H=\Phi_A^{(1)}\#e\otimes\Phi_A^{(2)}\#e\otimes\Phi_A^{(3)}\#e,\qquad
\alpha_H=\alpha_A\#e,\qquad
\beta_H=\beta_A\#e,\\
S_H(a\# g)=\sigma(g^{-1},g)^{-1}\,\conj{g^{-1}}{S_A(a)}\# g^{-1},
\end{gathered}
\]
where $a'\otimes a''=\Delta_A(a)$. Here the twisted automorphisms of the general construction are ordinary quasi-Hopf algebra automorphisms; hence their coproduct twists are trivial. The central group-like values of $\sigma$ then make these formulas compatible with the crossed-system identities, the antipode of the crossed system being $\upsilon_g=\sigma(g^{-1},g)^{-1}$. This last point rests on the identity $S_A(z)=z^{-1}$, valid for every $z\in\Gc{A}$. Indeed, set $t:=S_A(z)\,z$. The first identity of~\eqref{eq:QH6} applied to $z$ reads $S_A(z)\,\alpha_A\,z=\alpha_A$; hence $t\,\alpha_A=\alpha_A$ because $z$ is central. The same centrality gives $S_A(z)\,S_A(a)=S_A(az)=S_A(za)=S_A(a)\,S_A(z)$ for every $a\in A$; hence $t$ commutes with the image of $S_A$. Multiplying the second identity of~\eqref{eq:QH7} on the left by $t$ therefore gives $t=S_A\!(\Phi_A^{(-1)})\,t\,\alpha_A\,\Phi_A^{(-2)}\,\beta_A\,S_A\!(\Phi_A^{(-3)})=1$. The resulting algebra is strongly $G$-graded with $H_g=A\#g$ and $H_e\cong A$ via $a\mapsto a\#e$, and the elements $u_g=1\#g$ are homogeneous units. Hence $H$ is a $G$-crossed product in the sense of Definition~\ref{def:SG}.

For a normal subgroup $N\unlhd G$ and a $G$-invariant quasi-Hopf ideal $J\subseteq A$, the subalgebra $H_N=\bigoplus_{n\in N}A\#n$ equals $A\#_\sigma\kk[N]$. Since $J$ is $G$-invariant, $\conj{n}{J}\subseteq J$ for all $n\in N$; hence the action and factor set descend to $A/J$, and $J^N=H_NJ$ by Corollary~\ref{cor:RelLift}(i). Therefore the lifted quotient $H_N/J^N$ is identified with the $N$-crossed product $(A/J)\#_{\bar\sigma}\kk[N]$, with induced action $\varphi_n=\overline{\conj{n}{(\cdot)}}$, factor set $\bar\sigma$ the image of $\sigma$, and neutral component $H_e/J\cong A/J$. Thus the lifted quotient of every such $(N,J)$ is a crossed product, and Theorem~\ref{thm:MainB} applies.

\begin{proposition}\label{prop:CrossedApp}
Let $H=A\#_\sigma\kk[G]$. The assignments
\[
I\longmapsto\left(N_I,\,I\cap H_e,\,\tau\right)
\qquad\text{and}\qquad
(N,J,\tau)\longmapsto\sum_{\bar g\in G/N}H_{\bar g}K_\tau,
\]
where in the first assignment $\tau\colon (A/J)\#_{\bar\sigma}\kk[N]\to A/J$ is the unique quasi-Hopf algebra retraction with kernel $(I\cap H_N)/J^N$ and in the second $K_\tau\subseteq H_N$ is the inverse image of $\ker(\tau)$ under the quotient map $H_N\to H_N/J^N$, define a bijective correspondence between the set of quasi-Hopf ideals of $H$ and the set of split admissible data $(N,J,\tau)$. Under the bijection $r_n=\tau(u_n)$ of Proposition~\ref{prop:CrossedSplitData} the ideal is generated by $J^G$ together with the elements $(1\#n)-\widetilde r_n\#e$, for $n\in N$ and arbitrary lifts $\widetilde r_n\in A$ of $r_n$.
\end{proposition}

\begin{proof}
This is Theorem~\ref{thm:MainB} and Proposition~\ref{prop:CrossedSplitData} applied to $H_N/J^N\cong(A/J)\#_{\bar\sigma}\kk[N]$.
\end{proof}

When the lifted action is trivial, the split data reduce to a single cochain trivializing the cohomological obstruction (cf.~\cite[\S4]{Ce02}). An ideal $I$ with $N_I=N$ and $I\cap H_e=J$ determines the cochain $\gamma(n)=\tau(u_n)$ for its retraction $\tau$, and conversely a map $\gamma\colon N\to\Gc{A/J}$ satisfying the conditions below determines the ideal $I_\gamma=\left\langle\,J^G,\ (1\#n)-\widetilde\gamma(n)\#e\ \mid\ n\in N\,\right\rangle_H$, with $\widetilde\gamma(n)\in A$ an arbitrary lift of $\gamma(n)$.

\begin{corollary}\label{cor:CrossedApp}
Let $H=A\#_\sigma\kk[G]$, let $N\unlhd G$ and let $J\subseteq A$ be a $G$-invariant quasi-Hopf ideal such that $N$ acts trivially on $A/J$ and $\bar\sigma(n,m)\in\Gc{A/J}$ for all $n,m\in N$. The assignments
\[
I\longmapsto\gamma
\qquad\text{and}\qquad
\gamma\longmapsto I_\gamma
\]
define a bijective correspondence between the set of quasi-Hopf ideals $I$ of $H$ with $N_I=N$ and $I\cap H_e=J$ and the set of maps $\gamma\colon N\to\Gc{A/J}$ such that:
\begin{enumerate}
    \item $\gamma(e)=1$ and $\bar\sigma(n,m)=\gamma(n)\gamma(m)\gamma(nm)^{-1}$ for all $n,m\in N$;

    \item the twisted units $v_n:=\gamma(n)^{-1}(1\#n)$ satisfy $\conj{g}{v_n}=v_{gng^{-1}}$ for all $g\in G$ and $n\in N$.
\end{enumerate}
\end{corollary}

\begin{proof}
The units $u_n=1\#n$ are group-like of degree $n$ and, because the action is trivial, centralize $A/J$. Hence $v_n:=\gamma(n)^{-1}u_n$ is a group-like unit of degree $n$ centralizing $A/J$, and $v_nv_m=v_{nm}$ holds exactly when $\gamma$ satisfies~(i). Conversely, the degree-$n$ component is $(A/J)u_n$; hence any group-like splitting is $v_n=c_nu_n$ with $\gamma(n):=u_nv_n^{-1}\in\Gc{A/J}$ satisfying~(i). With~(ii) this identifies such $\gamma$ with the equivariant families for $(N,J)$. Theorem~\ref{thm:Trivialization} attaches to $(v_n)$ the retraction $\tau$ with $\tau(u_n)=\gamma(n)$, whose kernel is $G/N$-invariant by Lemma~\ref{lem:ConjInv}(ii) and equals $\sum_{n\in N}(A/J)\left(u_n-\gamma(n)\right)$ by Lemma~\ref{lem:ConjInv}(i); hence Theorem~\ref{thm:MainB} returns $I_\gamma$.
\end{proof}

For the untwisted product $H=A\otimes\kk[G]$ one has $\bar\sigma=1$; hence condition~(i) says that $\gamma$ is a group homomorphism, and the units $u_n=1\#n$ are group-like, multiplicative and centralize $A/J$, with $\conj{g}{u_n}=u_{gng^{-1}}$. Every pair $(N,J)$ is therefore realized, and the ideals over it are classified by the $G$-equivariant group homomorphisms $\gamma\colon N\to\Gc{A/J}$.

\begin{example}\label{ex:CrossedZSq}
Let $C=\langle z\rangle\cong\mathbb{Z}/m\mathbb{Z}$ with $m\geq2$, and set $\sigma(a,b)=z^{a_2b_1}$ for $a=(a_1,a_2)$ and $b=(b_1,b_2)$ in $\mathbb{Z}^2$. Then $H=\kk[C]\#_\sigma\kk[\mathbb{Z}^2]$ has trivial action and a factor set with central group-like values. For $N=G$ and $J=0$, the alternation of $\sigma$ is $z^{a_2b_1-b_2a_1}$; hence $[\sigma]\neq1$ in $H^2(\mathbb{Z}^2,C)$. Equivalently, $u_{(0,1)}u_{(1,0)}=z\,u_{(1,0)}u_{(0,1)}$. Since $\kk[C]$ is commutative and the action is trivial, replacing $u_a$ by $c_au_a$ with $c_a\in\kk[C]^\times$ leaves that relation unchanged. No family of homogeneous units can therefore be multiplicative, as $\mathbb{Z}^2$ is abelian. By Corollary~\ref{cor:LiftTensor}, no quasi-Hopf ideal $I$ of $H$ satisfies $N_I=\mathbb{Z}^2$ and $I\cap H_e=0$. The algebra $H$ is the twisted group algebra of $\mathbb{Z}^2$ over $\kk[C]$ for the cocycle $\sigma$ (cf.~\cite[\S7.1]{Mon93}); its homogeneous units generate a central extension of discrete Heisenberg type classified by $[\sigma]$.
\end{example}

\subsection{A strongly graded Hopf algebra that is not a crossed product}\label{subsec:SL2}

Strong grading does not imply the existence of homogeneous units, even for Hopf algebras.

\begin{example}\label{ex:SL2}
Let $\kk$ be algebraically closed of characteristic different from $2$ and $H=\mathcal{O}(\mathrm{SL}_2)=\kk[a,b,c,d]/(ad-bc-1)$. Restriction to the central subgroup $\mu_2$ gives a cocentral map $H\to\mathcal{O}(\mu_2)$ and hence a $\mathbb{Z}/2\mathbb{Z}$-grading in which $a,b,c,d$ are odd and $H_0=\mathcal{O}(\mathrm{PSL}_2)$. The relation $ad-bc=1$ gives $1\in H_1H_1$, whence $H_0=H_0\cdot1\subseteq H_0H_1H_1\subseteq H_1H_1$ and the grading is strong.

On the other hand, $\mathcal{O}(\mathrm{SL}_2)^\times=\kk^\times$. By~\cite[Thm.~3]{Ros61}, normalized invertible regular functions on a connected algebraic group are characters; see also~\cite[Prop.~1.2]{KKV89} for this formulation. Since $\mathrm{SL}_2$ has no nontrivial characters and scalars lie in $H_0$, the component $H_1$ has no units. Thus $H$ is not a crossed product.
\end{example}

In particular, Theorem~\ref{thm:Char} requires a pair realized by some quasi-Hopf ideal, and $\mathcal{O}(\mathrm{SL}_2)$ shows that a faithfully graded quasi-Hopf algebra need not supply one. Homogeneous units would not by themselves suffice, since the theorem needs a group-like splitting. With the grading of Example~\ref{ex:SL2}, the algebra $\mathcal{O}(\mathrm{SL}_2)$ cannot occur as the lifted quotient associated with a quasi-Hopf ideal, as the next result shows.

\begin{proposition}\label{prop:SL2NotQuot}
Let $\kk$ be as in Example~\ref{ex:SL2}, let $\Gamma$ be a group, let $H'$ be a faithfully $\Gamma$-graded quasi-Hopf algebra, and let $I$ be a quasi-Hopf ideal of $H'$, with $N=N_I$ and $J=I\cap H'_e$. Suppose that a group isomorphism $\varphi\colon N\to\mathbb{Z}/2\mathbb{Z}$ is fixed. Then $H'_N/J^N$ is not isomorphic, as an $N$-graded quasi-Hopf algebra, to $\mathcal{O}(\mathrm{SL}_2)$ endowed with the transported $N$-grading
\[
\mathcal{O}(\mathrm{SL}_2)_n:=\mathcal{O}(\mathrm{SL}_2)_{\varphi(n)}.
\]
Moreover, with respect to the grading of Example~\ref{ex:SL2}, $\mathcal{O}(\mathrm{SL}_2)$ has no augmented Hopf ideal meeting $\mathcal{O}(\mathrm{PSL}_2)$ trivially.
\end{proposition}

\begin{proof}
By Corollary~\ref{cor:LiftTensor}, such a lifted quotient has a unit in every homogeneous component. A graded isomorphism as above would therefore carry a unit in the component indexed by $\varphi^{-1}(1)$ to a unit in $\mathcal{O}(\mathrm{SL}_2)_1$, contradicting Example~\ref{ex:SL2}. The second assertion follows from the same fact and the equivalence of conditions~(ii) and~(iii) in Theorem~\ref{thm:Trivialization}.
\end{proof}

\subsection{Cocentral abelian extensions with infinite grading}\label{subsec:AbelExt}

We now consider the cocentral abelian extensions of~\cite{Mas02} with finite neutral component and arbitrary grading group. Masuoka constructs and classifies these extensions, and observes that the grading group may be taken infinite, see~\cite[Rem.~2.14]{Mas02}. Quotients of these extensions have been considered for finite $G$ in~\cite{MN14}, where the surjections onto the extension associated with a central quotient of the grading group are characterized cohomologically, which in the present terms is the case of a quasi-Hopf ideal $I$ with $I\cap H_e=0$ and central associated subgroup. What is new here is the classification of all quasi-Hopf ideals, for an arbitrary grading group and an arbitrary normal subgroup.

Let $G$ be an arbitrary group (possibly infinite) and let $F$ be a finite group equipped with a left action of $G$ by group automorphisms. This induces a left action of $G$ on the dual group algebra $(\kk F)^*$ by Hopf algebra automorphisms, via $(\conj{g}{p_a})(b)=p_a(g^{-1}\cdot b)$, with $\conj{g}{p_a}=p_{\conj{g}{a}}$, where $\{p_a\}_{a\in F}$ is the dual basis of $\kk F$. Since $F$ is finite, $(\kk F)^*$ carries the standard dual Hopf algebra structure.

Let $\sigma\colon G\times G\to((\kk F)^*)^\times$ and $\theta\colon G\to((\kk F)^*\otimes(\kk F)^*)^\times$ be maps, with components $\sigma(g,h)=\sum_{a\in F}\sigma_a(g,h)\,p_a$ and $\theta(g)=\sum_{b,c}\theta_{b,c}(g)\,p_b\otimes p_c$ in the dual basis. We require that they be normalized, $\sigma_a(e,h)=\sigma_a(g,e)=\sigma_e(g,h)=1$ and $\theta_{e,c}(g)=\theta_{b,e}(g)=\theta_{b,c}(e)=1$, that $\sigma$ be a $2$-cocycle twisted by the action of $G$,
\[
\sigma_a(g,h)\,\sigma_a(gh,k)=\sigma_{\conj{g^{-1}}{a}}(h,k)\,\sigma_a(g,hk),
\]
that $\theta$ be a dual $2$-cocycle,
\[
\theta_{x,y}(g)\,\theta_{xy,z}(g)=\theta_{y,z}(g)\,\theta_{x,yz}(g),
\]
and that the two be compatible,
\[
\sigma_{xy}(g,h)\,\theta_{x,y}(gh)=\sigma_x(g,h)\,\sigma_y(g,h)\,\theta_{x,y}(g)\,\theta_{\conj{g^{-1}}{x},\,\conj{g^{-1}}{y}}(h),
\]
for all $a,x,y,z\in F$ and $g,h,k\in G$. The first expresses associativity of the multiplication, the second coassociativity of the comultiplication before insertion of the neutral reassociator below, and the third that $\Delta$ is an algebra morphism; they are the normalized-cocycle and compatibility conditions of~\cite[Lem.~1.2, Prop.~1.8]{Mas02}. We define the algebra $H=(\kk F)^*\#_\sigma^\theta\kk[G]$ as the vector space $(\kk F)^*\otimes\kk[G]$, with basis $p_a u_g$ for $a\in F$, $g\in G$. For later use, write
\[
u_g:=\sum_{a\in F}p_a u_g=1\otimes\zeta_g.
\]

The multiplication is the crossed product twisted by $\sigma$,
\[
(p_a u_g)(p_b u_h)=\delta_{a,\conj{g}{b}}\,\sigma_a(g,h)\,p_a u_{gh},
\]
and the comultiplication is twisted by $\theta$,
\[
\Delta(p_a u_g)=\sum_{bc=a}\theta_{b,c}(g)\,p_b u_g\otimes p_c u_g.
\]
The counit is $\varepsilon(p_a u_g)=\delta_{a,e}$. Put $\bar a_g:=\conj{g^{-1}}{a^{-1}}$. The antipode is
\[
S(p_a u_g)=\theta_{a,a^{-1}}(g)^{-1}\,\sigma_{\bar a_g}(g^{-1},g)^{-1}\,p_{\bar a_g}\,u_{g^{-1}}.
\]
When the action of $G$ on $F$ is trivial, the dual cocycle identity gives $\theta_{a,a^{-1}}(g)=\theta_{a^{-1},a}(g)$, and the cocycle identity for $\sigma$ gives $\sigma_{a^{-1}}(g^{-1},g)=\sigma_{a^{-1}}(g,g^{-1})$. Thus the scalar is equivalently $\theta_{a^{-1},a}(g)^{-1}\sigma_{a^{-1}}(g,g^{-1})^{-1}$, the form appearing in~\cite[Lem.~2.10]{Mas02}. Following the Drinfeld associator of~\cite[Lem.~4.1]{Mas02}, choose a normalized, $G$-invariant $3$-cocycle $\omega\colon F\times F\times F\to\kk^\times$ on the neutral component, that is, $\omega(\conj{g}{a},\conj{g}{b},\conj{g}{c})=\omega(a,b,c)$ for all $g\in G$; this invariance is automatic when $G$ acts trivially on $F$. Set
\[
\Phi=\sum_{a,b,c\in F}\omega(a,b,c)\,p_a u_e\otimes p_b u_e\otimes p_c u_e,
\qquad
\alpha=\sum_{a\in F}\omega(a,a^{-1},a)\,p_{a^{-1}}u_e,
\qquad
\beta=1,
\]
with class $[\omega]\in H^3(F,\kk^\times)$. With these displayed formulas, $\Phi$ is strictly trivial precisely when $\omega=1$; if $[\omega]=1$, a neutral gauge transform identifies the reassociator with the trivial one.

\begin{lemma}\label{lem:AbelExtQH}
Under the normalized cocycle, compatibility and $G$-invariance assumptions above, the displayed formulas define a quasi-Hopf algebra structure on $H=(\kk F)^*\#_\sigma^\theta\kk[G]$.
\end{lemma}

\begin{proof}
All verifications are coefficientwise in the basis $\{p_a u_g\}$. The proof does not require $G$ to be finite. The associativity of the multiplication is exactly the twisted $2$-cocycle identity for $\sigma$, and the unit follows from normalization. The counit is multiplicative because $\sigma_e(g,h)=1$, while the two counit identities for the coproduct follow from the normalization of $\theta$. For brevity, put
\[
C_a(g):=\theta_{a,a^{-1}}(g)^{-1}\sigma_{\bar a_g}(g^{-1},g)^{-1},
\]
so that $S(p_a u_g)=C_a(g)p_{\bar a_g}u_{g^{-1}}$. To verify that $S$ is an algebra anti-morphism, consider two basis elements $p_au_g$ and $p_bu_h$. Both sides of
\[
S\left((p_au_g)(p_bu_h)\right)=S(p_bu_h)S(p_au_g)
\]
vanish unless $b=\conj{g^{-1}}{a}$. In the remaining case, $\bar b_h=\bar a_{gh}$, and the required equality reduces to
\[
\sigma_a(g,h)C_a(gh)
=C_b(h)C_a(g)\sigma_{\bar a_{gh}}(h^{-1},g^{-1}).
\]
This follows from the compatibility identity with $(x,y)=(a,a^{-1})$ and from the twisted cocycle identity for $\sigma$, applied successively to $(h^{-1},g^{-1},gh)$ and $(g^{-1},g,h)$. Thus $S$ is an algebra anti-morphism. The calculation does not require $G$ to be finite and does not involve the reassociator.

The condition that $\Delta$ be an algebra morphism is precisely the displayed compatibility between $\sigma$ and $\theta$. For quasi-coassociativity, compare the coefficient of $p_xu_g\otimes p_yu_g\otimes p_zu_g$, with $xyz=a$, in the two sides of~\eqref{eq:QH2}. The two coefficients are
\[
\omega(x,y,z)\,\theta_{x,y}(g)\theta_{xy,z}(g)
\quad\text{and}\quad
\omega(\conj{g^{-1}}{x},\conj{g^{-1}}{y},\conj{g^{-1}}{z})\,\theta_{y,z}(g)\theta_{x,yz}(g),
\]
which are equal by $G$-invariance of $\omega$ and the dual $2$-cocycle identity for $\theta$. The pentagon for $\Phi$ is the usual group-cohomology $3$-cocycle identity for $\omega$, and the normalization of $\omega$ gives~\eqref{eq:QH3} and~\eqref{eq:QH5}.

The identities~\eqref{eq:QH7} are checked in the neutral component. Reindexing the display above, $\alpha=\sum_{a\in F}\alpha_a p_au_e$ with $\alpha_a=\omega(a^{-1},a,a^{-1})$. The coefficient of $p_au_e$ in $\Phi^{(1)}\beta S(\Phi^{(2)})\alpha\Phi^{(3)}$ is $\omega(a,a^{-1},a)\,\omega(a^{-1},a,a^{-1})=1$ by the $3$-cocycle identity for $\omega$ applied to $(a,a^{-1},a,a^{-1})$, and that in $S(\Phi^{(-1)})\alpha\Phi^{(-2)}\beta S(\Phi^{(-3)})$ is $\omega(a^{-1},a,a^{-1})^{-1}\alpha_a=1$.

It remains to check the quasi-antipode identities~\eqref{eq:QH6}. In $S(x')\alpha x''$ for $x=p_a u_g$, the summand of $\Delta(p_a u_g)=\sum_{bc=a}\theta_{b,c}(g)p_bu_g\otimes p_cu_g$ survives only when $c=b^{-1}$; hence the expression vanishes unless $a=e$. In the surviving case its scalar part is $\theta_{b,b^{-1}}(g)\,C_b(g)\,\sigma_{\bar b_g}(g^{-1},g)=1$, and the coefficient of $\alpha$ is carried from $p_{b^{-1}}$ to $p_{\bar b_g}$, exactly its coefficient there by $G$-invariance of $\omega$. This proves the first identity in~\eqref{eq:QH6}. For the second, the surviving terms again have $a=e$, now with $c=b^{-1}$, and their scalar part is $\theta_{b,b^{-1}}(g)\,C_{b^{-1}}(g)\,\sigma_b(g,g^{-1})=1$, since $\theta_{b,b^{-1}}(g)=\theta_{b^{-1},b}(g)$ by the dual cocycle identity with $(x,y,z)=(b,b^{-1},b)$, and $\sigma_b(g,g^{-1})=\sigma_{\conj{g^{-1}}{b}}(g^{-1},g)$ by the twisted cocycle identity for $\sigma$ with $(h,k)=(g^{-1},g)$. Thus $(S,\alpha,\beta)$ is a quasi-antipode.
\end{proof}

The homogeneous components $H_g=(\kk F)^*u_g$ define a strong $G$-grading, and $\Delta(H_g)\subseteq H_g\otimes H_g$ as required by Definition~\ref{def:Grad}. When $G$ is infinite, this gives an infinite grading with finite-dimensional neutral component $H_e=(\kk F)^*$. Such an $H$ is infinite-dimensional and need not be semisimple. For $G=\mathbb{Z}$ with trivial action and $\sigma=\theta=1$, for instance,
\[
H\cong(\kk F)^*\otimes\kk[t,t^{-1}]
\cong\prod_{a\in F}\kk[t,t^{-1}],
\]
which is infinite-dimensional and not semisimple. The classification therefore reaches examples that are neither finite-dimensional nor semisimple.

The quasi-Hopf ideals of $H_e$ are the ideals $J_M$ of functions vanishing on a subgroup $M\le F$, with $H_e/J_M\cong(\kk M)^*$ and restricted associator $\omega|_{M^3}$. Such an ideal is $G$-invariant exactly when $M$ is. Consequently, Theorem~\ref{thm:MainB} identifies all quasi-Hopf ideals with triples $(N,M,\tau)$, where $N\unlhd G$, $M\le F$ is $G$-invariant, and $\tau\colon H_N/J_M^N\to(\kk M)^*$ is a retraction with $G/N$-invariant kernel. Theorem~\ref{thm:MainA} classifies all ideals by admissible pairs.

A nontrivial action of $N$ on $M$ already rules the pair out.

\begin{proposition}\label{prop:AbelExtAction}
Let $N\unlhd G$ and let $M\le F$ be $G$-invariant. If some element of $N$ acts nontrivially on $M$, then no quasi-Hopf ideal $I\subseteq H$ satisfies $N_I=N$ and $I\cap H_e=J_M$.
\end{proposition}

\begin{proof}
The multiplication rule gives $u_np_b=p_{\conj{n}{b}}u_n$, and $M$ is stable under $N$. Hence the images $\bar u_n$ and $p_b$, for $b\in M$, satisfy the same relation in $H_N/J_M^N$. By Corollary~\ref{cor:RelLift}(ii), that lifted quotient is strongly $N$-graded with neutral component $(\kk M)^*$. Since $\bar u_n$ is a unit of degree $n$, its component of degree $n$ is $(\kk M)^*\bar u_n$. Let $v$ be a unit of that component. Then $c:=v\bar u_n^{-1}$ and $\bar u_nv^{-1}$ both have degree $e$, hence lie in $(\kk M)^*$, and their product is $1$; therefore $c$ is invertible there. For $b\in M$, the algebra $(\kk M)^*$ being commutative, one has $vp_b=c\,p_{\conj{n}{b}}\bar u_n$ and $p_bv=c\,p_b\bar u_n$; hence $v$ centralizes $(\kk M)^*$ if and only if $\conj{n}{b}=b$ for all $b\in M$. If $n$ acts nontrivially on $M$, the lifted quotient therefore admits no group-like splitting, and Corollary~\ref{cor:LiftTensor} gives the assertion.
\end{proof}

Only the action of $N$ on $M$ is at issue. The grading group may act nontrivially on $F$, its elements outside $N$ entering through the $G$-invariance of $M$ and through the equivariance of Definition~\ref{def:EquivFamily}.

Suppose that $N$ acts trivially on $M$. Then $H_N/J_M^N\cong(\kk M)^*\#_{\bar\sigma}\kk[N]$ has trivial action and its homogeneous units centralize the neutral component. The two remaining requirements of Theorem~\ref{thm:Trivialization}(iii) constrain one and the same rescaling, by an arbitrary invertible $c_n=\sum_{a\in M}c_n(a)p_a$ and not only by one in $\Gc{(\kk M)^*}$. The element $c_nu_n$ is group-like exactly when $\theta_{a,b}(n)=c_n(a)c_n(b)c_n(ab)^{-1}$ for all $a,b\in M$, which determines $c_n$ up to a character, and $(c_nu_n)_{n\in N}$ is then multiplicative exactly when $\bar\sigma(n,m)=c_{nm}c_n^{-1}c_m^{-1}$. If both hold and the resulting family is $G$-conjugation-equivariant, Proposition~\ref{prop:Realized} applies and Theorem~\ref{thm:Char} classifies the ideals over $(N,J_M)$ by the $G$-equivariant group homomorphisms $N\to\Gc{(\kk M)^*}$.

\begin{lemma}\label{lem:GcDual}
For a finite group $F$, the invertible central group-like elements of the dual $(\kk F)^*$ are the characters of $F$:
\[
\Gc{(\kk F)^*}=\Hom(F,\kk^\times)\cong\Hom(F_{\mathrm{ab}},\kk^\times),
\]
the characters of the abelianization $F_{\mathrm{ab}}=F/[F,F]$.
\end{lemma}

\begin{proof}
Identifying $(\kk F)^*$ with functions on $F$, the identity $\Delta(\phi)=\phi\otimes\phi$ is equivalent to $\phi(ab)=\phi(a)\phi(b)$ and $\phi(e)=1$. Thus $\phi$ takes values in $\kk^\times$ and is a character. Centrality is automatic because $(\kk F)^*$ is commutative, and every character factors through $F_{\mathrm{ab}}$.
\end{proof}

In particular, $\Gc{(\kk M)^*}\cong\Hom(M_{\mathrm{ab}},\kk^\times)$. For $F$ and $G$ finite, this identifies the coefficient group of Proposition~\ref{prop:CohObstr} with the one carrying the categorical obstruction of~\cite[\S5.2.1]{GJ25}.

\begin{example}
Over $\kk=\mathbb{C}$, take $F=\mathbb{Z}/2\mathbb{Z}$, written additively, $G=\mathbb{Z}$, $\sigma=\theta=1$, and let $\omega$ represent the nontrivial class in $H^3(F,\mathbb{C}^\times)$, with $\omega(1,1,1)=-1$, the twisting of a finite group by a $3$-cocycle in this manner going back to~\cite{DPR91}. Writing $\{p_0,p_1\}$ for the dual basis, we obtain
\[
\Phi=1\otimes1\otimes1-2\,p_1u_e\otimes p_1u_e\otimes p_1u_e,
\qquad
\alpha=p_0u_e-p_1u_e,
\qquad
\beta=1,
\]
with $S(p_au_g)=p_au_{g^{-1}}$. Here $\Gc{(\kk F)^*}=\mu_2$, and the units $u_g$ are group-like, multiplicative and central, and therefore form an equivariant family. Accordingly, the character data over $(\mathbb{Z},J_F)$, where $J_F=0$, are the two elements of $\Hom(\mathbb{Z},\mu_2)$. Since $[\omega]\neq1$, the reassociator cannot be removed by a neutral gauge transformation.
\end{example}

\begin{example}\label{ex:SplitNotChar}
Let $F=G=N=\mathbb{Z}/2\mathbb{Z}$, with all three groups written additively and $G$ acting trivially on $F$, and let $n$ denote their nontrivial element. Suppose that $\kk$ contains a primitive fourth root of unity $i$. Take $\omega=1$, write $\{p_0,p_1\}$ for the dual basis of $(\kk F)^*$ and $q=p_0-p_1$, and set $\sigma(n,n)=q$ and
\[
\theta(n)=p_0\otimes p_0+p_0\otimes p_1+p_1\otimes p_0-p_1\otimes p_1,
\]
with the remaining values normalized. These data define a commutative and cocommutative Hopf algebra $H=(\kk F)^*\#_\sigma^\theta\kk[G]$ of dimension $4$. Thus $H$ is an $N$-crossed product with trivial action and a factor set with central group-like values, and hypotheses~\condref{cond:CP},~\condref{cond:TrivAct} and~\condref{cond:CentralCocy} all hold for the units $u_n$. Here $\Gc{H_e}=\{1,q\}$, while $u_n^2=q$ and $\Delta(u_n)=\theta(n)(u_n\otimes u_n)$; therefore $u_n$ is not group-like. Since no element of $\Gc{H_e}$ squares to $q$, $[\sigma]\neq1$, and the chosen homogeneous section admits no $\Gc{H_e}$-valued trivializing cochain.

Set $r=p_0+i\,p_1$. Then $r^2=q$, $\Delta(r)=\theta(n)(r\otimes r)$, and $r$ is central and invertible but not group-like. Hence $\tau|_{H_e}=\id$ and $\tau(bu_n)=br$ define a Hopf retraction. Consequently, $I=\langle u_n-r\rangle$ is split. For $u_n'=r^{-1}u_n$ one has $(u_n')^2=1$, $\Delta(u_n')=u_n'\otimes u_n'$, and $\tau(u_n')=1$. Here $\theta(n)$ is the coboundary of $r$ and the rescaled factor set is trivial; hence both conditions above are met by a rescaling outside $\Gc{H_e}$ and by none inside it. Thus the same ideal is $I=\langle u_n'-1\rangle$, and $\{u_n'\}$ is an equivariant family while $\{u_n\}$ is not. The original section cannot serve as reference units, whereas relative to the rescaled equivariant family the homomorphism associated with $I$ is trivial, illustrating Remark~\ref{rem:UnitDep}.
\end{example}

\section*{Acknowledgements}
F.C. was partially supported by the Vicerrector\'{\i}a de Investigaci\'{o}n y Extensi\'{o}n of the Universidad Industrial de Santander under project code 4463. C.G. was partially supported by Grant INV-2025-213-3452 from the Facultad de Ciencias of Universidad de los Andes. During the preparation of the latest version of this paper, we used GPT-5.6 and Claude Opus 5.
\bibliographystyle{alphaurl}
\bibliography{biblio}

@article {BC03,
    AUTHOR = {Bulacu, Daniel and Caenepeel, Stefaan},
     TITLE = {Integrals for (dual) quasi-{H}opf algebras. {A}pplications},
   JOURNAL = {J. Algebra},
  FJOURNAL = {Journal of Algebra},
    VOLUME = {266},
      YEAR = {2003},
    NUMBER = {2},
     PAGES = {552--583},
      ISSN = {0021-8693,1090-266X},
   MRCLASS = {16W30},
  MRNUMBER = {1995128},
MRREVIEWER = {Shuan\ Hong\ Wang},
       DOI = {10.1016/S0021-8693(03)00175-3},
}

@article {BCM86,
    AUTHOR = {Blattner, Robert J. and Cohen, Miriam and Montgomery, Susan},
     TITLE = {Crossed products and inner actions of {H}opf algebras},
   JOURNAL = {Trans. Amer. Math. Soc.},
  FJOURNAL = {Transactions of the American Mathematical Society},
    VOLUME = {298},
      YEAR = {1986},
    NUMBER = {2},
     PAGES = {671--711},
      ISSN = {0002-9947,1088-6850},
   MRCLASS = {16A24 (16A03 16A72 46L40)},
  MRNUMBER = {860387},
MRREVIEWER = {Donald\ S.\ Passman},
       DOI = {10.1090/S0002-9947-1986-0860387-X},
}

@book {BCPV19,
    AUTHOR = {Bulacu, Daniel and Caenepeel, Stefaan and Panaite, Florin and
              Van Oystaeyen, Freddy},
     TITLE = {Quasi-{H}opf algebras: A categorical approach},
    SERIES = {Encyclopedia of Mathematics and its Applications},
    VOLUME = {171},
 PUBLISHER = {Cambridge University Press, Cambridge},
      YEAR = {2019},
     PAGES = {xvi+528},
      ISBN = {978-1-108-42701-2},
   MRCLASS = {16T05 (18D10)},
  MRNUMBER = {3929714},
MRREVIEWER = {Sonia\ Natale},
       DOI = {10.1017/9781108582780},
}

@misc {Bic25,
    AUTHOR = {Bichon, Julien},
     TITLE = {Faithful flatness of {H}opf algebras over coideal subalgebras with a bimodule conditional expectation},
      YEAR = {2025},
      NOTE = {Preprint, arXiv:2301.05480v2},
       DOI = {10.48550/arXiv.2301.05480},
}

@article {BN03,
    AUTHOR = {Bulacu, Daniel and Nauwelaerts, Erna},
     TITLE = {Quasitriangular and ribbon quasi-{H}opf algebras},
   JOURNAL = {Comm. Algebra},
  FJOURNAL = {Communications in Algebra},
    VOLUME = {31},
      YEAR = {2003},
    NUMBER = {2},
     PAGES = {657--672},
      ISSN = {0092-7872,1532-4125},
   MRCLASS = {16W30},
  MRNUMBER = {1968919},
MRREVIEWER = {Shuan\ Hong\ Wang},
       DOI = {10.1081/AGB-120017337},
}

@article {BNY16,
    AUTHOR = {Bichon, Julien and Neshveyev, Sergey and Yamashita, Makoto},
     TITLE = {Graded twisting of categories and quantum groups by group
              actions},
   JOURNAL = {Ann. Inst. Fourier (Grenoble)},
  FJOURNAL = {Universit\'e{} de Grenoble. Annales de l'Institut Fourier},
    VOLUME = {66},
      YEAR = {2016},
    NUMBER = {6},
     PAGES = {2299--2338},
      ISSN = {0373-0956,1777-5310},
   MRCLASS = {16T05 (18D10 20G42 46L65)},
  MRNUMBER = {3580173},
MRREVIEWER = {Eliezer\ Batista},
       DOI = {10.5802/aif.3064},
}

@article {Ce02,
    AUTHOR = {Cegarra, Antonio M.},
     TITLE = {On the classification of strongly graded {H}opf algebras},
   JOURNAL = {J. Algebra},
  FJOURNAL = {Journal of Algebra},
    VOLUME = {251},
      YEAR = {2002},
    NUMBER = {1},
     PAGES = {358--370},
      ISSN = {0021-8693,1090-266X},
   MRCLASS = {16W30},
  MRNUMBER = {1900289},
MRREVIEWER = {Akira\ Masuoka},
       DOI = {10.1006/jabr.2001.9125},
}

@article {D89,
    AUTHOR = {Drinfel'd, Vladimir G.},
     TITLE = {Quasi-{H}opf algebras},
   JOURNAL = {Algebra i Analiz},
  FJOURNAL = {Algebra i Analiz},
    VOLUME = {1},
      YEAR = {1989},
    NUMBER = {6},
     PAGES = {114--148},
      ISSN = {0234-0852},
   MRCLASS = {17B37 (16W30 57M25 81T40)},
  MRNUMBER = {1047964},
MRREVIEWER = {Ya.\ S.\ So\u{\i}bel\cprime man},
      NOTE = {Translation in \emph{Leningrad Math. J.}, 1(6):1419--1457, 1990.},
       URL = {https://www.mathnet.ru/eng/aa53},
}

@incollection {DPR91,
    AUTHOR = {Dijkgraaf, Robbert and Pasquier, Vincent and Roche, Pierre},
     TITLE = {Quasi-quantum groups related to orbifold models},
 BOOKTITLE = {Modern quantum field theory ({B}ombay, 1990)},
     PAGES = {375--383},
 PUBLISHER = {World Sci. Publ., River Edge, NJ},
      YEAR = {1991},
      ISBN = {981-02-0199-0; 981-02-0200-8},
   MRCLASS = {81R50 (81T40)},
  MRNUMBER = {1237434},
}

@article {GJ25,
    AUTHOR = {Galindo, C\'esar and Jones, Corey},
     TITLE = {Equivariant fusion subcategories},
   JOURNAL = {Transform. Groups},
  FJOURNAL = {Transformation Groups},
    VOLUME = {30},
      YEAR = {2025},
    NUMBER = {1},
     PAGES = {267--282},
      ISSN = {1083-4362,1531-586X},
   MRCLASS = {18M20 (16T05)},
  MRNUMBER = {4863963},
MRREVIEWER = {Thibault\ D.\ D\'ecoppet},
       DOI = {10.1007/s00031-023-09838-9},
}

@article {GM12,
    AUTHOR = {Galindo, C\'{e}sar and Mombelli, Mart\'{\i}n},
     TITLE = {Module categories over finite pointed tensor categories},
   JOURNAL = {Selecta Math. (N.S.)},
  FJOURNAL = {Selecta Mathematica. New Series},
    VOLUME = {18},
      YEAR = {2012},
    NUMBER = {2},
     PAGES = {357--389},
      ISSN = {1022-1824,1420-9020},
   MRCLASS = {16D90 (16T05 18D10)},
  MRNUMBER = {2927237},
MRREVIEWER = {Alessandro\ Ardizzoni},
       DOI = {10.1007/s00029-011-0067-x},
}

@article {HN99,
    AUTHOR = {Hausser, Frank and Nill, Florian},
     TITLE = {Diagonal crossed products by duals of quasi-quantum groups},
   JOURNAL = {Rev. Math. Phys.},
  FJOURNAL = {Reviews in Mathematical Physics. A Journal for Both Review and
              Original Research Papers in the Field of Mathematical Physics},
    VOLUME = {11},
      YEAR = {1999},
    NUMBER = {5},
     PAGES = {553--629},
      ISSN = {0129-055X,1793-6659},
   MRCLASS = {81R50 (16S40 16W35 17B37)},
  MRNUMBER = {1696105},
MRREVIEWER = {Istv\'an\ Heckenberger},
       DOI = {10.1142/S0129055X99000210},
}

@incollection {KKV89,
    AUTHOR = {Knop, Friedrich and Kraft, Hanspeter and Vust, Thierry},
     TITLE = {The {P}icard group of a {$G$}-variety},
 BOOKTITLE = {Algebraische {T}ransformationsgruppen und
              {I}nvariantentheorie},
    SERIES = {DMV Sem.},
    VOLUME = {13},
     PAGES = {77--87},
 PUBLISHER = {Birkh\"auser, Basel},
      YEAR = {1989},
      ISBN = {3-7643-2284-5},
   MRCLASS = {14C22 (14D25 14L30)},
  MRNUMBER = {1044586},
       DOI = {10.1007/978-3-0348-7662-9_5},
}

@incollection {Mas02,
    AUTHOR = {Masuoka, Akira},
     TITLE = {Hopf algebra extensions and cohomology},
 BOOKTITLE = {New directions in {H}opf algebras},
    SERIES = {Math. Sci. Res. Inst. Publ.},
    VOLUME = {43},
     PAGES = {167--209},
 PUBLISHER = {Cambridge Univ. Press, Cambridge},
      YEAR = {2002},
      ISBN = {0-521-81512-6},
   MRCLASS = {16W30 (16E30)},
  MRNUMBER = {1913439},
MRREVIEWER = {Sonia\ Natale},
       DOI = {10.1017/9781009701396.005},
}

@incollection {MN14,
    AUTHOR = {Mason, Geoffrey and Ng, Siu-Hung},
     TITLE = {Cleft extensions and quotients of twisted quantum doubles},
 BOOKTITLE = {Developments and retrospectives in {L}ie theory},
    SERIES = {Dev. Math.},
    VOLUME = {38},
     PAGES = {229--246},
 PUBLISHER = {Springer, Cham},
      YEAR = {2014},
      ISBN = {978-3-319-09803-6; 978-3-319-09804-3},
   MRCLASS = {16T10 (18D10)},
  MRNUMBER = {3308786},
MRREVIEWER = {Daniel\ Ion\ Bulacu},
       DOI = {10.1007/978-3-319-09804-3\_11},
}

@book {Mon93,
    AUTHOR = {Montgomery, Susan},
     TITLE = {Hopf algebras and their actions on rings},
    SERIES = {CBMS Regional Conference Series in Mathematics},
    VOLUME = {82},
 PUBLISHER = {Conference Board of the Mathematical Sciences, Washington, DC;
              by the American Mathematical Society, Providence, RI},
      YEAR = {1993},
     PAGES = {xiv+238},
      ISBN = {0-8218-0738-2},
   MRCLASS = {16W30},
  MRNUMBER = {1243637},
MRREVIEWER = {E.\ J.\ Taft},
       DOI = {10.1090/cbms/082},
}

@incollection {Nik19,
    AUTHOR = {Nikshych, Dmitri},
     TITLE = {Classifying braidings on fusion categories},
 BOOKTITLE = {Tensor categories and {H}opf algebras},
    SERIES = {Contemp. Math.},
    VOLUME = {728},
     PAGES = {155--167},
 PUBLISHER = {Amer. Math. Soc., [Providence], RI},
      YEAR = {2019},
      ISBN = {978-1-4704-4321-4},
   MRCLASS = {18M15},
  MRNUMBER = {3943750},
MRREVIEWER = {Costel\ Gabriel\ Bontea},
       DOI = {10.1090/conm/728/14660},
}

@book {NVO04,
    AUTHOR = {N\u{a}st\u{a}sescu, Constantin and Van Oystaeyen, Freddy},
     TITLE = {Methods of graded rings},
    SERIES = {Lecture Notes in Mathematics},
    VOLUME = {1836},
 PUBLISHER = {Springer-Verlag, Berlin},
      YEAR = {2004},
     PAGES = {xiv+304},
      ISBN = {3-540-20746-5},
   MRCLASS = {16W50 (16-02)},
  MRNUMBER = {2046303},
MRREVIEWER = {Blas\ Torrecillas Jover},
       DOI = {10.1007/b94904},
}

@book {Pas79,
    AUTHOR = {Passi, Inder Bir S.},
     TITLE = {Group rings and their augmentation ideals},
    SERIES = {Lecture Notes in Mathematics},
    VOLUME = {715},
 PUBLISHER = {Springer, Berlin},
      YEAR = {1979},
     PAGES = {vi+137},
      ISBN = {3-540-09254-4},
   MRCLASS = {20C07},
  MRNUMBER = {537126},
MRREVIEWER = {E.\ Formanek},
       DOI = {10.1007/BFb0067186},
}

@article {Ros61,
    AUTHOR = {Rosenlicht, Maxwell},
     TITLE = {Toroidal algebraic groups},
   JOURNAL = {Proc. Amer. Math. Soc.},
  FJOURNAL = {Proceedings of the American Mathematical Society},
    VOLUME = {12},
      YEAR = {1961},
     PAGES = {984--988},
      ISSN = {0002-9939,1088-6826},
   MRCLASS = {14.50},
  MRNUMBER = {133328},
MRREVIEWER = {E.\ R.\ Kolchin},
       DOI = {10.1090/S0002-9939-1961-0133328-9},
}

@article {Sch04,
    AUTHOR = {Schauenburg, Peter},
     TITLE = {A quasi-{H}opf algebra freeness theorem},
   JOURNAL = {Proc. Amer. Math. Soc.},
  FJOURNAL = {Proceedings of the American Mathematical Society},
    VOLUME = {132},
      YEAR = {2004},
    NUMBER = {4},
     PAGES = {965--972},
      ISSN = {0002-9939,1088-6826},
   MRCLASS = {16W30},
  MRNUMBER = {2045410},
MRREVIEWER = {Salih\ \c Celik},
       DOI = {10.1090/S0002-9939-03-07181-8},
}

@incollection {Sch05,
    AUTHOR = {Schauenburg, Peter},
     TITLE = {Quotients of finite quasi-{H}opf algebras},
 BOOKTITLE = {Hopf algebras in noncommutative geometry and physics},
    SERIES = {Lecture Notes in Pure and Appl. Math.},
    VOLUME = {239},
     PAGES = {281--290},
 PUBLISHER = {Dekker, New York},
      YEAR = {2005},
      ISBN = {0-8247-5759-9},
   MRCLASS = {16W30},
  MRNUMBER = {2106936},
MRREVIEWER = {Fang\ Li},
}

@article {Sch90,
    AUTHOR = {Schneider, Hans-J\"urgen},
     TITLE = {Principal homogeneous spaces for arbitrary {H}opf algebras},
   JOURNAL = {Israel J. Math.},
  FJOURNAL = {Israel Journal of Mathematics},
    VOLUME = {72},
      YEAR = {1990},
    NUMBER = {1-2},
     PAGES = {167--195},
      ISSN = {0021-2172},
   MRCLASS = {16W30},
  MRNUMBER = {1098988},
MRREVIEWER = {Mitsuhiro\ Takeuchi},
       DOI = {10.1007/BF02764619},
}

@article {Sch92,
    AUTHOR = {Schneider, Hans-J\"urgen},
     TITLE = {Normal basis and transitivity of crossed products for {H}opf
              algebras},
   JOURNAL = {J. Algebra},
  FJOURNAL = {Journal of Algebra},
    VOLUME = {152},
      YEAR = {1992},
    NUMBER = {2},
     PAGES = {289--312},
      ISSN = {0021-8693,1090-266X},
   MRCLASS = {16W30},
  MRNUMBER = {1194305},
MRREVIEWER = {E.\ J.\ Taft},
       DOI = {10.1016/0021-8693(92)90034-J},
}

@misc {Skr25,
    AUTHOR = {Skryabin, Serge},
     TITLE = {On {T}akeuchi's correspondence},
      YEAR = {2025},
      NOTE = {Preprint, arXiv:2501.06045},
       DOI = {10.48550/arXiv.2501.06045},
}

@misc {Skr25b,
    AUTHOR = {Skryabin, Serge},
     TITLE = {Failure of flatness over finite-dimensional {H}opf subalgebras},
      YEAR = {2025},
      NOTE = {Preprint, arXiv:2506.16292},
       DOI = {10.48550/arXiv.2506.16292},
}
\end{document}